\documentclass[a4paper,11pt,english]{amsart}

\usepackage[T1]{fontenc}
\usepackage[utf8]{inputenc}
\usepackage{babel}
\usepackage{microtype}

\usepackage{tikz-cd} %authomatically loads tikz
\usetikzlibrary{arrows.meta}
\usetikzlibrary{calc}
\usetikzlibrary{decorations.markings}
\usetikzlibrary{intersections}

\usepackage{nicematrix}
\usepackage{mathtools}
\usepackage{amssymb}

\usepackage{graphicx}

\usepackage[italic=true]{derivative}

\usepackage{amsthm}

\usepackage{enumitem}
\setlist[enumerate,1]{label = {\arabic*)}}

\usepackage{color}

\usepackage{subcaption}

\usepackage[inkscapearea=page]{svg}

\usepackage[autostyle]{csquotes}
\usepackage[backend=biber, style=alphabetic, hyperref, giveninits=true]{biblatex}
\DeclareFieldFormat{url}{\url{#1}}

\usepackage{refcount}
\setrefcountdefault{0}
\usepackage{xstring}

\usepackage{hyperref}
\hypersetup{hidelinks}
\usepackage{cleveref}

\theoremstyle{plain}
\newtheorem{prop}{Proposition}[section]
\newtheorem{cor}[prop]{Corollary}
\newtheorem{lemma}[prop]{Lemma}
\newtheorem{thm}[prop]{Theorem}

\newcounter{secR}
\newtheorem{thmR}{Theorem}

\theoremstyle{definition}
\newtheorem{defin}[prop]{Definition}
\newtheorem{rmk}[prop]{Remark}

\counterwithin{figure}{section}
\counterwithin{equation}{section}

\crefname{prop}{Proposition}{Propositions}
\crefname{cor}{Corollary}{Corollaries}
\crefname{lemma}{Lemma}{Lemmas}
\crefname{thm}{Theorem}{Theorems}
\crefname{thmR}{Theorem}{Theorems}
\crefname{defin}{Definition}{Definitions}
\crefname{rmk}{Remark}{Remarks}
\crefname{definsn}{Definition}{Definitions}
\crefname{Oss}{Observation}{Observations}
\crefname{equation}{Equation}{Equations}
\crefname{figure}{Figure}{Figures}

\Crefname{prop}{Proposition}{Propositions}
\Crefname{cor}{Corollary}{Corollaries}
\Crefname{lemma}{Lemma}{Lemmas}
\Crefname{thm}{Theorem}{Theorems}
\Crefname{thmR}{Theorem}{Theorems}
\Crefname{defin}{Definition}{Definitions}
\Crefname{rmk}{Remark}{Remarks}
\Crefname{definsn}{Definition}{Definitions}
\Crefname{Oss}{Observation}{Observations}
\Crefname{equation}{Equation}{Equations}
\Crefname{figure}{Figure}{Figures}

\renewcommand{\epsilon}{\varepsilon}
\renewcommand{\theta}{\vartheta}

\newcommand{\zero}{\mathbf{0}}

\newcommand{\N}{\mathbb{N}}
\newcommand{\R}{\mathbb{R}}
\newcommand{\Z}{\mathbb{Z}}

\newcommand{\gk}{\mathfrak{g}}
\newcommand{\XX}{\mathfrak{X}}

\newcommand{\FF}{\mathcal{F}}
\newcommand{\KK}{\mathcal{K}}
\newcommand{\LL}{\mathcal{L}}
\newcommand{\Lie}{\mathcal{L}}
\newcommand{\NN}{\mathcal{N}}
\newcommand{\UU}{\mathcal{U}}
\newcommand{\ZZ}{\mathcal{Z}}

\newcommand{\fix}{\mathrm{fix}}
\newcommand{\eq}{\mathrm{eq}}
\newcommand{\std}{\mathrm{std}}
\newcommand{\h}{\mathrm{h}}
\newcommand{\wh}{\mathrm{wh}}

\newcommand{\id}{\mathrm{id}}
\newcommand{\ant}{\mathrm{ant}}
\newcommand{\swap}{\mathrm{swap}}
\newcommand{\half}{\mathrm{half}}

\newcommand{\bimath}{\bar{\imath}}

\newcommand{\eset}{\varnothing}

\DeclareMathOperator{\SO}{SO}
\DeclareMathOperator{\OOO}{O}
\DeclareMathOperator{\GL}{GL}

\DeclareMathOperator{\Fix}{Fix}
\DeclareMathOperator{\interior}{int}
\DeclareMathOperator{\Ker}{ker}

\DeclareMathOperator{\Diff}{Diff}
\DeclareMathOperator{\tb}{tb}
\DeclareMathOperator{\diver}{div}
\DeclareMathOperator{\sign}{sgn}
\DeclareMathOperator{\Tight}{Tight}
\DeclareMathOperator{\Sing}{Sing}
\DeclareMathOperator{\Emb}{Emb}

\DeclareMathOperator{\Loop}{Loop}
\DeclareMathOperator{\proj}{proj}
\DeclarePairedDelimiter{\abs}{\lvert}{\rvert}
\DeclarePairedDelimiter{\ang}{\langle}{\rangle}
\DeclarePairedDelimiter{\Abs}{\lVert}{\rVert}

\begin{document}
\title[Anti-invariant tight contact structures]{Classification of anti-invariant tight contact structures on various 3-manifolds}
\author{Mirko Torresani}
\address{DIMAI, viale Giovan Battista Morgagni 67/a, 50134, Firenze, Italy}
\begin{abstract}
	We establish classification results for tight contact structures on the manifolds $\R^3$, $D^3_\std$, $S^3$, and $S^2 \times S^1$, that are anti-invariant with respect to standard involutions. In contrast to the non-equivariant setting, a new integral rotation emerges, equipping the set of equivalence classes with the structure of a $\Z$-torsor. These results could be useful for future classification results regarding strongly invertible Legendrian unknots.
\end{abstract}
\maketitle
\tableofcontents
\section{Introduction}\label{sec:introdution}

The aim of this paper is to establish classification results for tight contact structures on $\R^3$, $S^3$ and on the standard closed 3-ball $D^3_{\std}$, that are anti-invariant with respect to a $\Z/2\Z$-action defined below.

Recall that a contact structure $\xi$ on an oriented three-manifold $M$ is a plane distribution satisfying the following condition. For any point $p \in M$, any neighborhood $U \ni p $, and any local 1-form $\alpha \in \Lambda^1U$ whose kernel is $\xi|_U$,
\begin{equation}\label{eq:contcondalphadalpha}
	(\alpha \wedge d\alpha)_q \neq 0 \ \forall q \in U
\end{equation}
holds.

The structure $\xi$ is \emph{cooriented} if it coincides with $\Ker(\alpha)$ for some global 1-form $\alpha$. If $M$ is \emph{cohomologically 1-connected}, meaning that $H^1(M;\Z/2\Z) = 0$, then all contact structures are coorientable. Note that if $M$ is oriented, then there is a natural bijection between an orientation and a coorientation of $\xi$. Moreover, the orientation sign of the volume form $\alpha \wedge d\alpha$ in \cref{eq:contcondalphadalpha} is invariant under rescaling of $\alpha$. If this sign agrees with the orientation of $M$, then $\xi$ is called positive.

As a consequence of \cref{eq:contcondalphadalpha}, $d\alpha|_{\xi}$ is a non-degenerate 2-form on $\xi$. A choice of a global 1-form then induces an orientation for the plane distribution $\xi$. We will suppose all the contact structures to be oriented. A \emph{positive contact form} is one inducing on $\xi$ the chosen orientation.

A contact structure $\xi$ is \emph{overtwisted} if there exists an embedded 2-disk $D \hookrightarrow M$ whose boundary is tangent to $\xi$, and such that $D$ is transverse to $\xi$ along $\partial D$. A contact structure that is not overtwisted is called \emph{tight}.

Furthermore, let $S$ be an oriented surface embedded in $M$, equipped with a tubular neighborhood diffeomorphic to $S \times \R$ ($S$ is identified with $S \times \{0\}$). In such coordinates, any positive 1-form $\alpha$ can be written as $\beta_z + w_z\,dz$, where $w_z \colon S \to \R$ is a smooth function and $\beta_z$ is a smooth family of 1-forms on $S$. The contact condition $\alpha \wedge d\alpha > 0$ can be rewritten as
\begin{equation}\label{eq:contcond}
	w_z \,d\beta_z +\beta_z \wedge(dw_z - \dot{\beta}_z) > 0\,,
\end{equation}
where the dot denotes the derivative with respect to $z$.

Moreover, if we equip $S$ with an area form $\Omega$, we can define the \emph{characteristic foliation} $S_\xi$ to be the oriented singular 1-foliation generated by the unique vector field $X$ satisfying $\iota_X\Omega = \beta_0$, where $\iota_X$ is the contraction given by 
\[
	\iota_X\Omega = \Omega(X, \cdot)\,.
\]
The oriented foliation generated by $X$ is independent of the choice of the positive contact form $\alpha$ and positive area form $\Omega$. One can show that $X$ vanishes only at the points $p$ where the planes $\xi_p$ and $T_pS$ coincide. We give signs to such singularities: \emph{positive} if the orientation on $T_pS$ coincides with the one on $\xi_p$ induced by $d\alpha_p|_{\xi_p}$, \emph{negative} otherwise.

Finally, given a contact form $\alpha$, the \emph{Reeb vector field} of $\alpha$ is the unique vector field $R_\alpha \in \mathfrak{X}(M)$ satisfying
\[
\begin{split}
	d\alpha(R_\alpha,\cdot) &\equiv 0 \\
	\alpha(R_\alpha) &\equiv 1.
\end{split}
\]
The Reeb vector field defines a direct sum decomposition 
\[
	TM = \ang{R_\alpha} \oplus \xi\,.
\]

Given a 3-manifold $M$, a \emph{rotation-like} smooth involution $u \colon M \to M$ is an orientation-preserving involution of $M$ having codimension-2 fixed-point set. In particular, at every fixed point $p \in M$ the differential $d_pu \colon T_pM \to T_pM$ is conjugate to the diagonal matrix $[1] \oplus [-1] \oplus [-1]$. 

We say that a contact structure $\xi$ is \emph{anti-invariant} if $du(\xi) = -\xi$ as oriented plane fields. Analogously, an anti-invariant $k$-form ($k \ge 0$) is any $k$-form $\omega$ satisfying $u^*\omega = -\omega$. An anti-invariant vector field $X$ is any vector field such that $u^*X = -X$, where $u^*X(p)$ is the preimage of $X(u(p))$ via $d_pu$.

The aim of this paper is to classify anti-invariant contact structures on $M =\R^3, S^3$ up to \emph{equivariant isotopy}. More precisely, two anti-invariant contact structures $\xi_1$, $\xi_2$ are equivalent if there exists an isotopy $\psi_t \colon M \to M$ such that
\begin{enumerate}
	\item $\psi_0 = \id_M$;
	\item $\psi_t \circ u = u \circ \psi_t$;
	\item $d\psi_1(\xi_1) = \xi_2$.
\end{enumerate}
Unless stated otherwise, condition 1) will be assumed for all isotopies.

Let $u_\std \colon \R^3 \to \R^3$ be the involution defined as
\[
u_{\std}(x,y,z) \coloneqq (-x,-y,z)\,.
\]
With a slight abuse of notation, we also will denote with $u_{\std}$ the unique infinity-preserving smooth involution of $S^3 = \R^3 \cup \{\infty\}$ extending the involution above. Analogously, if $M$ equals $S^2 \times [-1,1]$ or $S^2 \times S^1$, set $u_\times$ as the product involution $u_\std|_{S^2} \times \id$.

In this paper we shall prove the following results. Let $B^3_\std$ denote the (open) standard 3-ball, and let $D^3_\std$ denote the (closed) standard 3-disk.

\StrBefore{\getrefnumber{mainthmRthree}}{.}[\refsecnumber]
\IfEq{\refsecnumber}{}{\def\refsecnumber{0}}{}
\setcounter{secR}{\refsecnumber}
\StrBehind{\getrefnumber{mainthmRthree}}{.}[\refcounternumber]
\IfEq{\refcounternumber}{}{\def\refcounternumber{0}}{}
\setcounter{thmR}{\refcounternumber}
\addtocounter{thmR}{-1}
\begin{thmR}
	The pair $(\R^3, u_{\std})$ admits a unique anti-invariant tight contact structure up to equivariant isotopy.
\end{thmR}
\StrBefore{\getrefnumber{mainthmDthreecomplementary}}{.}[\refsecnumber]
\IfEq{\refsecnumber}{}{\def\refsecnumber{0}}{}
\setcounter{secR}{\refsecnumber}
\StrBehind{\getrefnumber{mainthmDthreecomplementary}}{.}[\refcounternumber]
\IfEq{\refcounternumber}{}{\def\refcounternumber{0}}{}
\setcounter{thmR}{\refcounternumber}
\addtocounter{thmR}{-1}
\begin{thmR}
	Two anti-invariant tight contact structures on $(\R^3 \setminus B^3_{\std}, u_{\std})$ inducing the same characteristic foliation on $\partial D^3_{\std}$ are equivariantly isotopic relative to $\partial D^3_\std$.
\end{thmR}
\StrBefore{\getrefnumber{mainthmSthree}}{.}[\refsecnumber]
\IfEq{\refsecnumber}{}{\def\refsecnumber{0}}{}
\setcounter{secR}{\refsecnumber}
\StrBehind{\getrefnumber{mainthmSthree}}{.}[\refcounternumber]
\IfEq{\refcounternumber}{}{\def\refcounternumber{0}}{}
\setcounter{thmR}{\refcounternumber}
\addtocounter{thmR}{-1}
\begin{thmR}
	Up to equivariant isotopy, the set of anti-invariant tight contact structures on $(S^3,u_{\std})$ is non-empty and a $\Z$-torsor.
\end{thmR}
\StrBefore{\getrefnumber{mainthmS2timesS1}}{.}[\refsecnumber]
\IfEq{\refsecnumber}{}{\def\refsecnumber{0}}{}
\setcounter{secR}{\refsecnumber}
\StrBehind{\getrefnumber{mainthmS2timesS1}}{.}[\refcounternumber]
\IfEq{\refcounternumber}{}{\def\refcounternumber{0}}{}
\setcounter{thmR}{\refcounternumber}
\addtocounter{thmR}{-1}
\begin{thmR}
	Up to equivariant isotopy, the set of anti-invariant tight contact structures on $(S^2 \times S^1, u_\times)$ is non-empty and a $\Z$-torsor.
\end{thmR}
\StrBefore{\getrefnumber{mainthmDthree}}{.}[\refsecnumber]
\IfEq{\refsecnumber}{}{\def\refsecnumber{0}}{}
\setcounter{secR}{\refsecnumber}
\StrBehind{\getrefnumber{mainthmDthree}}{.}[\refcounternumber]
\IfEq{\refcounternumber}{}{\def\refcounternumber{0}}{}
\setcounter{thmR}{\refcounternumber}
\addtocounter{thmR}{-1}
\begin{thmR}
	Let $X \in \XX(\partial D^3_\std)$ be an anti-invariant vector field which satisfies the divergence condition \eqref{eq:divcond}. Up to equivariant isotopy relative to $\partial D^3_\std$, the set of anti-invariant tight contact structures on $(D^3_{\std}, u_{\std})$ inducing the characteristic foliation $\ang{X}$, if non-empty\footnote{see \cref{XiXnonempty}}, is a $\Z$-torsor.
\end{thmR}
They will be consequences of the following result. Set $M \coloneqq S^2 \times [-1,1]$.
\StrBefore{\getrefnumber{mainthmtomographyStwo}}{.}[\refsecnumber]
\IfEq{\refsecnumber}{}{\def\refsecnumber{0}}{}
\setcounter{secR}{\refsecnumber}
\StrBehind{\getrefnumber{mainthmtomographyStwo}}{.}[\refcounternumber]
\IfEq{\refcounternumber}{}{\def\refcounternumber{0}}{}
\setcounter{thmR}{\refcounternumber}
\addtocounter{thmR}{-1}
\begin{thmR}
	Let $X\in \XX(\partial M)$ be an anti-invariant vector field. Up to equivariant isotopy relative to $\partial M$, the set of anti-invariant tight contact structures on $(M,u_\times)$ inducing characteristic foliation $\ang{X}$, if non-empty, is a $\Z$-torsor.
\end{thmR}

All spaces of maps between manifolds will be understood to be equipped with the weak $C^\infty$-Whitney topology, see \cite[\S 2.4]{benedettiLecturesDifferentialTopology2021}.

Finally, a \emph{region} $R$ in an $n$-manifold $M$ is defined as a closed codimension-0 submanifold of $M$ that is also a manifold with boundary.

The paper is organized as follows. In \cref{sec:EThm} we state and prove equivariant versions of classical theorems in contact topology; \cref{sec:EConvex} is dedicated to expanding those results in order to lay the foundation for a basic equivariant convex theory. Although the equivariant setting is new, much of the groundwork consists in verifying that the $\Z/2\Z$-action preserves the relevant geometric structures. The remaining preparatory results are established in \cref{sec:eqSmale}, where we prove an equivariant analogue of Smale’s theorem, and in \cref{sec:tomography}, where we prove \cref{mainthmtomographyStwo} and observe the emergence of a $\Z$-torsor structure that has no classical counterpart. Finally, \cref{sec:proofmainthm} is devoted to the proofs of \cref{mainthmRthree,mainthmDthreecomplementary,mainthmSthree,mainthmS2timesS1,mainthmDthree}.
\section{Equivariant Darboux and Giroux's Theorems}\label{sec:EThm}
The goal of this section is to state and prove equivariant analogues of several classical results in contact topology, specifically \cref{CorEGray}, \cref{equalneighorigin}, and \cref{Girouxbound}.

Let us recall useful facts that will be used throughout this and the following sections. They can be found in \cite[\S VI.2]{bredonIntroductionCompactTransformation1972}.
\begin{prop}
	Let $G$ be a finite group, and let $P$ be a $G$-manifold. Then $P$ admits an \emph{invariant metric} $\gk$, i.e.\ one for which $G$ acts by isometries.
\end{prop}
\begin{prop}\label{tubneigh}
	Let $P$ be a (not necessarily connected) manifold equipped with an orientation-preserving involution $\rho$, and let $W\subseteq P$ be an invariant closed codimension-$k$ submanifold, not necessarily connected. Then $W$ admits an invariant (open) tubular neighborhood, unique up to equivariant isotopy.
	
	If $k = 1$, and $V \in TP|_W$ is an (anti-)invariant vector field transverse to $TW$ then the exponential map 
	\[
	\begin{split}
		W \times [-\epsilon,\epsilon] &\to P \\
		(p,t) &\mapsto \exp_\gk(tV(p))
	\end{split}
	\]
	is well-defined for sufficiently small $\epsilon>0$, and conjugates $\rho$ to the action
	\[
	(p,t) \mapsto (\rho \cdot p, \nu\, t)
	\]
	where
	\[
	\begin{cases}
		\nu = 1 & V \text{ is invariant}\\
		\nu = -1 & V \text{ is anti-invariant}
	\end{cases}\,.
	\]
	
	When $k = \dim P$, this recovers the Bochner linearization theorem: any point $p \in \Fix(\rho)$ admits a local coordinate chart in which the action $\rho$ is linear and coincides with the differential $d_p\rho$ on $T_pP$.
\end{prop}
Following the previous result, we can define $V$ as the unique $\gk$-unit vector field orthogonal to $TW$, such that $(T_p W, V(p))$ is a positively oriented pair for all $p \in W$.

Let $M$ be a 3-manifold, possibly with boundary, and let $u$ be a rotation-like smooth involution on $M$. We fix a $u$-invariant Riemannian metric $\mathfrak{g}$ on $M$. Unless specified otherwise, all contact structures and forms considered hereafter are defined on $M$.
\begin{rmk}\label{rmkpositiveantiinvariantformnormalized}
	Given an anti-invariant contact structure $\xi$, the conditions
	\[
	\begin{cases}
		\Abs{\alpha^{*_\gk}}_\gk \equiv 1, \text{ where $*_\gk$ denotes the $\gk$-dual}\\
		\Ker(\alpha) = \xi\\
		\alpha \text{ is positive}
	\end{cases}
	\]
	uniquely determine a \emph{normalized positive contact form}. Since $\gk$ is an invariant metric, $\alpha$ inherits the anti-invariance property of the contact structure.
\end{rmk}
\begin{defin}
	A family $(\xi_t)_{t \in [0,1]}$ of plane distributions is smooth if there exists a smooth family of 1-forms $\beta_t$ such that $\xi_t = \Ker(\beta_t)$. 
\end{defin}
\begin{rmk}\label{rmksmoothfamilyantiinvariant}
	Let $(\xi_t)_{t \in [0,1]}$ be a smooth family of anti-invariant contact structures induced by a smooth family $\beta_t$ of contact 1-forms. The conditions
	\[
	\begin{cases}
		\Ker(\alpha_t) = \xi_t,\\
		\Abs{\alpha_t^{*_\gk}}_\gk = 1,\\
		\alpha_t = f_t\,\beta_t \text{ for some smooth function $f_t > 0$},
	\end{cases}
	\]
	uniquely determine a smooth family $\alpha_t$ of normalized anti-invariant contact forms inducing $\xi_t$.
\end{rmk}

\begin{prop}\label{EGray}
	Let $\xi_t$, $t \in [0,1]$, be a smooth family of anti-invariant contact structures and $\alpha_t$ an associated smooth family of anti-invariant contact forms. Equation 
	\begin{equation}\label{eq:Xt}
		\dot{\alpha}_t + \iota_{X_t}d\alpha_t + d(\alpha_t(X_t)) = \mu_t\,\alpha_t
	\end{equation}
	can be uniquely solved for an invariant vector field $X_t\in\xi_t$ and an invariant smooth function $\mu_t$. Moreover, if $\dot{\alpha}_t(p)(v) = 0$ for some $p \in M$, $t \in [0,1]$, and $v \in \xi_t(p)$, then $X_t(p)$ and $v$ are parallel. In particular, if $\dot{\alpha}_t(p)$ vanishes, so does $X_t(p)$.
\end{prop}
\begin{proof}
	Since $X_t$ belongs to the plane distribution $\xi_t = \Ker(\alpha_t)$, the term $d(\alpha_t(X_t))$ vanishes.
	
	By evaluating \cref{eq:Xt} on the Reeb vector field $R_t$ of $\alpha_t$ we obtain the following expression for $\mu_t$:
	\begin{equation}\label{eq:mut}
		\dot{\alpha}_t(R_t) = \mu_t\,.
	\end{equation}
	
	Once $\mu_t$ has been determined by \cref{eq:mut}, \cref{eq:Xt} uniquely determines a vector field $X_t \in \xi_t$. Indeed, let us consider the direct sum decomposition
	\[
		TM = \ang{R_t} \oplus \xi_t\,.
	\]
	Evaluating \eqref{eq:Xt} on $\xi_t$, the non-degeneracy of $d\alpha_t|_{\xi_t}$ implies that there exists a unique $X_t \in \xi_t$ solving the equation on this summand. Evaluating \eqref{eq:Xt} on $R_t$, the equation reduces to \eqref{eq:mut}, which has already fixed $\mu_t$ and is independent of $X_t$. Therefore $X_t$ is uniquely determined. 
	
	Moreover, if $\dot{\alpha}_t(p)(v) = 0$ for some $p \in M$, $t \in [0,1]$, and $v \in \xi(p)$, then
	\[
		0 = (\iota_{X_t}d\alpha_t)(p)(v) = d\alpha_t(p)(X_t(p),v)\,,
	\]
	implying that $X_t$ and $v$ are parallel.
	
	Since $\alpha_t$ is anti-invariant, so must be the corresponding Reeb vector field. Thus,
	\[
		\mu_t \circ u = u^*\mu_t = u^*[\dot{\alpha}_t(R_t)] = - \dot{\alpha}_t(- R_t) = \mu_t\,.
	\] 
	
	Hence $u^*X_t$ satisfies the following equality:
	\[
		u^*\dot{\alpha}_t + \iota_{u^*X_t}du^*\alpha_t = \mu_t\, u^*\alpha_t\,,
	\]
	which can be rewritten as
	\[
		-\dot{\alpha}_t - \iota_{u^*X_t} d\alpha_t = -\mu_t\alpha_t\,.
	\]
	Consequently $u^*X_t$ satisfies the same equation as $X_t$, and by uniqueness they must coincide. 
\end{proof}
\begin{rmk}\label{rmkXtofpsitsolving}
	Let $\xi_t$ be a family of contact structures induced by a smooth family of 1-forms $\alpha_t$. If $\psi_t$ is an isotopy such that $d\psi_t(\xi_0) = \xi_t$, the generating vector field $X_t \coloneqq \tfrac{d}{dt}\psi_t \circ \psi_t^{-1}$ satisfies \cref{eq:Xt} (cf.\ \cite[Thm.\ 2.2.2]{geigesIntroductionContactTopology2008}).
\end{rmk}

As an application of \cref{EGray}, we obtain an equivariant version of Gray's theorem, see \cite[Thm.\ 2.2.2]{geigesIntroductionContactTopology2008}.
\begin{cor}\label{CorEGray}
	Let $M$ be closed, and let $(\xi_t)_{t \in [0,1]}$ be a smooth family of anti-invariant contact structures. There exists an equivariant isotopy $\psi_t$ such that $d\psi_t(\xi_0) = \xi_t$
\end{cor}
\begin{proof}
	By \cref{rmksmoothfamilyantiinvariant} we fix a smooth family of anti-invariant 1-forms $\alpha_t$ defining $\xi_t$. According to \cref{rmkXtofpsitsolving}, \cref{eq:Xt} yields a unique invariant vector field $X_t$.
	
	Because $M$ is compact, $X_t$ can be integrated to generate a global isotopy $\psi_t$. Since $X_t$ is invariant, its integral flow $\psi_t$ is equivariant at all times, and by \cref{rmkXtofpsitsolving}, $\psi_t$ satisfies $d\psi_t(\xi_0) = \xi_t$.
\end{proof}

\Cref{EGrayglobal} provides a mild extension of \cref{CorEGray}, which will be crucial for proving \cref{equalneighorigin}.
\begin{lemma}\label{extisoeq}
	Let $W \subseteq M$ be a closed submanifold, and let $\phi_t \colon B' \to M$ be a smooth family of equivariant embeddings of an open invariant neighborhood $B'$ of $W$. Moreover, fix a relatively compact open set $B$ containing the closure of $\bigcup_t\phi_t(B')$. There exists a global equivariant isotopy $\psi_t \colon M \to M$ such that 
	\[
	\begin{cases}
		\psi_t|_{B'} = \phi_t|_{B'}\\
		\psi_t|_{B^c} = \id_{B^c}
	\end{cases}\forall t \in [0,1]\,.
	\]
\end{lemma}
\begin{proof}
	The map $\phi_t$ can be extended to a (not necessarily equivariant) global isotopy $\Phi_t$ which is the identity outside $B$, cf.\ \cite[Prop.\ 7.2]{benedettiLecturesDifferentialTopology2021}. According to \cite[Thm.\ VI.3.1]{bredonIntroductionCompactTransformation1972} there exists an equivariant isotopy $\psi_t$ which agrees with $\Phi_t$ when restricted to $B'$ and $B^c$.
\end{proof}
\begin{prop}\label{EGrayglobal}
	Let $(\alpha_t)_{t \in [0,1]}$ be a smooth family of anti-invariant 1-forms, and let $(\xi_t)_{t \in [0,1]}$ be the induced family of plane distributions. Suppose that the family $(\alpha_t)_{t \in [0,1]}$ satisfies the contact condition on an invariant neighborhood $B$ of a point $p \in M$. If $\dot{\alpha}_t(p) = 0$, then there exists an open invariant neighborhood $B' \subseteq B$ surrounding $p$ and an equivariant global isotopy $\psi_t$ such that
	\[
	\begin{cases}
		\psi_t(B') \subseteq B\\
		d\psi_t(\xi_0)|_ {B'}= \xi_t|_{B'}\\
		\psi_t|_{B^c} = \id_{B^c}
	\end{cases}\forall t \in [0,1]\,.
	\]
\end{prop}
\begin{proof}
	Following \cref{EGray}, the restriction of \cref{eq:Xt} on $B$ is satisfied by a unique 1-parameter invariant vector field $X_t \in \mathfrak{X}(B)$. By the same proposition, $X_t(p)$ vanishes for every $t \in [0,1]$.
	
	In particular, the restriction $X_t|_{\{p\}}$ can then be integrated up to time $t=1$. The same property holds for $X_t$ restricted to a sufficiently small invariant open neighborhood $B' \subseteq B$ of $p$, yielding an isotopy
	\[
		\phi_t \colon B' \to M\,.
	\] 
	After possibly shrinking $B'$, we may assume that $\phi_t(B')$ remains compactly contained in $B$ for all $t \in [0,1]$.
	
	By \cref{extisoeq} we can extend $\phi_t$ to a global equivariant isotopy $\psi_t$ that acts as the identity outside $B$. Furthermore, since $X_t|_{B'}$ satisfies \cref{eq:Xt}, \cref{rmkXtofpsitsolving} implies that the plane distribution $d\psi_t(\xi_0)|_{B'}$ coincides with $\xi_t|_{B'}$.
\end{proof}

\begin{prop}\label{equalneighorigin}
	Let $\xi_1$, $\xi_0$ be anti-invariant contact structures on $M$, and let $p \in \interior(M)$ be a fixed point of $u$. There exists a compactly supported equivariant isotopy $\psi_t$ such that $d\psi_1(\xi_0)$ coincides with $\xi_1$ in a small neighborhood of $p$.
\end{prop}
\begin{proof}
	According to \cref{tubneigh}, $p$ admits a coordinate chart $\phi \colon U \to \R^3$ such that $\phi(p) = \zero$ and
	\[
		(\phi \circ u \circ \phi^{-1})(x,y,z) = (-x,-y,z)\,.
	\]	
	Thus, we can restrict ourselves to the case $M = \R^3$,  $p = \zero$ and 
	\[
		\xi_1 = \Ker(dx - z\,dy)\,.
	\]
	
	Let $\gk$ be the standard Euclidean metric on $\R^3$, and let $\alpha$ be the normalized positive anti-invariant contact form which defines $\xi_0$. The dual vector $\alpha(\zero)^{*_\gk}$ belongs to the $(-1)$-eigenspace of $d_\zero u$, which coincides with the $xy$-plane. Thus, a rotational isotopy around the $z$-axis aligns $\xi_0(\zero)$ with the $yz$-plane, and $\alpha(\zero)$ with $dx$. Applying \cref{extisoeq} with $W = \{p\}$, we may assume after a \emph{compactly supported} equivariant isotopy that $\xi_0(\zero)$ coincides with the $yz$-plane, and that $\alpha(\zero) = dx$.
		
	Restricting $d\alpha(\zero)$ to $\xi_0(\zero)$ yields a positive symplectic 2-form. Hence $d\alpha(\zero)|_{\xi(\zero)}$ coincides with $\lambda\,dy \wedge dz$ for some positive $\lambda$. In particular, after applying to $\xi_0$ a cutoff of the isotopy
	\[
	\begin{bmatrix}
		x \\
		y \\
		z
	\end{bmatrix}
	\mapsto (1-t)
	\begin{bmatrix}
		x\\[0.5ex]
		y\\[0.5ex]
		z
	\end{bmatrix}
	+ t
	\begin{bmatrix}
		x\  \\[0.5ex]
		y \\[0.5ex]
		\lambda\,z
	\end{bmatrix}\,,
	\]
	we may assume that $d\alpha(\zero)|_{\xi_0(\zero)}$ coincides with $dy \wedge dz$.
	
	Next, we adjust the 2-form $d\alpha$ on the whole tangent space $T_\zero\R^3$. Since $d\alpha $ is anti-invariant under $u_*$, the $dx \wedge dy$ component must vanish at the origin. Consequently, $d\alpha(\zero)$ takes the form
	\[
		dy\wedge dz + \mu\, dx \wedge dz\,, \quad \mu \in \R\,.
	\]
	We then consider the modified contact form
	\[
	\alpha_0 \coloneqq e^{\mu z}\alpha\,,
	\]
	which defines the same structure $\xi_0$. Its exterior derivative satisfies
	\[
		d\alpha_0 = \mu e^{\mu z}\,dz \wedge \alpha + e^{\mu z}\,d\alpha\,.
	\]
	In particular,
	\[
	\begin{cases}
		d\alpha_0(\zero) = dy \wedge dz\\
		\alpha_0(\zero) = dx
	\end{cases}\,.
	\]\vspace{2pt}
	
	Finally, set $\alpha_1 \coloneqq dx - z\,dy$ and define the family of anti-invariant forms
	\[
	\alpha'_t \coloneqq (1-t)\alpha_0 + t\alpha_1\,.
	\]
	Since $\alpha'_t(\zero) = dx$, $d\alpha'_t(\zero) = dy\wedge dz$ for all $t \in [0,1]$, the derivative $\dot{\alpha}_t'(\zero)$ vanishes.
	
	By \cref{EGrayglobal} there exists a compactly supported equivariant isotopy $\psi_t$ such that $d\psi_1(\xi_0)$ coincides with $\Ker(\alpha_1)$ in a small neighborhood of the origin.
\end{proof}

We now move forward to the proof of \cref{Girouxbound}. We first introduce the relevant classes of surfaces and vector fields.
\begin{defin}
	An invariant surface $S \subseteq M$ is $u$-oriented if it is oriented and $u|_S$ is orientation-preserving. 
\end{defin}
\begin{rmk}\label{rmkactiononalphaloccoord}
	Following \cref{tubneigh}, let $\NN(S) \simeq S \times \R$ be an invariant tubular neighborhood of $S$. If $\beta_z + w_z\,dz$ is a contact form on such neighborhood, the condition $u^*\alpha = -\alpha$ is equivalent to both $(u|_S)^*\beta_z = -\beta_z$, and $w_z \circ u|_S = -w_z$ holding.
\end{rmk}

If $S$ is $u$-oriented and $p \in S$ is fixed by $u$, the differential $d_pu$ restricts to $T_pS$ as $-\id_{T_pS}$. Indeed, since $u$ is rotation-like, the $1$-eigenspace of $d_pu$ is $1$-dimensional. Thus, the endomorphism $d_pu|_{T_pS}$ is an  orientation-preserving linear involution of $T_pS$ having the $1$-eigenspace of dimension at most $1$. The only such map is $-\id_{T_pS}$.

\begin{defin}
	Let $S \subseteq M$ be a $u$-oriented surface. An anti-invariant singular foliation on $S$ is an element of the following quotient:
	\begin{equation}\label{singfol}
		\{X \in \XX(S) \mid u^*X = -X\}/\sim\,,
	\end{equation}
	where
	\begin{enumerate}
		\item $u^*X \coloneqq du^{-1} \circ X \circ u$, and
		\item $X \sim X'$ if there exists a smooth function $f \colon S \to (0,\infty)$ such that $f \circ u = f$ and $X' = f\,X$.
	\end{enumerate}
\end{defin}
By equipping a $u$-oriented surface $S$ with the volume form $\Omega_\gk$ induced by $\gk|_S$, the characteristic foliation $S_\xi$ is an element of \eqref{singfol} for any anti-invariant contact structure $\xi$. Indeed, let $\alpha$ be the normalized positive contact form. If $X$ generates $S_\xi$, then 
\[
	\iota_{u^*X}\Omega_\gk = u^*(\iota_X\Omega_\gk) = u^*(i_S^* \alpha) = i_S^*(u^*\alpha) = -i_S^*\alpha = -\iota_{X}\Omega_\gk = \iota_{-X}\Omega_\gk.
\]
Since $\Omega_\gk$ is non-degenerate on $S$, this immediately yields $u^*X = -X$.\vspace{0.1cm}

\Cref{equationXtnotonxit} and \cref{Girouxbound} together form an equivariant analogue of \cite[Thm.\ 2.5.23]{geigesIntroductionContactTopology2008}, which we split into two separate results for clarity.
\begin{prop}\label{equationXtnotonxit}
	Let $S \subseteq M$ be a $u$-oriented compact surface, and let $(\alpha_t)_{t \in [0,1]}$ be a smooth family of anti-invariant contact forms defined on an invariant tubular neighborhood $\NN(S)$. Suppose that $i^*_S\dot{\alpha}_t \equiv 0$. Then \cref{eq:Xt}, restricted to $\NN(S)$, can be uniquely solved for an invariant smooth function $\mu_t$ and an invariant vector field $X_t$ that vanishes identically on $S$.
\end{prop}
\begin{proof}
	In the following proof all the objects are defined on $\NN(S)$. 
	
	Let $R_t$ be the Reeb vector field of $\alpha_t$, and decompose $X_t$ as $Y_t + H_t\,R_t$ for a smooth function $H_t \colon \NN(S) \to \R$ and a vector field $Y_t \in \xi_t$. Substituting such this decomposition of $X_t$ into \cref{eq:Xt} yields the condition
	\begin{equation}\label{eq:Yt}
		\dot{\alpha}_t + \iota_{Y_t}d\alpha_t + dH_t = \mu_t \alpha_t\,.
	\end{equation}
	
	Since $i^*_S\dot{\alpha}_t \equiv 0$ holds, the system of equations
	\begin{equation}\label{eq:system}
		\begin{cases}
			dH_t = -\dot{\alpha}_t\\
			H_t \equiv 0
		\end{cases} \text{along } S\,,
	\end{equation} 
	admits an anti-invariant smooth solution $H_t \colon \NN(S) \to \R$. Indeed, choose equivariant coordinates $\NN(S) \simeq S \times (-\epsilon, \epsilon)$, and set $H_t(x,z) \coloneqq -z\,\dot{w}_{t,0}(x)$, where $\alpha_t = \beta_{t,z} + w_{t,z}\,dz$.
	
	Evaluating \cref{eq:Yt} on the Reeb vector field $R_t$ determines the function $\mu_t$ as follows:
	\begin{equation}\label{eq:mutHt}
		\mu_t = \dot{\alpha}_t(R_t) + dH_t(R_t)\,.
	\end{equation}
	Moreover, \cref{eq:system} forces $\mu_t$ to vanish identically on $S$. Since $\dot{\alpha}_t$, $R_t$, $dH_t$ are all anti-invariant, we have that $\mu_t \circ u = \mu_t$, so $\mu_t$ is invariant.
	
	Under the direct sum decomposition $T\NN(S) = \xi_t \oplus \ang{R_t}$, the non-degeneracy of $d\alpha_t|_{\xi_t}$ implies that there exists a unique vector $Y_t \in \xi_t$ solving \cref{eq:Yt} on the summand $\xi_t$. Evaluating \cref{eq:Yt} on $R_t$ reduces it to \eqref{eq:mutHt}, which has already uniquely determined $\mu_t$ independently of $Y_t$. Therefore $Y_t$ is uniquely determined.
	
	Notice that \cref{eq:Yt} reduces to $\iota_{Y_t}d\alpha_t = 0$ when restricted to $S$. Thus, $Y_t|_S \equiv 0$, which yields $X_t|_S \equiv 0$ as desired. 
	
	Finally, by applying the pull-back operator $u^*$ to \cref{eq:Yt}, we obtain:
	\[
		-\dot{\alpha}_t -\iota_{u^*Y_t}d\alpha_t -dH_t = -\mu_t\alpha_t\,.
	\]
	Thus, $u^*Y_t$ and $Y_t$ satisfy the same defining equation, so by uniqueness $u^*Y_t = Y_t$. Consequently $X_t = Y_t + H_t\,R_t$ is an invariant vector field.
\end{proof}
\begin{cor}\label{Girouxbound}
	Let $\xi_0$, $\xi_1$ be two anti-invariant contact structures on $M$, and let $S_0$, $S_1$, be two $u$-oriented compact surfaces. Fix an equivariant orientation-preserving diffeomorphism $\phi \colon S_0 \to S_1$ mapping the characteristic foliation $S_{0,\xi_0}$ to $S_{1,\xi_1}$. There exist invariant tubular neighborhoods and an equivariant contactomorphism $\psi \colon \NN(S_0) \to \NN(S_1)$ such that $\psi|_{S_0} = \phi$. 
	
	Moreover, if $S_0 = S_1$ and $\phi = \id_{S_0}$, the map $\psi$ can be extended to a compactly supported self-diffeomorphism of $M$ which is isotopic to the identity via a compactly supported equivariant isotopy.
\end{cor}
\begin{proof}
	Fix invariant tubular neighborhoods $\NN'(S_i) \simeq S_i \times \R$, as in \cref{tubneigh}, and let $\alpha_i$ ($i = 0,1$), be positive normalized contact forms which define $\xi_i$. Because $\phi$ preserves the characteristic foliations, there exists an invariant smooth function $f \colon S_0 \to (0,\infty)$ such that
	\[
		f\cdot \phi^*(i^*_{S_1}\alpha_1) = i^*_{S_0}\alpha_0\,.
	\]
	
	We extend $\phi$ to a diffeomorphism $S_0 \times \R \to S_1 \times \R$ satisfying
	\[
		\phi(x,z) = (\phi(x,0),\epsilon\,z)\,,
	\]
	where $\epsilon = \pm 1$ is chosen so that the pullback $f\cdot\phi^*\alpha_1$ is positively cooriented. 
	
	Following Geiges' proof, the convex combination 
	\[
		\alpha_t \coloneqq (1-t)\alpha_0 + t\,f\cdot\phi^*\alpha_1\,.
	\]
	is a contact form in a small neighborhood $\NN \subseteq \NN'(S_0)$ for all $t \in [0,1]$. Furthermore, $i^*_{S_0}\dot{\alpha}_t \equiv 0$ for all $t \in [0,1]$.
	
	Let $X_t$ be the invariant time-dependent vector field provided by \cref{equationXtnotonxit} applied to $(\alpha_t)_{t \in [0,1]}$ and $W = S_0$. Since $X_t|_{S_0}$ vanishes identically, $X_t$ can be integrated up to $t = 1$ on $S_0$. Thus, after potentially shrinking $\NN'$, $X_t$ integrates to a flow on $\NN'$. By \cref{rmkXtofpsitsolving}, such flow defines an isotopy 
	\[
		\psi_t \colon \NN' \to M,\ t \in [0,1],
	\]
	which satisfies $d(\phi \circ \psi_1)(\xi_0) = \xi_1$. 
	
	Defining $\psi \coloneqq \phi \circ \psi_1 \colon \NN \to \NN'(S_1)$, we obtain an equivariant contactomorphism satisfying $d\psi(\xi_0) = \xi_1$. Because $X_t|_{S_0}$ vanishes identically, $\psi_t|_{S_0} = \id_{S_0}$, and therefore $\psi|_{S_0} = \phi$.
	
	Finally, if $S_0 = S_1$ and $\phi = \id_{S_0}$, \cref{extisoeq} allows us to extend $\psi_t$ to a compactly supported global isotopy of $M$, completing the proof.
\end{proof}
\section{Equivariant Convex Theory on Spheres}\label{sec:EConvex}
As in the previous section, let $M$ be a 3-manifold, and let us fix on $M$ a rotation-like involution $u$ and an invariant metric $\gk$.
\subsection{Characteristic Foliations of Special Type}\label{ssec:specialtype}
The first objective is to prove a genericity result for equivariant vector fields on surfaces.

Recall that a smooth map $f \colon W \to P$, where $W$, $P$ are boundaryless manifolds, is \emph{transverse} to a submanifold $N \subseteq P$ on a subset $U \subseteq W$ if 
\[
	d_xf[T_xW] + T_{f(x)}N = T_{f(x)}P
\] 
for every $x \in f^{-1}(N) \cap U$. 
\begin{defin}
	An action of a group $G$ on a manifold $X$ is \emph{free} if whenever $g\cdot x = x$ for some $g \in G$ and $x \in X$, the element $g$ is the identity. 
\end{defin}
\begin{prop}\label{Gtrasv}
	Let $G$ be a finite group. Let $W$, $P$ be $G$-manifolds, $C$ a $G$-invariant closed subset of $W$, and $N$ a $G$-submanifold of $P$. Let $f\colon W \to P$ be a smooth equivariant map, transversal to $N$ on an open neighborhood $U \supseteq C$. If the $G$-action on $W \setminus C$ is free, there exists a smooth $G$-map $g \colon W \to P$ that is $C^\infty$-close to $f$ and transverse to $N$ on $W$.
\end{prop}
\begin{proof}
	A proof of the above statement can be found in \cite[Prop.\ 2.2]{leeGroupsJOG1975}.
\end{proof}

\begin{lemma}\label{nondegequiv}
	Let $S \subseteq \interior(M)$ be a closed $u$-oriented surface, and let 
	\[
		X \colon S \times [-1,1] \to TS
	\]
	be a smooth 1-parameter anti-invariant vector field such that $X_z(p) \neq 0$ for all $p \in \Fix(u|_S)$ and $z \in [-1,1]$. Then $X$ can be $C^\infty$-approximated by an anti-invariant 1-parameter vector field that is transverse to the zero-section (as a map from $S \times [-1,1]$ to $TS$).
	
	The analogous result holds for a non-parametric vector field.
\end{lemma}
\begin{proof}
	Let $r$ be the involution of $S \times [-1,1]$ defined as $u|_S \times \id_{[-1,1]}$. Set $G \coloneqq \Z/2\Z$, and consider the $G$-spaces $W \coloneqq S \times [-1,1]$ and $P \coloneqq TS$, the latter equipped with the action of $-du|_S$.
	
	Anti-invariant 1-parameter vector fields coincide with the $G$-equivariant 1-parameter sections of the tangent bundle $TS$. Since $X$ is non-vanishing on $\Fix(r)$, this property holds in an open set $U \subseteq W$ containing $\Fix(r)$. Hence $X|_U$ is vacuously transverse to the zero-section. Furthermore, the group $G$ acts freely on $W \setminus \Fix(r)$.
	
	Thus, by \cref{Gtrasv} we can approximate $X$ with a $G$-map 
	\[
		X' \colon S \times [-1,1] \to TS
	\] 
	that is transverse to the zero-section. Define 
	\[
		h \colon S \times [-1,1] \to S
	\]
	by $h \coloneqq \pi \circ X'$, where $\pi \colon TS \to S$ is the canonical projection.
	
	Since $X'$ is arbitrarily $C^\infty$-close to a section map, the map $h$ is arbitrarily close to the canonical projection. Thus, we may assume that $h_z$ is a diffeomorphism for every $z \in [-1,1]$. Then $Z_z \coloneqq X'_z \circ h^{-1}_z$ defines a 1-parameter vector field $C^\infty$-close to $X$ and transverse to the zero-section. 
	
	Moreover,
	\[
		h \circ r = \pi \circ X' \circ r = \pi \circ (-du) \circ X' =  u \circ \pi \circ X' = u \circ h\,,
	\]
	hence $h$ is $r$-equivariant. Thus, the vector field $Z$ is anti-invariant:
	\[
	\begin{split}
		Z \circ r &= X' \circ h^{-1} \circ r\\
		&= X' \circ u \circ h^{-1} \\
		&= -du \circ X' \circ h^{-1} \\
		&= -du \circ Z\,.
	\end{split}
	\]
	Therefore, $Z$ is the required perturbation.
	
	If $X$ is a non-parametric vector field, then the previous proof proceeds analogously.
\end{proof}

Recall that given a vector field $X$ defined on a surface $S$, a singularity $p \in S$ is transverse to the zero-section, i.e.\ \emph{non-degenerate}, if and only if the following condition holds. When the vector field $X$ is expressed in local coordinates as a map from $\R^2$ to $T\R^2 = \R^2$, the differential $d_pX \colon \R^2 \to \R^2$ has non-zero determinant. Such a point $p$ is necessarily an isolated singularity. Moreover, if the divergence $d\iota_X\Omega/\Omega = \diver_\Omega(X)(p)$ is non-zero for some area form $\Omega \in \Lambda^2(S)$, the point $p$ is said to be \emph{generic}. This condition is independent of the area form $\Omega$, see \cite[Eq.\ 4.21]{geigesIntroductionContactTopology2008}.

A generic singularity $p$ is \emph{elliptic} if the eigenvalues of $d_pX$ have real parts of the same signs. Otherwise, $p$ is called \emph{hyperbolic}.

On the other hand, consider the case where $p$ is an isolated degenerate singularity with non-vanishing divergence. Moreover, let $\pi_1 \colon T_pS \to E_0$ denote the projection onto the $0$-eigenspace $E_0$ of the differential $d_pX$. The point $p$ is \emph{half-hyperbolic} if the projection
\[
	\pi_1(\Lie_V\Lie_VX)(p) \in E_0
\]
is non-zero for all vector fields $V$ defined near $p$ and satisfying $V(p) \in E_0 \setminus \{0\}$. Since $X$ has non-zero divergence at $p$, the space $E_0$ is 1-dimensional, and $\pi_1(\Lie_V\Lie_VX)(p)$ solely depends on $V(p) \in E_0$, see \cite[\S I.3]{sotomayorgenericoneparameter}.

Up to a local diffeomorphism, the local topological models for elliptic, hyperbolic, and half-hyperbolic singularities are shown in \cref{fig:Source,fig:Sink}, \ref{fig:Hyp}, \ref{fig:HalfHypSource} and \ref{fig:HalfHypSink} respectively (see \cite[p.\ 166-170]{geigesIntroductionContactTopology2008} and \cite[Lemma I.3.3]{sotomayorgenericoneparameter}). The so-called \emph{separatrices} are shown in red: those entering the singularity are called \emph{stable}, while the ones exiting are called \emph{unstable}.
\begin{figure}[tb]
	\centering
	\subcaptionbox{\label{fig:Source}}[0.3\textwidth]{
		\def\svgwidth{0.7\linewidth}
		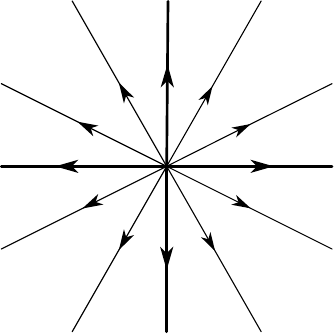}
	\subcaptionbox{\label{fig:Sink}}[0.3\textwidth]{
		\def\svgwidth{0.7\linewidth}
		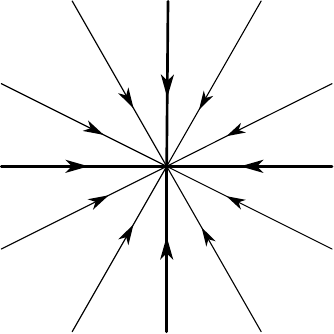}
	\subcaptionbox{\label{fig:Hyp}}[0.3\textwidth]{
		\def\svgwidth{0.7\linewidth}
		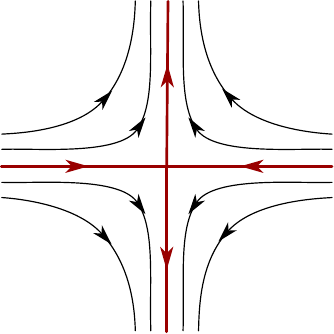}\\
	\subcaptionbox{\label{fig:HalfHypSource}}[0.3\textwidth]{
		\def\svgwidth{0.7\linewidth}
		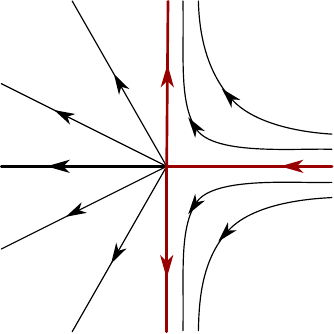}
	\subcaptionbox{\label{fig:HalfHypSink}}[0.3\textwidth]{
		\def\svgwidth{0.7\linewidth}
		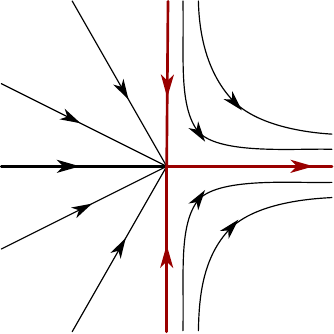}
	\caption{Topological models for the mentioned singularities.}
	\label{fig:Singularities}
\end{figure}

\begin{defin}\label{specialtype}
	The $\alpha$– or $\omega$–limit set of a flow line $t \mapsto \phi_t(x_0)$ (equivalently of the point $x_0$) is the set of points that can be written as $\lim_{k \to \infty}\phi_{t_k}(x_0)$ for some sequence $t_k$ going to $-\infty$ or $+\infty$, respectively. A vector field $X$ is of \emph{special type} if the following conditions hold:
	\begin{enumerate}
		\item there are only finitely many non-generic singularities, all of which are half-hyperbolic;
		\item there are no closed orbits;
		\item the $\alpha$- and $\omega$- limit of each trajectory is either a singularity, or a poly-cycle made up of (half-)hyperbolic points and connecting separatrices;
		\item there are no \emph{retrograde trajectories}, i.e.\ those running from a \linebreak (half-)hyperbolic point with negative divergence to a (half-)hyperbolic point with positive divergence.
	\end{enumerate}
\end{defin}
Let $c$ be a limit poly-cycle. On $c$ the vector field $X$ defines a coherent orientation and a sequence $(h_1, \dots, h_n = h_1)$ of connected hyperbolic points. Such a poly-cycle needs not to be an embedded polygon, as shown in \cref{fig:degeneratepolycycle}. 
\begin{figure}[t]
	\centering
	\def\svgwidth{0.5\linewidth}
	%% Creator: Inkscape 1.4.3 (0d15f75, 2025-12-25), www.inkscape.org
%% Accompanies image file 'degeneratepolycycle_svg-tex.pdf' (pdf, eps, ps)
\begingroup%
  \makeatletter%
  \providecommand\color[2][]{%
    \errmessage{(Inkscape) Color is used for the text in Inkscape, but the package 'color.sty' is not loaded}%
    \renewcommand\color[2][]{}%
  }%
  \providecommand\transparent[1]{%
    \errmessage{(Inkscape) Transparency is used (non-zero) for the text in Inkscape, but the package 'transparent.sty' is not loaded}%
    \renewcommand\transparent[1]{}%
  }%
  \providecommand\rotatebox[2]{#2}%
  \newcommand*\fsize{\dimexpr\f@size pt\relax}%
  \newcommand*\lineheight[1]{\fontsize{\fsize}{#1\fsize}\selectfont}%
  \ifx\svgwidth\undefined%
    \setlength{\unitlength}{187.86703275bp}%
    \ifx\svgscale\undefined%
      \relax%
    \else%
      \setlength{\unitlength}{\unitlength * \real{\svgscale}}%
    \fi%
  \else%
    \setlength{\unitlength}{\svgwidth}%
  \fi%
  \global\let\svgwidth\undefined%
  \global\let\svgscale\undefined%
  \makeatother%
  \begin{picture}(1,0.38408069)%
    \lineheight{1}%
    \setlength\tabcolsep{0pt}%
    \put(0,0){\includegraphics[width=\unitlength,page=1]{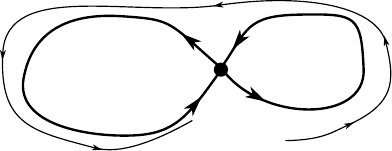}}%
    \put(0.60504045,0.20761749){\color[rgb]{0,0,0}\makebox(0,0)[lt]{\lineheight{1.25}\smash{\begin{tabular}[t]{l}$h_1 = h_n$\end{tabular}}}}%
  \end{picture}%
\endgroup%

	\caption{A degenerate limit poly-cycle.}
	\label{fig:degeneratepolycycle}
\end{figure}
\begin{defin}
	A subcycle is a subsequence $(h_i, \dots h_{i+k})$ such that the only repeated vertices are $h_i = h_{i+k}$.
\end{defin}
Any limit poly-cycle contains a subcycle, which, on a 2-sphere, bounds an oriented embedded polygon with piece-wise smooth boundary.

\begin{rmk}\label{rmksituationXgeneratescharfol}
	If $M$ is equipped with a contact structure $\xi$ and $X$ generates the characteristic foliation of a surface $S \subseteq \interior(M)$, the sign of a singularity $p$, as defined in \cref{sec:introdution}, coincides with the sign of the divergence $\diver_\Omega(X)(p)$ for any area form $\Omega \in \Lambda^2(S)$. 
	
	Moreover, by \cite[Lemma 4.6.3]{geigesIntroductionContactTopology2008} $X$ satisfies the $C^1$-open condition 
	\begin{equation}\label{eq:divcond}
		\diver_\Omega(X)(p)\neq 0 \text{ if } X(p)=0\,,
	\end{equation}
	and thus any non-degenerate singularity of $X$ is generic. 
	
	On the other hand, suppose that $S$ is an embedded 2-sphere and that $\xi$ is tight. The vector field $X$ cannot have closed orbits, since those would bound a disk violating the tightness condition \cite[Prop.\ 4.6.28]{geigesIntroductionContactTopology2008}. In particular, if $X$ has a finite number of singularities the Poincaré--Bendixson theorem \cite[Thm.\ 4.7.8]{geigesIntroductionContactTopology2008} implies that $X$ satisfies condition 3) of \cref{specialtype}. 
\end{rmk}

Equip $M$ with an anti-invariant tight contact structure $\xi$, and let $S \subseteq \interior(M)$ be an embedded $u$-oriented 2-sphere. In \cref{perturbgenericsingularities} we will prove that $S_\xi$ is generically of special type. Observe that $u$ reverses the signs of the singularities of $X$, since $du$ reverses the orientation of $\xi$, while keeping the orientation of the surface $S$.

\begin{lemma}\label{smallchangesvecfield}
	Let $X \in \XX(S)$ be an anti-invariant vector field generating $S_\xi$. Then $X$ admits a $C^1$-small open neighborhood $U \subseteq \XX(S)$ with the following property. For every anti-invariant vector field $Z \in U$, there exists a global equivariant isotopy $\psi_t \colon M \to M$ such that $Z$ generates $S_{d\psi_1(\xi)}$. Moreover, the support of $\psi_t$ can be chosen to lie arbitrarily close to $S$.
\end{lemma}
\begin{proof}
	Let $U'$ be the set of anti-invariant vector fields on $S$ satisfying condition \eqref{eq:divcond}. According to \cref{rmksituationXgeneratescharfol}, the field $X$ belongs to $U'$.
	
	If $\tilde{\alpha}$ is the unique normalized positive anti-invariant contact form for $\xi$, then $\iota_X \Omega_\gk = f\,\iota^*_S\tilde{\alpha}$ for a unique $u$-invariant function $f \colon S \to (0,+\infty)$. If we extend $f$ smoothly and equivariantly outside $S$ using an arbitrarily small invariant neighborhood $\NN \simeq S \times \R$, the form $\alpha \coloneqq f\,\tilde{\alpha}$ satisfies $i^*_S\alpha = \iota_X\Omega_\gk$. 
	
	Since condition \eqref{eq:contcond} is $C^1$-open, there exists some $C > 0$ such that all the forms in
	\[
		T_C \coloneqq \{\beta \in \Lambda^1(\NN) \mid \Abs{\beta - \alpha}_{C^1,\ \abs{z} \le 1} < C \}
	\]
	satisfy condition \eqref{eq:contcond}.
	
	Next, choose an even smooth function $\rho \colon \R \to \R$ supported inside $\{\abs{z} \le 1\}$ with $\rho(0) = 1$. Since $S$ is compact and $\rho$ has compact support, the
	\[
		W_C \coloneqq \{\beta \in \Lambda^1(\NN) \mid \Abs{\rho(z) \cdot (\beta-\alpha)}_{C^1} < C\}\,.
	\]
	set is well-defined, and open in the (weak) $C^1$-topology.
	
	According to \cite[Lemma 4.6.3]{geigesIntroductionContactTopology2008}, for every anti-invariant vector field $Z \in U'$, there exists an explicit 1-form $\beta_Z \in \Lambda^1(\NN)$ such that 
	\[
		i^*_S\beta_Z = \iota_Z\Omega_{\gk}.
	\] 
	One can directly verify that since $Z$ is anti-invariant, Geiges' construction provides an anti-invariant form $\beta_Z$. Moreover, if $Z$ and $X$ are $C^1$-close, then also $\beta_Z|_{\abs{z} \le 1}$ and $\alpha|_{\abs{z} \le 1}$ are $C^1$-close. In particular, the set
	\[
		U \coloneqq \{Z \in U' \mid \beta_Z \in W_C\}
	\]
	is $C^1$-open.
	
	Pick $Z \in U$, and set $\beta \coloneqq \beta_Z$. The 1-form
	\[
		\alpha_t \coloneqq t\rho\cdot\beta + (1-t\rho)\cdot\alpha,\ t \in [0,1],
	\] 
	belongs to $T_C$ for all $t \in [0,1]$, and is therefore a contact form. Moreover, 
	\[
		\alpha_0 = \alpha, \text{ and } \alpha_1 = 
		\begin{cases}
			\alpha & \text{for } \abs{z} \ge 1\\
			\beta & \text{for } z = 0
		\end{cases}.
	\]
	
	According to \cref{EGray}, there exists a 1-parameter invariant vector field $Y_t$ solving \cref{eq:Xt}. Since $\dot{\alpha}_t$ vanishes identically for $\abs{z} > 1$, the vector field $Y_t$ has compact support. Let $\psi_t$ be its equivariant integral flow. 
	
	By \cref{rmkXtofpsitsolving}, the structure $d\psi_1(\xi)$ equals $\ker(\alpha_1)$, which in turn equals $\ker(\beta)$ along $S = \{z = 0\}$. In particular, $Z$ generates $S_{d\psi_1(\xi)}$.
\end{proof}

\begin{prop}\label{perturbgenericsingularities}
	There is a compactly supported equivariant isotopy $\psi_t$ such that $S_{d\psi_1(\xi)}$ has only generic singularities. Moreover, the support of $\psi_t$ can be chosen to lie arbitrarily close to $S$.
\end{prop}
\begin{proof}
	Let $X \in \XX(S)$ be an anti-invariant vector field generating $S_\xi$. For any point $p \in \Fix(u|_S)$ the planes $\xi_p$ and $T_pS$ cannot coincide, since $d_pu$ reverses the orientation of the former, while keeping the orientation of the latter. So $X(p) \neq 0$ for all $p \in \Fix(u|_S)$.
	
	Thus, by \cref{nondegequiv} there exists a perturbed vector field $X' \in \XX(S)$ arbitrarily $C^1$-close to $X$ having only non-degenerate singularities. By \cref{rmksituationXgeneratescharfol}, any non-degenerate singularity of $X'$ is generic. Applying \cref{smallchangesvecfield} to $X'$ yields the desired compactly supported equivariant isotopy $\psi_t$.
\end{proof} 

Let $X$ generate $S_\xi$. By \cref{perturbgenericsingularities}, after applying an equivariant isotopy of $\xi$ supported in a neighborhood of $S$, we may assume that $X$ has only generic singularities. Since $S$ is compact, there are only finitely many such singularities. Thus, by \cref{rmksituationXgeneratescharfol}, the vector field $X$ satisfies conditions 2)  and 3) of \cref{specialtype}.

Our next goal is to perturb $\xi$ so that $S_\xi$ satisfies condition 4) of \cref{specialtype}.

\begin{lemma}\label{alphaomegalimittight}
	Let $\gamma$ be a trajectory of $X$. The $\alpha$- and $\omega$-limit sets of $\gamma$ are either a singularity, or a poly-cycle containing a retrograde connection. Moreover, in the second case, a retrograde connection can be found in any embedded subcycle of such poly-cycle.
\end{lemma}
\begin{proof}
	By \cref{rmksituationXgeneratescharfol}, the $\alpha$- and $\omega$-limit sets of $\gamma$ consist of either a single singularity, or a poly-cycle made up of hyperbolic singularities and connecting separatrices. If we are in the second case, let $c$ be any subcycle of the limit poly-cycle. The foliation $S_\xi$ defines on $c$ a coherent orientation.
	
	Moreover, since $\xi$ is tight, by \cite[Lemma 4.6.33]{geigesIntroductionContactTopology2008} the poly-cycle $c$ cannot consist of singularities of the same sign. 
	 
 	Let $h_-$ be a negative hyperbolic point contained in $c$. Following the orientation of $c$ yields a sequence of hyperbolic points $(h_i)_{i = 1}^n$, with $h_1 = h_n = h_-$. If $\bimath$ is the first index such that $h_{\bimath}$ is positive, then the separatrix connecting $h_{\bimath - 1}$ with $h_{\bimath}$ is the desired retrograde connection.
\end{proof}
As a consequence of \cref{alphaomegalimittight}, any limit poly-cycle $c = (h_1, \dots h_k)$ defines a non-empty set
\[
	P_c \coloneqq \{p \in \{h_i, \dots, h_{i+k}\} \mid (h_i, \dots, h_{i+k}) \text{ is a subcycle of $c$ and $p \neq h_i$}\}\,.
\]

Next, consider a retrograde connection $\gamma \subseteq S_\xi$, connecting a negative hyperbolic point $h_-$ to a positive one $h_+$. A \emph{$\omega$-$\delta$-sequence}
\[
	(\omega_n,\delta_n)_{n = 0}^k, \quad k \in \N \cup \{\infty\},
\]
is defined as follows. Set $(\omega_0, \delta_0) \coloneqq  (h_-, \gamma)$. If $\omega_n$ is a hyperbolic point, define
\[
\tilde{\omega}_{n+1} \coloneqq
\begin{cases}
	\text{$\alpha$-limit set of $\delta_n$} & \text{if $\omega_n$ has positive sign,}\\
	\text{$\omega$-limit set of $\delta_n$} & \text{if $\omega_n$ has negative sign,}
\end{cases}
\] 
and 
\[
\omega_{n+1} \coloneqq
\begin{cases}
	\tilde{\omega}_{n+1} & \text{if $\tilde{\omega}_{n+1}$ is a singularity,}\\
	\text{an arbitrary element in } P_{\tilde{\omega}_{n+1}} & \text{if $\tilde{\omega}_{n+1}$ is a poly-cycle,}
\end{cases}
\] 
while the definition of $\delta_{n+1}$ depensd on whether $\tilde{\omega}_{n+1}$ is a singularity or a poly-cycle. If $\tilde{\omega}_{n+1}$ is a singularity, define
\[
\delta_{n+1} \coloneqq
\begin{cases}
	\text{the stable sep.\   of $\omega_{n+1}$ different from $\delta_n$} & \text{if } \sign(\omega_{n+1}) > 0,\\
	\text{the unstable sep.\ of $\omega_{n+1}$ different from $\delta_n$} & \text{if } \sign(\omega_{n+1}) < 0,
\end{cases}
\]
(where ``sep.'' abbreviates separatrix). On the other hand, if $\tilde{\omega}_{n+1}$ is a poly-cycle, define
\[
\delta_{n+1} \coloneqq
\begin{cases}
	\text{the stable sep.\ of $\omega_{n+1}$ not contained in $\tilde{\omega}_{n+1}$} & \text{if } \sign(\omega_{n+1}) > 0,\\
	\text{the unstable sep.\ of $\omega_{n+1}$ not contained in $\tilde{\omega}_{n+1}$} & \text{if } \sign(\omega_{n+1}) < 0,
\end{cases}
\]
(see \cref{fig:ChainPolyCycle}).

If $\tilde{\omega}_{n+1}$ is neither a hyperbolic point nor a poly-cycle, then $(\omega_n,\delta_n)_n$ terminates.

Let us prove some dynamical properties of $X$ and an associated $\omega$-$\delta$-sequence.
\begin{figure}[t]
	\centering
	\def\svgwidth{0.45\linewidth}
	%% Creator: Inkscape 1.4.3 (0d15f75, 2025-12-25), www.inkscape.org
%% Accompanies image file 'ChainPolyCycle_svg-tex.pdf' (pdf, eps, ps)
\begingroup%
  \makeatletter%
  \providecommand\color[2][]{%
    \errmessage{(Inkscape) Color is used for the text in Inkscape, but the package 'color.sty' is not loaded}%
    \renewcommand\color[2][]{}%
  }%
  \providecommand\transparent[1]{%
    \errmessage{(Inkscape) Transparency is used (non-zero) for the text in Inkscape, but the package 'transparent.sty' is not loaded}%
    \renewcommand\transparent[1]{}%
  }%
  \providecommand\rotatebox[2]{#2}%
  \newcommand*\fsize{\dimexpr\f@size pt\relax}%
  \newcommand*\lineheight[1]{\fontsize{\fsize}{#1\fsize}\selectfont}%
  \ifx\svgwidth\undefined%
    \setlength{\unitlength}{196.687562bp}%
    \ifx\svgscale\undefined%
      \relax%
    \else%
      \setlength{\unitlength}{\unitlength * \real{\svgscale}}%
    \fi%
  \else%
    \setlength{\unitlength}{\svgwidth}%
  \fi%
  \global\let\svgwidth\undefined%
  \global\let\svgscale\undefined%
  \makeatother%
  \begin{picture}(1,0.752432)%
    \lineheight{1}%
    \setlength\tabcolsep{0pt}%
    \put(0,0){\includegraphics[width=\unitlength,page=1]{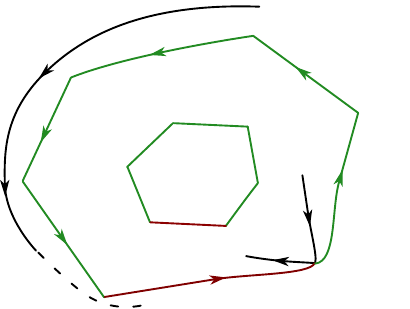}}%
    \put(0.89848903,0.26746487){\color[rgb]{0.13333333,0.54509804,0.13333333}\makebox(0,0)[lt]{\lineheight{1.25}\smash{\begin{tabular}[t]{l}$\subseteq\tilde{\omega}_n$\end{tabular}}}}%
    \put(0.0318848,0.714914){\color[rgb]{0,0,0}\makebox(0,0)[lt]{\lineheight{1.25}\smash{\begin{tabular}[t]{l}$\delta_{n-1}$\end{tabular}}}}%
    \put(0,0){\includegraphics[width=\unitlength,page=2]{ChainPolyCycle_svg-tex.pdf}}%
    \put(0.73910619,0.6318494){\color[rgb]{0.13333333,0.54509804,0.13333333}\makebox(0,0)[lt]{\lineheight{1.25}\smash{\begin{tabular}[t]{l}$B_n$\end{tabular}}}}%
    \put(0.54512527,0.53060778){\color[rgb]{0.13333333,0.54509804,0.13333333}\makebox(0,0)[lt]{\lineheight{1.25}\smash{\begin{tabular}[t]{l}$A_n'$\end{tabular}}}}%
    \put(0,0){\includegraphics[width=\unitlength,page=3]{ChainPolyCycle_svg-tex.pdf}}%
  \end{picture}%
\endgroup%

	\caption{A possible step in an $\omega$-$\delta$-sequence. Retrograde connections are shown in red, while blue arrows denote how the trajectories are followed.}
	\label{fig:ChainPolyCycle}
\end{figure}
\begin{lemma}\label{noinfinitesequencepolycycles}
	Let $\gamma$ be a retrograde trajectory of $X$, and let $(\omega_n, \delta_n)_n$ be an associated $\omega$-$\delta$-sequence. The set
	\begin{equation}\label{eq:setpolycycle}
		\{n \in \N \mid \tilde{\omega}_n \text{ is a poly-cycle}\}
	\end{equation}
	is finite.
\end{lemma}
\begin{proof}
	Consider \cref{fig:ChainPolyCycle}. Any poly-cycle $\tilde{\omega}_n$ divides the sphere $S$ into two open (possibly disconnected) regions, $A_n$ and $B_n$, where $B_n$ is connected and contains the orbit $\delta_{n-1}$. A subcycle $c$ of $\tilde{\omega}_n$ bounds a connected component $A_n'$ of $A_n$. 
	
	In particular, $c$ cannot be an $\alpha$- or $\omega$-limit set of a trajectory originating in $A_n'$: the separatrices of $\omega_n$ not belonging to $c$ do not allow a trajectory in $A_n$ to have $\tilde{\omega}_n$ as limit poly-cycle. Hence, all subsequent separatrices $\delta_m$ ($m > n$) lie in $A_n'$. 
	
	Thus, if set \eqref{eq:setpolycycle} is infinite, it would define a subsequence $(\tilde{\omega}_{n_k})_{k \in \N}$ of poly-cycles with strictly nested associated domains. Such a sequence would accumulate on a non-isolated singularity, contradicting condition 1) of \cref{specialtype}.
\end{proof}

For the next lemma, we will briefly abandon the equivariant setting.
\begin{lemma}\label{nocircleretrograde}
	Let $\Sigma \subseteq M$ be any embedded oriented 2-sphere, equipped with a positive area form $\Omega$. If $Z$ generates $\Sigma_\xi$, $Z$ cannot contain an embedded poly-cycle $c$ consisting only of retrograde connections.
\end{lemma}
\begin{proof}
	Let $c$ be a poly-cycle consisting solely of retrograde connections. It defines a smooth embedded curve, as every singularity $p \in c$ is adjacent to two separatrices that are both stable or both unstable. 
	
	We want to prove that along $c$, the contact structure $\xi|_c$ rotates positively with respect to the tangent distribution $T\Sigma|_c$. This will imply that the Thurston--Bennequin invariant $\tb(c)$ is positive \cite[\S 3.5.1]{geigesIntroductionContactTopology2008}. Since $c$ bounds a 2-disk $D \subseteq \Sigma$, this computation will contradict the Bennequin inequality \cite[Thm.\ 4.6.36]{geigesIntroductionContactTopology2008}.
	
	Let us fix a (retrograde) separatrix $\gamma \subseteq c$ (see \cref{fig:angolo}). Orient $\gamma$ according to the orientation of $\Sigma_\xi$, and orient $D$ according to the outward-vector-first convention. Let us define the \emph{angle map}
	\[
		\theta \colon \gamma \to [0,2\pi]
	\]
	by assigning to each point $p \in \gamma$ the rotation required to align the oriented plane $T_p\Sigma$ with $\xi_p$, measured counterclockwise with respect to $\vec{\gamma}$. These oriented planes, along with the dual vector $\alpha_p^{*_\gk}$, are illustrated in \cref{fig:angolo}. At the negative (resp.\ positive) endpoint $p$ of $\gamma$, the angle $\theta(p)$ equals $\pi$ (resp.\  $0$ or $2\pi$).
	
	On the other hand, let $p$ be a point in the interior of $\gamma$, and let $v_p \in T_p\Sigma$ be an arbitrary vector complementing $Z_p$ to an $\Omega$-positive basis of $T_p\Sigma$. Since
	\[
		\alpha(v_p) = \iota_{Z_p}\Omega(v_p) = \Omega(Z_p,v_p) > 0\,,
	\]
	the angle $\theta(p)$ belongs to $(\pi,2\pi)$. 
	
	Thus, the map $\theta$ has values ranging in $[\pi,2\pi]$, starting at $\pi$ and ending at $2\pi$. Hence, $\theta$ is homotopic rel boundary to the linear map from $\pi$ to $2\pi$. Consequently, $\xi|_\gamma$ undergoes a positive half-twist with respect to $T\Sigma|_\gamma$.
	
	Since $\gamma \subseteq c$ was chosen arbitrarily, the distribution $\xi|_c$ rotates positively with respect to $T\Sigma|_c$. 
\end{proof}
\begin{figure}[t]
	\centering
	\begin{tikzpicture}[x=0.75pt, y=0.75pt,>=Stealth]
		\draw (-50,0) -- (50,0);
		\draw (30:40) -- (210:40);
		\draw (-50,0) -- +(25,20);
		\draw (50,0) -- +(25,20);
		\draw (30:40) -- +(20,20);
		\draw (210:40) -- +(20,20);
		\draw[decoration={markings, mark = at position 0.5 with \arrow{>}},postaction={decorate}] (25,25) ..controls +(-150:0.2cm) and +(60:0.2cm).. (0,0);
		\draw[->] (-70,10) arc [x radius=10, y radius= 5, start angle = 170, end angle = 440];
		
		\draw[->] (-15,15) -- +(-135:17);
		\draw[->] (-15,15) -- +(0:15);

		\draw[->] (0,0) -- (-45:30);
		\node at (-45:30) [font = \small, shift = {(15,5)}] {$\alpha_p^{*_\gk}$};
		\node at (80,0) [font = \small] {$T_pS$};
		\node at (-45,-25) [font = \small] {$\xi_p$};
		\node at (10,25) [font = \small] {$\vec{\gamma}$}; 
		
		\begin{scope}[shift = {(200,0)}]
			\draw (0,0) circle [radius=40];
			\filldraw (90:40) circle [radius=1.5pt];
			\filldraw (310:40) circle [radius=1.5pt];
			\node at (0,0) [font = \small,shift = {(0,-25)}] {$\theta$};
			\draw (0,5) -- (90:40);
			\draw (310:5) -- (310:40);
			\draw (0,0) circle [radius=5];
			\filldraw (0,0) circle [radius=1];
			\draw[->] (80:15) arc [radius=15, start angle = 80, end angle = 330];
			\node at (30:13) [font = \small] {$\vec{\gamma}$};
		\end{scope}
	\end{tikzpicture}
	\caption{Visualization of the angle $\theta(p)$.}
	\label{fig:angolo}
\end{figure}
\begin{prop}\label{sequenceeliminating}
	If $\gamma$ is a retrograde trajectory of $X$, then any associated $\omega$-$\delta$-sequence terminates. That is, there exists a positive integer $m$ such that $\delta_m$ connects a hyperbolic point with an elliptic singularity of the same sign. 
\end{prop}
\begin{proof}
	Suppose for contradiction that $(\omega_n, \delta_n)_n$ is defined for all $n \in \N$. By \cref{noinfinitesequencepolycycles}, there exists a positive integer $m$ such that $\tilde{\omega}_i$ is a singularity for every $i \ge m$; in particular, $\tilde{\omega}_i = \omega_i$ for all $i \ge m$. Moreover, by the compactness of $S$, $X$ has only finitely many singularities. Thus, there exists an integer $n > m$ such that $\omega_n = \omega_m$. By extracting a sub-loop, may assume that there are no other repeated singularities in the sequence $(\omega_i)_{m \le i \le n}$.
	
	The sequence $(\delta_i)_{m \le i \le n}$ defines an \emph{embedded polygon} $Q$
	\begin{enumerate}
		\item whose edges are separatrices of hyperbolic points in $X$, and
		\item whose vertices are hyperbolic points.
	\end{enumerate} 
	By construction, at every positive (resp.\ negative) vertex $\omega_i$ ($m \le i \le n$), $\delta_i$ is a stable (resp.\ unstable) separatrix of $\omega_i$. Thus, a vertex $\omega_i$ shares with $\partial Q$ a stable and an unstable separatrix if and only if $\omega_i$ is adjacent to another vertex \emph{of the same sign}. 
	
	We now recall the non-equivariant deformation described in \cite[p.\ 201]{geigesIntroductionContactTopology2008}, which allows us to suppose that $\partial Q$ contains only hyperbolic vertices sharing with $\partial Q$ two stable or two unstable separatrices. Consequently, the new embedded polygon contains vertices of alternating sign, and every remaining edge is retrograde, contradicting \cref{nocircleretrograde}.
\end{proof}

The next result proves that special type characteristic foliations are in fact generic. For future use, we recall the objects involved.
\begin{prop}\label{Cinftyspecialtype}
	Let $\xi$ be an anti-invariant tight contact structure on $M$, and let $S \subseteq \interior(M)$ be a $u$-oriented embedded 2-sphere. There exists a global equivariant isotopy $\phi_t \colon M \to M$ such that $S_{d\phi_1(\xi)}$ has only generic singularities and is of special type. Moreover, the support of $\phi_t$ can be chosen to lie arbitrarily close to $S$.
\end{prop}
\begin{proof}
	Choose a generating vector field $X$ for $S_\xi$. As noted previously, we may assume that $X$ satisfies conditions 1)-3) of \cref{specialtype}. Let $H_\xi$ be the number of hyperbolic-hyperbolic connections, and let $R_\xi$ be the number of those that are retrograde. 
	
	Suppose $R_\xi > 0$, and choose a retrograde connection $\gamma$. If $(\omega_n, \delta_n)_n$ is an $\omega$-$\delta$-sequence starting at $\gamma$, then by \cref{sequenceeliminating} there exists some positive integer $m$ such that $\delta_m$ connects a hyperbolic point to  an elliptic singularity of the same sign.
	
	Since the involution $u$ reverses the sign of the singular points, the set $\{\omega_m, \omega_{m+1}\}$  is disjoint from $\{u(\omega_m), u(\omega_{m+1})\}$.  Thus, there exists a neighborhood $U \subseteq S$ of $\delta_m$ disjoint from $u(U)$. Let $\psi_t \colon M \to M$ be the isotopy defined in \cite[Prop.\ 4.6.26]{geigesIntroductionContactTopology2008}, supported in a neighborhood $\NN \subseteq M$ of $U$. The following equivariant isotopy
	\[
		\phi_t \coloneqq
		\begin{cases}
			\psi_t^{-1} & \text{on } \NN\\
			u \circ \psi_t^{-1} \circ u & \text{on } u(\NN)\\
			\id_M & \text{elsewhere}
		\end{cases}
	\]
	is well-defined. 
	
	Moreover, by Geiges' arguments, the deformed contact structure $\xi' \coloneqq d\phi_1(\xi)$ has the property that $S_{\xi'}$ has two fewer hyperbolic-hyperbolic connections. Furthermore, no other hyperbolic-hyperbolic connections are introduced, and $\xi'$ satisfies conditions 1)-3), of \cref{specialtype}. 
	
	Hence, the chain of inequalities
	\[
		R_{\xi'} \le H_{\xi'} \le H_{\xi} - 2 < H_{\xi}
	\]
	holds. In particular, after a finite number of perturbations $\xi_1, \dots, \xi_n$, we obtain $R_{\xi_n} = 0$. The foliation $S_{\xi_n}$ is then of special type.
\end{proof}

\subsection{Genericity of Convex Invariant Spheres}

We now focus on the theory of convex surfaces. For the rest of the section, let $\xi$ be an anti-invariant tight contact structure on $M$.
\begin{defin}\label{defdividingcurveity}
	A $u$-oriented surface $S\subseteq \interior(M)$ is convex if there exist
	\begin{enumerate}
		\item an embedded invariant tubular neighborhood $\Psi \colon S \times \R \hookrightarrow M$, and
		\item an anti-invariant contact form $\alpha$ for $\xi$,
	\end{enumerate} 
	such that $\Psi^*\alpha$ determines a vertically invariant contact form.
\end{defin}

\begin{rmk}\label{invariantcontactcondition}
	If $S$ is convex, in an invariant neighborhood of $S$ the contact form $\alpha$ can be expressed as $\beta + w\,dz$, for some 1-form $\beta \in \Lambda^1(S)$ and some smooth function $w \in C^\infty(S)$. In this setting, the contact condition \eqref{eq:contcond} simplifies to
	\begin{equation}\label{eq:invariantcontcond}
		w\,d\beta + \beta \wedge dw > 0\,.
	\end{equation}
	In particular, for every point $p \in w^{-1}(0)$ the differential $d_pw$ is non-zero.
\end{rmk}

As in the non-equivariant scenario, the information regarding the contact structure near a convex surface is contained in a family of circles, called \emph{dividing curves}.
\begin{defin}\label{defdividingcurve}
	Let $\FF$ be an anti-invariant singular foliation on a $u$-oriented closed surface $S \subseteq \interior(M)$. An invariant collection $\Gamma$ of embedded circles is said to \emph{divide} $\FF$ if
	\begin{enumerate}
		\item $\Gamma$ is transverse to $\FF$; in particular, $\Gamma$ does not meet the singularities of $\FF$;
		\item $S$ admits an invariant volume form $\Omega$ and an anti-invariant vector field $X$ generating $\FF$ such that $\diver_\Omega(X)$ does not vanish on $S \setminus \Gamma$;
		\item setting $S_\pm \coloneqq \{p \in S \mid \pm \diver_\Omega(X) > 0\}$, then $S \setminus \Gamma = S_+ \sqcup S_-$ and the vector field $X$ points out of $S_+$ along $\Gamma$.
	\end{enumerate}
\end{defin}

Observe that in \cref{defdividingcurve} the restriction $u|_\Gamma$ is necessarily orientation-reversing, as $u|_S$ is orientation-preserving, $\FF$ is anti-invariant, and $\Gamma$ is transverse to $\FF$.
\begin{rmk}\label{rmkconveximpliesdividingcurve}
	Let $S$ be a convex surface, and let $\beta + w\,dz$ be an anti-invariant contact form generating $\xi$, the latter restricted to a tubular neighborhood of $S$. By \cite[pp.\ 230-231]{geigesIntroductionContactTopology2008}, the set
	\[
		\Gamma = \{p \in S \mid w(p) = 0\}
	\]
	divides the foliation $S_\xi$. Although Geiges works in the non-equivariant setting, it can be directly checked that $\Gamma$ is a dividing curve in the equivariant sense.
\end{rmk}

In \cref{divconvex} we prove the converse: if a surface $S$ admits dividing curves, then $S$ is convex.
\begin{prop}\label{divconvex}
	Let $S \subseteq \interior(M)$ be a $u$-oriented compact surface, and let $\beta_z + w_z\,dz$ be the local expression for an anti-invariant contact form for the restriction of $\xi$ to an invariant neighborhood $\NN \simeq S \times \R$. If the foliation $S_\xi$ is divided by a collection $\Gamma$ of embedded circles, then there exists an anti-invariant function $w \in C^\infty(S)$ such that $\bar{\alpha} = \beta_0 + w\,dz \in \Lambda^1(\NN)$ is a vertically invariant and anti-invariant contact form. 
	
	In particular, by \cref{Girouxbound}, the structures $\xi$ and $\Ker(\bar{\alpha})$ are locally equivariantly contactomorphic, and thus $S$ is a convex surface with respect to $\xi$.
	
	Moreover, the above construction is smooth in the triple $(\Gamma,\Omega,X)$ appearing in \cref{defdividingcurve}. More precisely, for all smooth maps
	\[
	\begin{split}
		\Gamma \times \R&\to S,\\
		S \times \R &\to \Lambda^2(T^*S),\\
		S \times \R &\to TS,
	\end{split}
	\]
	of 1-parameter embedded dividing curve, 1-parameter positive area form, and 1-parameter vector field, there exists a 1-parameter smooth family of functions
	\[
		w \colon S \times \R \to \R 
	\]
	such that $\bar{\alpha}_t \coloneqq \iota_{X_t}\Omega + w_t\,dz$ defines a 1-parameter family of vertically invariant, and anti-invariant, contact forms.
\end{prop}
\begin{proof}
	Although we could check that the proof contained in \cite[Thm.\ 4.8.5]{geigesIntroductionContactTopology2008} already produces invariant forms, we choose a quicker approach.
	
	Set $\beta \coloneqq \iota_X\Omega$. Since $X$ is an anti-invariant vector field and $\Omega$ is an invariant form, $u^*\beta = -\beta$ holds. According to \cite{geigesIntroductionContactTopology2008}, there exists a smooth function $m \in C^\infty(S)$ such that 
	\[
		\gamma \coloneqq \beta + m\,dz 
	\]
	is a (vertically invariant) contact form. Let us define 
	\[
		\Lambda^1(\NN) \ni \gamma^{\eq} \coloneqq \frac{\gamma - u^*\gamma}{2} = \beta + \frac{m - m \circ u}{2}\,dz\,.
	\]
	By construction, $\gamma^{\eq}$ is an anti-invariant 1-form. Moreover, $\gamma^{\eq}$ is a contact form, since Condition \eqref{eq:invariantcontcond} is convex in $w$. 
	
	Since $X$ generates $S_\xi$, there exists a unique positive invariant function $f \in C^\infty(S)$ such that $\beta_0 = f\,\beta$. Setting $m^{\eq} \coloneqq (m - m \circ u)/2$, the 1-form
	\[
		\bar{\alpha} \coloneqq f\,\gamma^{\eq} = \beta_0 + f\,m^{\eq}\,dz
	\] 
	is a vertically invariant and anti-invariant contact form, as desired.
	
	Moreover, Geiges' construction already depends smoothly on the triple $(\Gamma,\Omega,X)$, as does the above averaging procedure.
\end{proof}
\begin{rmk}\label{rmkdivsetS2iscircle}
When $S$ is a sphere, the tightness condition forces the dividing set to be a single circle, see \cite[Prop.\ 4.8.13]{geigesIntroductionContactTopology2008}.
\end{rmk}

Using \cref{divconvex}, we want to prove that a $u$-oriented sphere is generically convex.
\begin{lemma}\label{subregionannulus}
	Let $S \subseteq \interior(M)$ be a $u$-oriented 2-sphere, and let $X \in \XX(S)$ be an anti-invariant vector field that generates $S_\xi$ and has a finite number of singularities. Let $T \subseteq S$ be an invariant region such that
	\begin{enumerate}
		\item $X$ is transversal to $\partial T$,
		\item $\partial T$ can be decomposed as $\partial_+T \sqcup \partial_-T$, where $X$ enters from $\partial_-T$ and exits through $\partial_+T$,
		\item $T$ contains no singularities.
	\end{enumerate}
	Then there exists an invariant 1-submanifold $\Gamma \subseteq \interior(T)$ transverse to $X$.
\end{lemma}
\begin{proof}
	Since $X$ is anti-invariant, the involution $u$ interchanges $\partial_+T$ with $\partial_-T$.
	
	Since $T$ has no singularities, by \cref{rmksituationXgeneratescharfol} the $\omega$-limit of any trajectory originating in $\partial_-T$ cannot be an interior point in $T$. Thus, the integral flow of $X$ originating in $\partial_-T$ must land in $\partial_+T$. Conversely, every trajectory entering $\partial_+T$ must enter through $\partial_-T$. Thus, there exists a unique map
	\[
		\tau \colon \partial_-T \to \R_{> 0}
	\]
	satisfying $\Phi^X_{\tau(p)}(p) \in \partial_+T$. Although it is standard that $\tau$ is smooth, let us recall the reasoning.
	
	Fix $p_0 \in \partial_-T$; we wish to prove that $\tau$ is smooth in a neighborhood of $p_0$. Let $\tau_0 \coloneqq\tau(p_0)$, and let $q_0 \in \partial_+T$ be the point $\Phi^X_{\tau_0}(p_0)$. Since $\partial_+T$ is a codimension-1 submanifold, locally near $q_0$ we can find a smooth function $f \colon U \to \R$ (where $U$ is an open neighborhood in $S$) such that 
	\[
		\partial_+T \cap U = \{ x \in U \mid f(x) = 0 \}.
	\]
	In particular, $f(\Phi^X_{\tau_0}(p_0)) = 0$. Notice that
	\[
		\pdv{}{\tau}\Big|_{\tau=\tau_0}f(\Phi^X_\bullet(p_0))=df(q_{0})\cdot \pdv{\Phi^X_{\bullet}}{\tau}(\tau_0,p_0)=df(q_{0})\cdot X(q_{0})\ne 0
	\]
	since the vector field $X$ is transversal to $\partial_+T$. According to the Implicit Function theorem, the crossing time $\tau$ is a (local) smooth function in terms of $p \in \partial_-T$.
	
	Having established that $\tau$ is smooth, the map
	\[
	\begin{split}
		\Psi \colon \partial_-T \times [0,1] &\to T, \\
		(p,s) &\mapsto \Phi^X_{s\,\tau(p)}(p),
	\end{split}
	\]
	is then a diffeomorphism onto $T$. Moreover, the set
	\[
		\Gamma \coloneqq \{\Psi(p,1/2) \mid p \in \partial_-T\}\,.
	\]
	is an invariant 1-submanifold. Indeed, for every point $\Phi^X_{\tau(p)/2}(p) \in \Gamma$, the following chain of equalities holds:
	\[
	\begin{split}
		(u \circ \Phi^X_{\tau(p)/2})(p) &= \Phi^X_{-\tau(p)/2}(u(p)) \\
		&= \Phi^X_{\tau(p)/2 - \tau(p)}(u(p)) \\
		&= \Phi^X_{\tau(p)/2}\left[\Phi^X_{-\tau(p)}(u(p))\right]\\
		&= \Phi^X_{\tau(p)/2}\left[(u \circ \Phi^X_{\tau(p)})(p)\right] \in \Gamma\,,
	\end{split}
	\]
	since the point $(u \circ \Phi^X_{\tau_p})(p)$ belongs to $\partial_-T$.
\end{proof}
We can now give an equivariant criterion for when a surface is convex. The analogous non-equivariant criterion is classical, cf.\  \cite[Prop.\ II.2.6]{girouxConvexiteTopologieContact1991} and \cite[Prop.\ 2.5]{girouxStructuresContactDimension2000}.
\begin{prop}\label{specialconvex}
	Let $S \subseteq \interior(M)$ be a $u$-oriented 2-sphere whose characteristic foliation $S_\xi$ has finitely many degenerate singular points, all half-hyperbolic, and no retrograde connections. Then $S_\xi$ admits a set of dividing curves, whose definition solely depends on an anti-invariant vector field $X$ generating $S_\xi$. In particular, $S$ is convex by \cref{divconvex}. 
\end{prop}
\begin{proof}
	Let $X$ be an anti-invariant vector field generating $S_\xi$. Define $S^0_+ \subseteq S$ to be an arbitrary submanifold (with boundary) consisting of disks around the positive singular points of $S_\xi$ and bands around the stable separatrices of the positive (half-)hyperbolic points\footnote{Recall that positive half-hyperbolic points have a unique stable separatrix.}. Set $S^0_- = u(S^0_+)$. If $S^0_+$ is sufficiently small, we may assume the following:
	\begin{enumerate}
		\item $S^0_ + \cap S^0_- = \eset$;
		\item $X$ is transverse to $\partial S^0_\pm$.
	\end{enumerate}
	
	By \cref{subregionannulus}, in the interior of $T \coloneqq \overline{S \setminus (S^0_+ \cup S^0_-)}$ there is an invariant 1-submanifold $\Gamma \subseteq S$. Let us partition $S \setminus \Gamma$ as $S_+ \sqcup S_-$, such that $S_\pm$ contains $S_\pm^0$.
	
	According to the classical theory, the surface $S$ is $\xi$-convex in the non-equivariant sense, and $\Gamma$ represents a dividing curve. In particular, there exists a positive non-invariant area form $\omega$ such that $\pm\diver_\omega(X) > 0$ on $S_\pm$. Since $u$ is orientation-preserving, the convex combination
	\[
		\Omega \coloneqq \frac{\omega+ u|_S^*\omega}{2}
	\]
	is an area form, and $\Omega$ satisfies conditions 2) and 3) of \cref{defdividingcurve}. Indeed, for every $p \in S_+$ 
	\[
	\begin{split}
		\diver_\Omega(X)(p) &= \tfrac{1}{2}(\diver_\omega(X)(p) + \diver_{u^*\omega}(X)(p)) \\
		&= \tfrac{1}{2}(\diver_\omega(X)(p) - \diver_{u^*\omega}(u^*X)(p))\\
		&= \tfrac{1}{2}(\diver_\omega(X)(p) - u^*\diver_\omega(X)(p))\\
		&= \tfrac{1}{2}(\diver_\omega(X)(p) -\diver_\omega(X)(u(p)))
	\end{split}
	\]
	which is positive, since $u(p) \in S_-$. Analogously, $\diver_\Omega(X) < 0$ on $S_-$.
\end{proof}
\begin{cor}\label{Cinftyconvex}
	Let $S \subseteq \interior(M)$ be a $u$-oriented 2-sphere. There exists a global equivariant isotopy $\psi_t$ such that $S$ is convex with respect to $d\psi_1(\xi)$. Moreover, the support of $\psi_t$ can be chosen to lie arbitrarily close to $S$.
\end{cor}
\begin{proof}
	By \cref{Cinftyspecialtype} we can perturb $\xi$ so that $S_{d\psi_1(\xi)}$ is of special type. According to \cref{specialconvex}, $S$ is convex with respect to $d\psi_1(\xi)$.
\end{proof}
\section{Equivariant Smale Theorem}\label{sec:eqSmale}
In this section, we will prove an equivariant version of the classical Smale theorem. Let $u \colon \R^3 \to \R^3$ be the rotation by $\pi$ about the $z$-axis, and let $n = (0,0,1)$, $s = (0,0,-1)$ be the north and south poles. 
\begin{thm}\label{eqSmale}
	The topological group
	\[
		Q^u \coloneqq \{f \in \Diff^+(S^2) \mid f \circ u = u \circ f\},
	\]
	is weakly homotopy equivalent to $\OOO(2)$. More precisely, the subgroup $Q^u_\fix \subseteq Q^u$ fixing the poles individually has the weak homotopy type of $\SO(2)$, as does the connected component $Q^u_\swap$ swapping $n$ and $s$. Moreover, the weak homotopy equivalence is given by the composition of one of the two derivative maps
	\[
		Q^u_\fix \xrightarrow{d_p} \GL^+(T_pS^2) \simeq \GL^+_2(\R), \quad p \in \{n,s\},
	\]
	with the standard homotopy equivalence between $\GL_2^+(\R)$ and $\SO(2)$.
\end{thm}
The proof of \cref{eqSmale} partially follows the structure of Smale's proof that
\[
	\Diff^+(S^2) \simeq_\h \SO(3)
\]
\cite{smaleDiffeomorphisms2Sphere1959}, where $\simeq_\h$ denotes homotopy equivalence. The symbol $\simeq_\wh$ will denote weak homotopy equivalence.

Let us recall some useful facts.
\begin{prop}
	Given a Serre fibration
	\[
	F \to E \to B\,,
	\]
	over a path-connected base space $B$, sending $x_0 \in F$ to $b_0 \in B$, cf.\ \cite[p.\ 376]{hatcherAlgebraicTopology2002}, there exists a long exact sequence 
	\begin{equation}\label{eq:exactsequenceSerre}
		\dots \to \pi_i(F,x_0) \to \pi_i(E,x_0) \to \pi_i(B,b_0) \to \pi_{i-1}(F,x_0) \to \dots \to 0\,,
	\end{equation}
	of homotopy groups for $i \ge 1$, and sets for $i = 0$ (recall that the kernel of a map $f \colon (A,a) \to (B,b)$ between pointed sets is the preimage $f^{-1}(b)$). In particular, if $B$ (resp.\ $F$) is weakly contractible, then $F$ and $E$ (resp. $E$ and $B$) are weakly equivalent.  
\end{prop}
\begin{proof}
	See \cite[Thm.\ 4.41]{hatcherAlgebraicTopology2002}.
\end{proof}
\begin{rmk}
	Although $\pi_0(X)$ ($X \in \mathbf{Top}$) is generally just a set, if $G$ is a Lie group, the composition
	\[
	[f]\cdot[g] \coloneqq [fg], \quad [f],[g] \in \pi_0(G),
	\]
	defines a natural group structure on $\pi_0(G)$. 
\end{rmk}
\begin{prop}\label{palaisfibration}
	Let $W$ be a compact manifold, and let $Y \subseteq W$ be a properly embedded submanifold, and let $Z \subseteq W$ be a  closed submanifold. The restriction map from the space $\Diff_Z(W)$ of diffeomorphisms of $W$ fixing $Z$ to the space $\Emb_Z(Y,W)$ of smooth embeddings of $Y$ into $W$, fixing $Z$, is a Serre fibration. Moreover, the result still holds if $Y = \partial W$.
\end{prop}
\begin{proof}
	If $Z = \eset$, the statement is a classical result, proven by Palais \cite{PalaisLocaltrivialityoftherestrictionmap}. More precisely, small perturbations of an embedding of $Y$ into $W$ correspond to small sections of $\XX(TW|_Y)$.
	
	If $Z \neq \eset$, the result is easily adapted by restricting Palais' argument to the space $\XX_Z(TW|_Y)$ of vector fields vanishing on $Z$. 
\end{proof}
The next result is classical.
\begin{prop}\label{diffgroupannulus}
	Let $\Diff_{N\partial}^+(A)$ be the space of orientation-preserving diffeomorphisms of the annulus
	\[
		A=S^1\times[0,1]
	\]
	which are the identity near the boundary. Then $\Diff^+_{N\partial}(A)$ has contractible components, and
	\[
	\pi_0\Diff_{N\partial}^+(A) \cong \Z,
	\]
	where $\cong$ denotes group isomorphism. Moreover, the above group is generated by the Dehn twist
	\[
	T(e^{i\theta},r) =	(e^{i(\theta+2\pi\rho(r))},r),
	\]
	where $\rho = 0$ near $r = 0$ and $\rho = 1$ near $r = 1$.
\end{prop}
\begin{proof} 
	\renewcommand{\baselinestretch}{1.5}
	By the same classical arguments contained in \cite[Prop.\ 1.2-3]{igusaStabilityTheoremSmooth1988}, the space $\Diff^+_{N\partial}(A)$ is weakly homotopy equivalent to the space $\Diff_\partial(A)$ of diffeomorphisms fixing the boundary pointwise. In particular, since $\pi_0(\Diff^+_{\partial}(A)) \cong \Z$ is generated by a Dehn twist (cf.\ \cite[Prop.\ 2.4]{farbPrimerMappingClass2011}), the same holds for $\pi_0(\Diff^+_{N\partial}(A))$.
	
	It remains to show that the group $\Diff^+_\partial(A)$ has contractible components. Set $I \subseteq A$ to be the arc $\{1\} \times [0,1]$. According to \cref{palaisfibration}, the restriction map 
	\[
		p \colon \Diff^+_\partial(A) \to \Emb_{\partial I}(I,A)
	\]
	is a Serre fibration with fiber $\Diff^+_{\partial A\, \cup\, I}(A)$. According to \cite[Thm.\ 5]{gramainTypeDhomotopieGroupe1973}, the base space of $p$ has contractible components. 
	
	On the other hand, the space $\Diff^+_{\partial A \cup I}(A)$ is weakly equivalent to the space $\Diff^+_{N(\partial A \cup I)}(A)$ of diffeomorphisms that are the identity near $\partial A \cup I$. Indeed, the weak equivalence is obtained by applying Igusa's straightening arguments first to a collar of $\partial A$, and then to a neighborhood of $I$. Finally, by cutting along $I$, the fiber space is homeomorphically identified with the diffeomorphism group $\Diff^+_{N\partial}(D^2)$. By \cite{smaleDiffeomorphisms2Sphere1959}, this space is contractible.
\end{proof}

\subsection{Equivariant Diffeomorphisms of the Disk}
Let us equip the disk $D^2$ with the antipodal involution
\[
	\tau \colon D^2 \to D^2, \qquad \tau(x)=-x.
\]
Define
\[
	\FF^\tau \coloneqq \{f\in \Diff^+(D^2) \mid f=\id \text{ near } \partial D^2,\ f\circ \tau=\tau \circ f\}.
\]
\begin{thm}\label{eqthmdisk}
	The space $\FF^\tau$ is weakly contractible.
\end{thm}
The following lemma replaces Smale's approximation arguments, which rely on convex combinations, with arguments based on vector fields. Let $\FF^\tau_\R$ be the space of orientation-preserving, equivariant diffeomorphisms of $\R^2$. Set $\LL_\R^\tau \subseteq \FF_\R^\tau$ to be the subspace of functions which are linear in a neighborhood of $\zero$.
\begin{lemma}\label{cutoffisot}
	Let $f_t \colon \R^2 \to \R^2$, $f_t \in \FF^\tau_\R$ be a smooth isotopy of $\R^2$. Then there exists a smooth isotopy $\phi_t \colon \R^2 \to \R^2$, $\phi_t \in \FF^\tau_\R$, and a small ball $B_\epsilon(\zero) \subseteq B_{1/2}(\zero)$, such that for all $t \in [0,1]$
	\begin{equation}\label{eq:condcutoffisot}
	\begin{cases}
		\phi_t(x) = f_t(x) & \text{for } x \in B_\epsilon(\zero),\\
		\phi_t(x) = f_0(x) & \text{for } f_0(x) \notin B_{3/4}(\zero).
	\end{cases}
	\end{equation}
	Moreover, for any subspace $W \subseteq \FF^\tau_\R$, and any continuous assignment
	\[
	\begin{split}
		W &\to \{f_t \colon \R^2 \to \R^2 \mid \text{smooth isotopy with } f_t \in \FF^\tau_\R\},\\
		 g &\mapsto f_t,
	\end{split}
	\]
	the above construction defines a continuous map
	\[
	\begin{split}
		W \times [0,1] &\to \FF^\tau_\R,\\
		(g,t) &\mapsto \phi^g_t .
	\end{split}
	\]
	If $W$ is compact, then the constant $\epsilon$ appearing in \cref{eq:condcutoffisot} can be chosen independently of $g \in W$.
\end{lemma}
\begin{proof}
	For ease of notation, all the balls will be understood to be centered at $\zero$.
	
	Since $f_t(\zero) \equiv \zero$, the unique time-dependent vector field $X_t$ defined by
	\[
		X_t \circ f_t = \odv{}{t} f_t,
	\]
	vanishes at $\zero$. Let $\chi \colon \R^2 \to [0,1]$ be a radial, smooth cutoff function equal to $1$ on $B_{1/2}$ and supported on $B_{3/4}$. Set $Y_t \coloneqq \chi \cdot X_t$, and let $\phi_t$ be the integral flow generated by $Y_t$ with initial condition $\phi_0 = f_0$.
	
	Since $f_t$ is a continuous map on the compact set $[0,1] \times \{0\}$, choose $\epsilon > 0$ small enough such that
	\begin{equation}\label{eq:inclusionBepsilon}
		\bigcup_{t \in [0,1]}f_t(B_\epsilon) \subseteq B_{1/2}.
	\end{equation}
	Then
	\[
		\phi_t(x) = f_t(x), \quad \forall t \in [0,1],\ \forall x \in B_\epsilon.
	\]
	Indeed, for any $x \in B_\epsilon$
	\[
		(Y_t \circ f_t)(x) = (X_t \circ f_t)(x) = \odv{}{t}f_t(x),
	\]
	which implies $f_t(x) = \phi_t(x)$ by uniqueness of solutions of ODEs.
	
	On the other hand, $Y_t \equiv 0$ outside $B_{3/4}$, so any trajectory startnig at $\phi_0(x) = f_0(x) \notin B_{3/4}$ remains constant: $\phi_t(x) = f_0(x)$ for all $t \in [0,1]$.
	
	Since $f_t$ is equivariant, the vector field $X_t$ is invariant, and since $\chi$ is a radial function, we obtain
	\[
		Y_t \circ \tau  = (\chi \circ \tau) \cdot (X_t \circ \tau) = \chi \cdot (d\tau \circ X_t) = d\tau \circ Y_t.
	\]
	Consequently, $\phi_t$ is equivariant at every time.\vspace{0.2cm}
	
	The continuous dependence of ODE flows on parameters implies that the assignment $(f,t) \mapsto \phi_t^f$, defined in the second part of the statement, is continuous. Moreover, if $W$ is compact, we can choose a minimum $\epsilon > 0$ satisfying inclusion \eqref{eq:inclusionBepsilon} for every $g \in W$.
\end{proof}
Any element $f \in \FF^\tau$ (which is the identity near $\partial D^2$) can be extended to a diffeomorphism of $\R^2$. We will thus regard $\FF^\tau$ as a natural subspace of $\FF^\tau_\R$. Define 
\[
	\LL^\tau \coloneqq \LL^\tau_\R \cap \FF^\tau.
\]
\begin{prop}\label{hLLtauFFtau}
	The inclusion
	\[
		i \colon \LL^\tau_\R \hookrightarrow \FF^\tau_\R
	\]
	and its restriction
	\[
		\LL^\tau\hookrightarrow \FF^\tau
	\]
	are both homotopy equivalences. The homotopy inverse of the above maps can be chosen to be supported inside the disk.
	
	Moreover, for a suitable homotopy inverse $h$ of $i$ (resp.\ of $i|_{\LL^\tau}$), and for every compact subset $K$ of $\FF^\tau_\R$ (resp.\ of $\FF^\tau$), every function $f \in h(K)$ is linear in a uniform neighborhood of the origin.
\end{prop}

\begin{proof}
	For every $f \in \FF^\tau_\R$, define the standard linearizing isotopy of $\R^2$:
	\begin{equation}\label{eq:fttoitsdifferential}
		f_t(x) = 
		\begin{cases}
			f(tx)/t & t \neq 0, \\[0.5ex]
			d_\zero f(x) & t = 0.
		\end{cases}
	\end{equation}
	According to \cref{cutoffisot} applied to $f_{1-t}$, there exists a continuous map
	\[
	\begin{split}
		\FF^\tau_\R \times [0,1] &\to \FF^\tau_\R\\
		(f,t) &\mapsto \phi^f_t
	\end{split}
	\]
	such that $\phi_t^f$ satisfies conditions \eqref{eq:condcutoffisot} (applied to $f_{1-t}$).
	
	In particular, $\phi^f_1$ coincides with the differential $d_\zero f$ in a neighborhood of $\zero$, meaning that $\phi^f_1 \in \LL^\tau_\R$. The map 
	\begin{equation}\label{eq:phif}
		\FF^\tau_\R \to \LL^\tau_\R, \quad f \mapsto \phi^f_1
	\end{equation}
	is then a \emph{right homotopy inverse} for the inclusion $i$.
	
	Conversely, choose $f \in \LL^\tau_\R$. There exists an open set $B_\eta(\zero)$, $\eta \ll 1/2$, such that $f$ is linear when restricted to $B_\eta(\zero)$, and such that
	\[
		\phi^f_t(x) = f_t(x) = f(x), \quad \forall x \in B_\eta(\zero),\quad \forall t \in [0,1].
	\]
	Thus, $\phi^f_t$ belongs to $\LL^\tau_\R$ for every $t \in [0,1]$. Consequently, the map \eqref{eq:phif}, when restricted to $\LL^\tau_\R$, is a \emph{left homotopy inverse} for $i$. The inclusion $i$ is therefore a homotopy equivalence. \vspace{0.3cm}
	
	On the other hand, choose $f \in \FF^\tau$. For every $t \in [0,1]$ and every $x$ belonging to a neighborhood of $\partial D^2$,
	\[
		\phi^f_t(x) = f(x) = x.
	\]
	Thus, $\phi^f_t$ belongs to $\LL^\tau$ at all times, and restricting the map \eqref{eq:phif} to $\FF^\tau$ defines a homotopy inverse for the inclusion $\LL^\tau \hookrightarrow \FF^\tau$.
\end{proof}

Consider the differential map
\[
	d_\zero \colon \LL^\tau \to \GL^+(T_\zero\R^2) = \GL^+_2(\R), \qquad d_\zero(f)=d_\zero f.
\]

\begin{lemma}\label{dzeroloctrivial}
	The map $ d_\zero $ is a locally trivial fibration. Its fiber over the identity is
	\[
		\KK^\tau \coloneqq \{f \in \LL^\tau \mid f=\id \text{ near } \zero\}.
	\]
\end{lemma}

\begin{proof}
	Local sections are constructed as follows. First, let $U$ be a convex open neighborhood of the identity $I \in \GL_2^+(\R)$. By applying \cref{cutoffisot} to the smooth family of linear maps
	\[
		U \times [0,1] \to \GL^+_2(\R), \quad (B,t) \mapsto B_t \coloneqq (1-t)I + tB,
	\]
	we obtain an isotopy $\phi^B_t$ of $D^2$ that depends continuously on $B$ and satisfies conditions \eqref{eq:condcutoffisot}. Since  
	\[
		\phi_t^B(x) = I(x) = x,
	\]
	for every $t \in [0,1]$ and for every $x$ near $\partial D^2$, the map $\phi^B_t$ defines a diffeomorphism of $D^2$. Furthermore,
	\[
		d_\zero \phi^B_1 = B\,
	\]
	so the map 
	\[
		s \colon U \to \LL^\tau, \quad B \mapsto \phi^B_1,
	\]
	defines a continuous local section around $I$.
	
	On the other hand, let $X$ be an element in $\GL^+_2(\R)$. Choose a smooth path $X_t$ connecting $I$ to $X$. The arguments above produce a diffeomorphism $f_X \in \LL^\tau$ that agrees with $X$ in a neighborhood of $\zero$. 
	
	If 
	\[
		U_X \coloneqq \{XB \mid B \in U \} \subseteq \GL^+_2(\R),
	\]
	then
	\[
		s_X \colon U_X \to \LL^\tau, \quad C \mapsto f_X \circ s(X^{-1}C),
	\]
	yields the desired local section around $X$.
\end{proof}

We now study the fiber $\KK^\tau$. Consider the annulus $A$, equipped with the involution $\tau_A$ 
\[
	\tau_A(e^{i\theta},r) = (e^{i(\theta + \pi)}, r).	
\]
Since $\tau_A$ acts freely on $A$, the quotient $A/\ang{\tau_A}$ is diffeomorphic to an annulus. Moreover, every $\tau_A$-equivariant map $g \in \Diff^+(A)$ descends to a map $[g] \in \Diff^+(A/\ang{\tau_A})$.

Define
\[
	\phi \colon \KK^\tau \to \Diff^+_{N\partial}(A/\ang{\tau_A}),
\]
as
\[
	\phi(f) \coloneqq [\sigma \circ f \circ \sigma^{-1}]
\]
where $\sigma$ is the smooth embedding
\[
	\sigma \colon D^2 \setminus \{0\} \to A, \quad \rho\,e^{i\theta} \mapsto (e^{i\theta}, \rho).
\]
Observe that since $f \in \KK^\tau$ is the identity near the origin, the composition $\sigma \circ f \circ \sigma^{-1}$ admits a unique extension to a smooth equivariant diffeomorphism of $A$. The map $\phi$ is then well-defined.

Analogously, since $A \to A/\ang{\tau_A}$ is a double covering, for every map $f \in \Diff^+_{N\partial}(A/\ang{\tau_A})$ there exists a unique lift to an equivariant map $\tilde{f} \in \Diff^+(A)$ which is the identity in a neighborhood of the internal boundary.
\begin{lemma}\label{hKtaudiffAtauA}
	The map $\phi$ is a homeomorphism onto the connected components of $\Diff^+_{N\partial}(A/\ang{\tau_A})$ generated by the even powers of a Dehn twist. In particular, by \cref{diffgroupannulus}, the group $\KK^\tau$ has weakly contractible components, and $\pi_0(\KK^\tau) \cong \Z$ is generated by the lift of a double Dehn twist on $A/\ang{\tau_A}$.
\end{lemma}
\begin{proof}	
	Let $E \subseteq \Diff^+_{N\partial}(A/\ang{\tau_A})$ be the union of connected components corresponding to even powers of a Dehn twist. Then, the map
	\[
	\psi \colon E \to \KK^\tau, \quad f \mapsto \sigma^{-1} \circ \tilde{f} \circ \sigma
	\]
	is a continuous left inverse for $\phi$. 
\end{proof}

\begin{proof}[Proof of \cref{eqthmdisk}]
	According to \cref{dzeroloctrivial} the differential operator $d_\zero$ yields a Serre fibration
	\[
		\KK^\tau \hookrightarrow  \LL^\tau \xlongrightarrow{d_\zero} \GL^+_2(\R),
	\]
	where 
	\begin{equation}\label{eq:hGL}
		\GL^+_2(\R) \simeq_\h \SO(2) \simeq S^1.
	\end{equation}
	In particular, we can analyze the following portion of the homotopy sequence \eqref{eq:exactsequenceSerre}:
	\[
		\pi_1(\GL^+_2(\R),I) \overset{p}{\to} \pi_0(\KK^\tau,\id_{D^2}) \to \pi_0(\LL^\tau,\id_{D^2}) \to 0.
	\]
	The map $p$ is defined as follows: for each class $[\gamma] \in \pi_1(\GL_2^+(\R))$, $p[\gamma] = \tilde{\gamma}_1$, where $\tilde{\gamma}$ is the unique lift of $\gamma$ to $\LL^\tau$ starting at the identity (cf.\ \cite[Thm.\ 4.41]{hatcherAlgebraicTopology2002}).\vspace{0.2cm}
	
	We claim that the map $p$ is an isomorphism of groups. Indeed, consider a generator of $\pi_1(\GL_2^+(\R)) \cong \Z$ represented by the loop
	\[
		A_t \coloneqq R_{2\pi t}, \quad 0\leq t\leq 1,
	\]
	where $R_{2\pi t}$ is rotation by angle $2\pi t$. We can lift the loop $A_t \in \GL_2^+(\R)$ to a path $F_t$ in $\LL^\tau$ defined by
	\[
	F_t(re^{i\theta}) \coloneqq re^{i(\theta+2\pi t\rho(r))},
	\]
	where $\rho\colon [0,1] \to [0,1]$ is a smooth cutoff function satisfying $\rho(r) = 1$ near $r = 0$ and $\rho(r) = 0$ near $\partial D^2$. The endpoint $p([A_t]) = [F_1]$ descends to a double Dehn twist of $A/\ang{\tau_A}$. Consequently, $p([A_t])$ is a generator of $\pi_0(\KK^\tau)$. Since $p$ is an isomorphism, the group $\pi_0(\LL^\tau)$ is trivial.\vspace{0.2cm}
	
	On the other hand, consider the exact subsequence
	\[
		\pi_1(\KK^\tau) \to \pi_1(\LL^\tau,\id_{D^2}) \to \pi_1(\GL^+_2(\R),I) \overset{p}{\to} \pi_0(\KK^\tau,\id_{D^2}).
	\]
	We can conclude that $\pi_1(\LL^\tau) = 0$, since $p$ is an isomorphism and by \cref{hKtaudiffAtauA} $\KK^\tau$ has contractible components.\vspace{0.2cm}
	
	Finally, by the homotopy equivalence \eqref{eq:hGL} and \cref{hKtaudiffAtauA}, the external groups in the exact subsequence
	\[
		\pi_j(\KK^\tau,\id_{D^2}) \to \pi_j(\LL^\tau,\id_{D^2}) \to \pi_j(\GL_2^+(\R), I)
	\]
	are trivial for $j \ge 2$. Consequently, $\pi_j(\LL^\tau) = 0$ for all $j \ge 2$.\vspace{0.2cm}
	
	We have proven that $\LL^\tau$ has vanishing homotopy groups, so $\LL^\tau$ is weakly contractible. By \cref{hLLtauFFtau}, $\FF^\tau$ is also weakly contractible.
\end{proof}

\subsection{Proof of the Equivariant Smale Theorem}
Set
\[
\begin{split}
	Q^u_\fix &\coloneqq \{f\in Q^u \mid f(n)=n,\ f(s)=s\},\\
	Q^u_\swap &\coloneqq \{f\in Q^u \mid f(n)=s,\ f(s)=n\}.
\end{split}
\]
We shall first prove that
\[
	Q^u_\fix \simeq_\wh \SO(2).
\]
Define
\[
	Q^u_0 \coloneqq \{f\in Q^u_\fix \mid d_sf =\id_{T_sS^2}\}.
\]

For the next proof, a compact subset $K \subseteq X$ of a topological space $X$ is \emph{$S$-compact} if $K$ is the image of a continuous map from some $n$-sphere to $X$.
\begin{prop}\label{hQu0}
	The space $Q^u_0$ is weakly contractible.
\end{prop}
\begin{proof}	
	We shall prove that for every $S$-compact subset $K \subseteq Q^u_0$, there exists a continuous contraction $K \times [0,1] \to Q^u_0$ which joins the inclusion map to the constant one.
	
	Let us consider the stereographic projection from $n$
	\[
		q_n \colon S^2 \setminus\{n\}\to \R^2,
	\]
	sending $s$ to $\zero$. With this choice, the rotation $u$ is identified with the map $\tau$ defined earlier. For every $f \in Q^u_0$, let $\tilde{f}$ be the local representation of $f$ through $q_n$. Since $d_0\tilde{f} = \id_{\R^2}$ for every $f \in Q^u_0$, by \cref{hLLtauFFtau} there is a homotopy from the space
	\[
		\{\tilde{f} \mid f \in K\} \subseteq \FF^\tau_\R
	\]
	to the subspace of $\LL^\tau_\R$ consisting of maps which are the identity on a uniform open ball $B_\epsilon(\zero)$. By representing such a homotopy via $q_n$, we obtain a homotopy from $K$ to the subspace $\Omega_0 \subseteq Q^u_0$ of diffeomorphisms that agree with the identity on a uniform neighborhood $U$ of $s$.
	
	The result then follows once we prove $\Omega_0$ is weakly contractible. Fix an $S$-compact subspace $K' \subseteq \Omega_0$. By the stereographic projection $q_s$ from $s$, we can represent the space $K'$ as an $S$-compact subset $K''$ of 
	\[
		\FF^\tau_{\R,R} \coloneqq  \{g \in \FF^\tau_\R \mid g \text{ has support contained in } B_R(\zero)\}
	\]
	for some large $R > 0$. By a radial rescaling of $\R^2$, we can suppose that $R= 1/2$, so that $\FF^\tau_{\R,R}$ is included in $\FF^\tau$. According to \cref{eqthmdisk}, the space $\FF^\tau$ is weakly contractible. Therefore, there exists a continuous map 
	\[
		K'' \times [0,1] \to \FF^\tau
	\] 
	connecting the inclusion with the constant map to the identity. Since the maps in $\FF^\tau$ are the identity near $\partial D^2$, they can be naturally extended to be the identity outside $D^2$. Thus, the above homotopy can be read, via the map $q_s$, as a homotopy
	\[
		K' \times [0,1] \to Q^u_0,
	\] 
	connecting the inclusion with the constant map to the identity.
\end{proof}
\begin{rmk}\label{rmkderivativeatsfibration}
	As in \cref{dzeroloctrivial}, the derivative map $d_s$ from $Q^u_\fix$ mapping to $\GL^+(T_sS^2) = \GL^+_2(\R)$ is a locally trivial fibration. In this scenario, the result is actually immediate to prove: any transformation $A \in \GL^+(T_sS^2)$ canonically defines a diffeomorphism of $S^2$ as follows:
	\[
		(x,y,z) \mapsto \frac{(A(x,y),z)}{\sqrt{\abs{A(x,y)}^2 + z^2}},
	\]
	which is a $u$-equivariant diffeomorphism of $S^2$ fixing $n$ and $s$ with differential $A$ at $s$.
\end{rmk}
\begin{proof}[Proof of \cref{eqSmale}]
Consider the derivative fibration
\[
	Q^u_0 \hookrightarrow Q^u_\fix \xrightarrow{d_s} \GL^+(T_sS^2) \simeq \GL^+_2(\R).
\]
By \cref{hQu0}, the fiber is weakly contractible. Thus, the map $d_s$ is a weak homotopy equivalence between $Q^u_\fix$ and $\GL^+_2(\R) \simeq_\h \SO(2)$. Since the above argument is symmetric in $n$ and $s$, the derivative fibration defined using $d_n$ is a weak homotopy equivalence too.

Let
\[
	J(x,y,z)=(x,-y,-z)
\]
be the rotation by angle $\pi$ about the $x$-axis. Then
\[
	J \circ u = u \circ J,
\]
and $J$ exchanges $n$ and $s$. Since composition with $J$ defines a homeomorphism
\[
	Q^u_\fix \longrightarrow Q^u_\swap,
\]
the space $Q^u_\swap$ is also weakly homotopy equivalent to $\SO(2)$.
\end{proof}
\section{Equivariant Tomography Theory}\label{sec:tomography}
Classically, tomography theory in the sense of Giroux \cite{girouxStructuresContactDimension2000} consists of studying contact structures on a 3-manifold $M = S \times [-1,1]$ via the characteristic foliation defined on the slices $S_z \coloneqq S \times \{z\}$ as $z$ varies along $[-1,1]$. We will specifically examine the case $S \coloneqq S^2$, which we think of as canonically embedded in $\R^3$. See \cite[\S 4.9]{geigesIntroductionContactTopology2008} for an exposition of this theory in the classical setting. Throughout this section, the manifold $M$ will be equipped with a vertically invariant metric $\gk \coloneqq \gk_S + dz^2$, where $\gk_S$ is a $u_{\std}$-invariant metric on $S^2$.

We will prove the following theorem. Recall that a contact structure $\xi$ on $M$ is anti-invariant if $du_\times(\xi) = -\xi$ as oriented 2-plane fields, where $u_\times(s,z) \coloneqq (u_{\std}(s),z)$. With a slight abuse of notation, anti-invariant tight contact structures on $M$ are understood to be defined on a slightly larger collar $S \times [-1-\epsilon,1+\epsilon]$ for some $0 < \epsilon \ll 1$. In particular, we will apply \cref{Girouxbound} to the boundary spheres.
\begin{thm}\label{mainthmtomographyStwo}
	Let $X\in \XX(\partial M)$ be an anti-invariant vector field. Up to equivariant isotopy relative to $\partial M$, the set of anti-invariant tight contact structures on $(M,u_\times)$ inducing characteristic foliation $\ang{X}$, if non-empty, is a $\Z$-torsor. Moreover, if two such structures $\xi_1$ and $\xi_2$ are equivalent, and they coincide near $\partial M$, the connecting isotopy can be taken to be the identity near $\partial M$. 
\end{thm}
For the rest of the section, set $u \coloneqq u_\std$. 
\subsection{Definition of a Rotational Torsor}\label{sec:defrotdif}
The aim of this subsection is to define the integer appearing in the statement of \cref{mainthmtomographyStwo}. Let $\{n,s\}$ be the north and south poles of $S^2$, and let $F = \{p\} \times [-1,1]$, for $p \in \{n,s\}$, be a connected component of $\Fix(u_\times)$. 
\begin{lemma}\label{antiinvarianttangenttoF}
	Every anti-invariant contact structure $\xi$ is tangent to $F$.
\end{lemma}
\begin{proof}
	For every point $q \in F$, the restriction $d_qu_\times|_{\xi_q}$ is an orientation-reversing involution. Thus, $d_qu_\times|_{\xi_q}$ is conjugate to the matrix $[-1] \oplus [1]$. In particular, the 1-eigenspace of $d_qu_\times|_{\xi_q}$ coincides with the unique $1$-eigenspace of $d_qu_\times$ on $T_qM$, which is precisely the vertical line $T_qF$.
\end{proof}
Consequently, the oriented characteristic foliation $\xi \cap T_{(p,z)}S_z$, for $z \in [-1,1]$ defines a 1-dimensional oriented line inside $T_{(p,z)}S_z$, which is canonically identified with $T_pS^2$. Such oriented line can be identified with an element $\ell_{z,\xi,p} \in S^1$.

Now fix two anti-invariant contact structures $\xi_1$, $\xi_2$ that induce the same characteristic foliations on the boundary slices. Observe that for every point $(p,z) \in \Fix(u_\times) \cap S_{\pm1}$, the planes $\xi_1(p,z)$, $\xi_2(p,z)$ coincide, since they are tangent to both $TF$ and $TS_z$. In particular, the paths $\ell_{z,\xi_1,p}$ and $\ell_{z,\xi_2,p}$ share the same endpoints. Thus, the concatenation $\ell_{\bullet,\xi_2,p}^{\ } * \ell_{\bullet,\xi_1,p}^{-1}$ is a well-defined map from $S^1$ to itself.
\begin{defin}\label{rotdif}
	The \emph{$F$-rotational difference} is the $\Z$-degree
	\[
		d_F(\xi_2, \xi_1) \coloneqq \deg_\Z(\ell_{\bullet,\xi_2,p}^{\ } * \ell_{\bullet,\xi_1,p}^{-1})
	\]
	(see \cref{fig:Frotational}).
\end{defin}
\begin{lemma}\label{rotdiffinvaiso}
	The integer $d_F(\xi_2, \xi_1)$ is invariant under the application of an equivariant isotopy rel boundary of $M$ to $\xi_2$ or $\xi_1$.
\end{lemma}
\begin{proof}
	Since any such isotopy $h_t$ is equivariant, it preserves $F$ setwise. In particular, $h_t$ restricts to a homotopy of hyperplane distributions tangent to $F$. Thus, for any pair $\xi_1$, $\xi_2$ of contact structures as above, the concatenations
	\[
		\ell_{\bullet,\xi_2,p} * \ell_{\bullet,\xi_1,p}^{-1},\text{ and }  \ell_{\bullet,dh_1(\xi_2),p} * \ell_{\bullet,\xi_1,p}^{-1}
	\]
	define homotopic maps from $S^1$ to itself. Hence these maps have the same degree.
\end{proof}
\begin{figure}[t]
	\centering
	\def\svgwidth{0.3\linewidth}
	%% Creator: Inkscape 1.4.3 (0d15f75, 2025-12-25), www.inkscape.org
%% Accompanies image file 'FixedLineOnSameLine_svg-tex.pdf' (pdf, eps, ps)
\begingroup%
  \makeatletter%
  \providecommand\color[2][]{%
    \errmessage{(Inkscape) Color is used for the text in Inkscape, but the package 'color.sty' is not loaded}%
    \renewcommand\color[2][]{}%
  }%
  \providecommand\transparent[1]{%
    \errmessage{(Inkscape) Transparency is used (non-zero) for the text in Inkscape, but the package 'transparent.sty' is not loaded}%
    \renewcommand\transparent[1]{}%
  }%
  \providecommand\rotatebox[2]{#2}%
  \newcommand*\fsize{\dimexpr\f@size pt\relax}%
  \newcommand*\lineheight[1]{\fontsize{\fsize}{#1\fsize}\selectfont}%
  \ifx\svgwidth\undefined%
    \setlength{\unitlength}{175.79054945bp}%
    \ifx\svgscale\undefined%
      \relax%
    \else%
      \setlength{\unitlength}{\unitlength * \real{\svgscale}}%
    \fi%
  \else%
    \setlength{\unitlength}{\svgwidth}%
  \fi%
  \global\let\svgwidth\undefined%
  \global\let\svgscale\undefined%
  \makeatother%
  \begin{picture}(1,2.18320069)%
    \lineheight{1}%
    \setlength\tabcolsep{0pt}%
    \put(0.68589644,0.97507607){\color[rgb]{1,0.27058824,0}\makebox(0,0)[lt]{\lineheight{1.25}\smash{\begin{tabular}[t]{l}$\xi_2$\end{tabular}}}}%
    \put(0.92464944,0.18207447){\color[rgb]{0,0,0}\makebox(0,0)[lt]{\lineheight{1.25}\smash{\begin{tabular}[t]{l}$S_{-1}$\end{tabular}}}}%
    \put(0,0){\includegraphics[width=\unitlength,page=1]{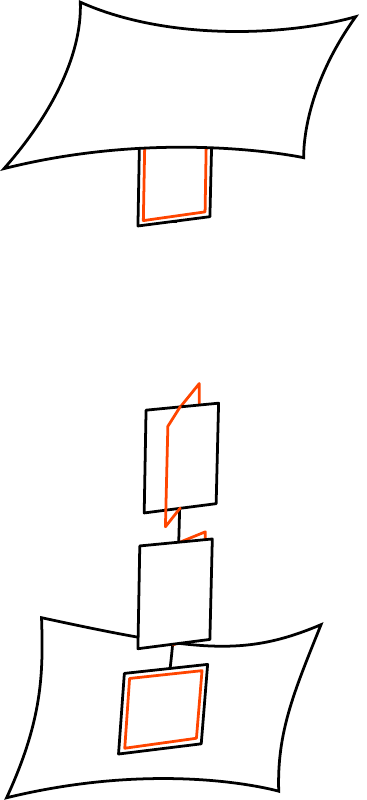}}%
    \put(0.98689664,1.92105408){\color[rgb]{0,0,0}\makebox(0,0)[lt]{\lineheight{1.25}\smash{\begin{tabular}[t]{l}$S_1$\end{tabular}}}}%
    \put(0.31835636,0.94741624){\color[rgb]{0,0,0}\makebox(0,0)[rt]{\lineheight{1.25}\smash{\begin{tabular}[t]{r}$\xi_1$\end{tabular}}}}%
    \put(0,0){\includegraphics[width=\unitlength,page=2]{FixedLineOnSameLine_svg-tex.pdf}}%
    \put(0.63194654,1.56209632){\color[rgb]{0,0,0}\makebox(0,0)[lt]{\lineheight{1.25}\smash{\begin{tabular}[t]{l}$F$\end{tabular}}}}%
    \put(0,0){\includegraphics[width=\unitlength,page=3]{FixedLineOnSameLine_svg-tex.pdf}}%
  \end{picture}%
\endgroup%

	\caption{The $F$-rotational difference between $\xi_1$ and $\xi_2$.}
	\label{fig:Frotational}
\end{figure}
We will redefine $d_F(\xi_2, \xi_1)$ using dividing curves.

Let $r$ be the standard reflection of $S^1$ about the $x$-axis. For future reference, let $S^1_{\half}$ denote the upper semicircle. Moreover, equip $S^2$ with $u_{\std}$. 
\begin{defin}
	The space of \emph{equivariant unparameterized embedded loops} is the quotient
	\begin{equation}\label{eq:loopeq}
		\Loop^{\eq}(S^2) \coloneqq \Emb^{\eq}(S^1,S^2)/\Diff^{+,\eq}(S^1)\,,
	\end{equation}
	of equivariant embeddings under the right-action of equivariant orientation-preserving diffeomorphisms of $S^1$.
\end{defin}

We claim that the space $\Loop^{\eq}(S^2)$ is weakly homotopy equivalent to $\SO(2) \simeq S^1$. First, we need to analyze the quotient map 
\[
	q \colon \Emb^{\eq}(S^1,S^2) \to \Loop^{\eq}(S^2)\,.
\]
Recall by \cref{sec:eqSmale} that
\[
	Q_\fix^u = \{f\in \Diff^+(S^2) \mid u \circ f = f \circ u,\  f(n)=n,\ f(s)=s\}.
\]         
\begin{prop}\label{quotientloopisSerre}
	Let $S^1$ be canonically embedded in $S^2$ as a great circle passing through $n$ and $s$, and let $\Emb^{\eq}_\fix(S^1,S^2)$ be the set of embeddings fixing both $n$ and $s$. Then the restriction of $q$ to $\Emb^{\eq}_\fix(S^1,S^2)$ is a locally trivial principal $\Diff^{+,\eq}_\fix(S^1)$-bundle, where
	\[
		\Diff^{+,\eq}_\fix(S^1) = \{f \in \Diff^{+,\eq}(S^1) \mid f(\pm 1,0) = (\pm1,0)\}.
	\]
\end{prop} 
\begin{proof}
	Observe, first, that any embedding $e \in \Emb^\eq_\fix(S^1,S^2)$ maps $\Fix(r)  = \{(\pm1,0)\}$ onto $\{n,s\}$. Since $-\id_{S^1}$ belongs to $\Diff^{+,\eq}(S^1)$, one of the embeddings $e$, $e \circ (-\id_{S^1})$ belongs to $\Emb^{\eq}_\fix(S^1,S^2)$. In particular, $\Loop^{\eq}(S^2)$ is canonically homeomorphic to the quotient $\Emb^\eq_\fix(S^1,S^2)/\Diff^{+,\eq}_\fix(S^1)$.
	
	Let $[i]$ be a class in $\Loop^{\eq}(S^2)$, with $i \in \Emb^{\eq}_\fix(S^1,S^2)$. We aim to prove that $q$ admits a continuous local section defined in a neighborhood $U \ni [i]$ of the quotient space. Equipping $S^2$ with the metric $\gk_S$, the assertion is followd from direct adaptation of Ernst and Fischer's classical construction \cite[pp.\ 310-317]{ernstManifoldEmbeddingsClosed1981}. Indeed, since $\gk_S$ is invariant, their construction is automatically equivariant.
\end{proof}

We next analyze the space $\Diff^+_{\{a,b\}}(D^2)$ of orientation-preserving diffeomorphisms of the disk fixing two specified points $a,b \in \partial D^2$. We will need an equivariant version of \cref{palaisfibration}.
\begin{prop}\label{Palaiseq}
	Let $W$ be a compact manifold equipped with a smooth involution $\rho$, and let $Z \subseteq Y \subseteq W$ be a chain of $\rho$-invariant submanifolds. Moreover, let $\Diff^{\eq}_Z(W)$ be the space of $\rho$-equivariant diffeomorphisms of $W$ fixing $Z$ pointwise, and let $\Emb^{\eq}_Z(Y,W)$ be the space of smooth $\rho$-equivariant embeddings of $Y$ into $W$ fixing $Z$ pointwise. For every $e_0 \in \Emb^{\eq}_Z(Y,W)$, the evaluation map
	\[
	\begin{split}
		\Diff^{\eq}_Z(W) &\to \Emb^{\eq}_Z(Y,W),\\
		f &\mapsto f \circ e_0,
	\end{split}
	\]
	is a Serre fibration. The result continues to hold if $Y = \partial W$. 
	
	Moreover, for every smooth isotopy $f_t \colon Y \to W$ of embeddings, consisting of elements in $\Emb^\eq_Z(Y,W)$, there exists a smooth isotopy of equivariant diffeomorphisms $F_t \colon W \to W$ such that $F_t \circ e_0 = f_t$ and $F_t \in \Diff_Z^\eq(W)$.
\end{prop}
\begin{proof}
	The statement is a direct adaptation of Palais' classical results \cite[Thm.\ B]{PalaisLocaltrivialityoftherestrictionmap} in the equivariant setting. Fixing a $\rho$-invariant metric $\gk_W$ on $W$, we consider the exponential map
	\[
	\exp \colon TW \to W\,,
	\] 
	and the associated parallel transport map
	\[
		P \colon TW \oplus TW \to TW\,,
	\]
	sending $(v,w) \in T_pW \times T_pW$ to the parallel transport of $v$ along the geodesics $\gamma_w(t) \coloneqq \exp_p(tw)$. Since $\gk_W$ is $\rho$-invariant, the exponential map $\exp$ is $(d\rho, \rho)$-equivariant. Analogously, the map $P$ is $(d\rho \times d\rho, \rho)$-equivariant. 
	
	Palais' local cross-section construction then carries over directly to the space of equivariant vector fields vanishing on $Z$, yielding local trivializing sections for the evaluation map.
\end{proof}

We can now study the homotopy groups of $\Loop^{\eq}(S^2)$.
\begin{figure}[t]
	\centering
	\subcaptionbox{\label{fig:bigonI}}[0.8\textwidth]{
		\def\svgwidth{0.8\linewidth}
		%% Creator: Inkscape 1.4.3 (0d15f75, 2025-12-25), www.inkscape.org
%% Accompanies image file 'BigonOperationI_svg-tex.pdf' (pdf, eps, ps)
\begingroup%
  \makeatletter%
  \providecommand\color[2][]{%
    \errmessage{(Inkscape) Color is used for the text in Inkscape, but the package 'color.sty' is not loaded}%
    \renewcommand\color[2][]{}%
  }%
  \providecommand\transparent[1]{%
    \errmessage{(Inkscape) Transparency is used (non-zero) for the text in Inkscape, but the package 'transparent.sty' is not loaded}%
    \renewcommand\transparent[1]{}%
  }%
  \providecommand\rotatebox[2]{#2}%
  \newcommand*\fsize{\dimexpr\f@size pt\relax}%
  \newcommand*\lineheight[1]{\fontsize{\fsize}{#1\fsize}\selectfont}%
  \ifx\svgwidth\undefined%
    \setlength{\unitlength}{805.54301104bp}%
    \ifx\svgscale\undefined%
      \relax%
    \else%
      \setlength{\unitlength}{\unitlength * \real{\svgscale}}%
    \fi%
  \else%
    \setlength{\unitlength}{\svgwidth}%
  \fi%
  \global\let\svgwidth\undefined%
  \global\let\svgscale\undefined%
  \makeatother%
  \begin{picture}(1,0.18909077)%
    \lineheight{1}%
    \setlength\tabcolsep{0pt}%
    \put(0,0){\includegraphics[width=\unitlength,page=1]{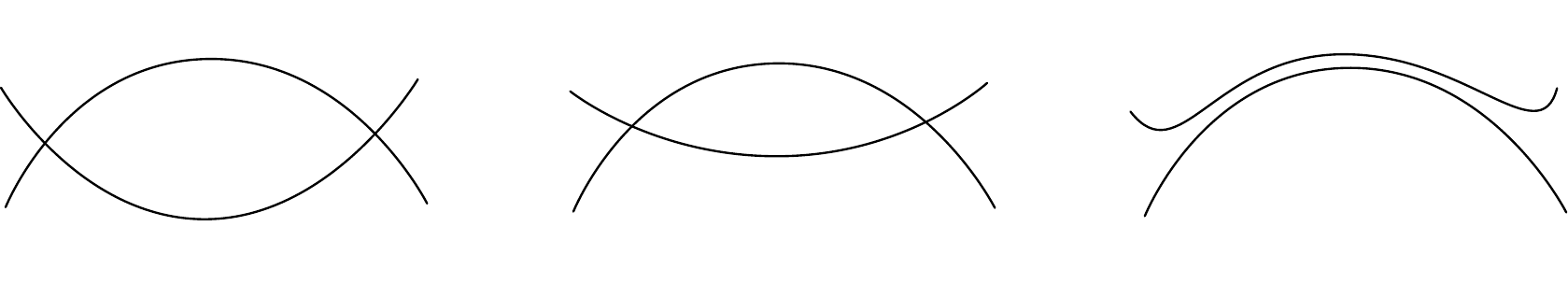}}%
    \put(0.11789383,0.17978934){\color[rgb]{0,0,0}\makebox(0,0)[lt]{\lineheight{1.25}\smash{\begin{tabular}[t]{l}$e_0$\end{tabular}}}}%
    \put(0.12533001,0.00267318){\color[rgb]{0,0,0}\makebox(0,0)[lt]{\lineheight{1.25}\smash{\begin{tabular}[t]{l}$e_1$\end{tabular}}}}%
    \put(0,0){\includegraphics[width=\unitlength,page=2]{BigonOperationI_svg-tex.pdf}}%
  \end{picture}%
\endgroup%
}\\
	\subcaptionbox{\label{fig:bigonII}}[0.8\textwidth]{
		\def\svgwidth{0.8\linewidth}
		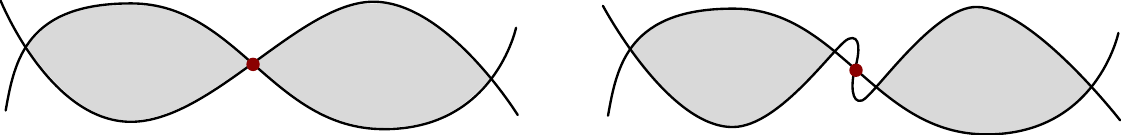}\\[1ex]
	\subcaptionbox{\label{fig:bigonIII}}[0.7\textwidth]{
		\def\svgwidth{0.7\linewidth}
		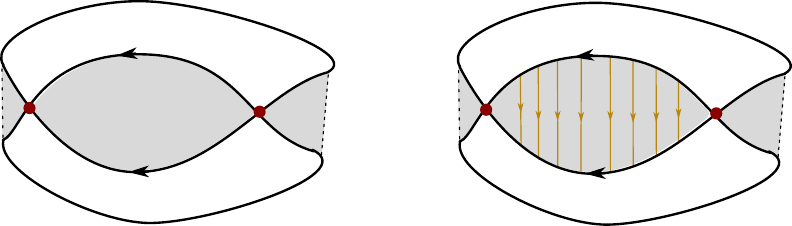}
	\caption{Operations on bigons.}
	\label{fig:bigon}
\end{figure}
\begin{lemma}\label{loopspaceconnected}
	The space $\Loop^{\eq}(S^2)$ is connected.
\end{lemma}
\begin{proof}
	The following proof is an equivariant adaptation of the classical Baer theorem (see \cite[Prop.\ 1.7]{farbPrimerMappingClass2011} and \cite[Thm.\ 6.2.5]{stillwellClassicalTopologyCombinatorial1993}).
	
	Let $[e_0]$, $[e_1]$ be two classes in $\Loop^{\eq}(S^2)$. Up to performing local rotations around the north and south poles, we can suppose that $e_1$ is transverse to $e_0$ in an open neighborhood of $\{n,s\}$. By \cref{Gtrasv}, applied with $P = S^2$, $W=S^1$, $N = e_1(S^1)$, and $C = e_0^{-1}(\{n,s\})$, we can further suppose that $e_0$, $e_1$ are transverse maps.
	
	According to \cite[Prop.\ 1.1.7]{farbPrimerMappingClass2011}, the curves $e_0$, $e_1$ cut off an \emph{innermost bigon} (cf.\ the left-hand side of \cref{fig:bigonI}). Let $B$ be the closed area delimited by this bigon, and let $p$, $q$ be its corner points. Moreover, let $\ell_0$, $\ell_1$ be the sub-arcs of $e_0$, $e_1$ bounding $B$.
	
	Suppose $\{p,q\} \cap \{n,s\} = \eset$. We claim that $u(p) \neq q$. Indeed, since $u|_{e_i}$, $i = 0,1$, is orientation-reversing, the condition $u(p) = q$ would imply $u(\ell_i) = \ell_i$ for $i = 0,1$. In particular, both $\ell_0$, $\ell_1$ would contain fixed points of $u$. This is a contradiction, since $n$ and $s$ are common intersection points of $e_0$ and $e_1$. 
	
	As a consequence, $u(B)$ is disjoint from $B$. Thus, the classical isotopy shown in \cref{fig:bigonI} can be equivariantly performed simultaneously on $B$ and $u(B)$, allowing us to modify $e_0$, $e_1$ to have four fewer points of intersection.
	
	On the other hand, suppose that the intersection $\{p,q\} \cap \{n,s\}$ is a singleton. By applying the local rotation shown in \cref{fig:bigonII}, we return to the previous case, modifying $e_0$, $e_1$ to have two fewer points of intersection.
	
	Finally, the procedure reduces to the case in which every bigon $B$ satisfies $\{p,q\} = \{n,s\}$. In this case, the configuration is forced to be the one in \cref{fig:bigonIII}. In particular, up to considering a different pair of bigons, we may suppose that $e_0$ and $e_1$ have the orientations as shown. Thus, if $h_t \colon \ell_0 \to \ell_1$ is the classical isotopy shown in golden, then
	\[
		\tilde{h}_t \colon S^1 \to S^2, \quad x \mapsto 
		\begin{cases}
			(h_t \circ e_0)(x) & \text{if } x \in S^1_\half,\\
			(u \circ h_t \circ e_0 \circ r)(x) & \text{if } x \in r(S^1_\half),		
		\end{cases}
	\]
	defines a $C^1$-smooth family of equivariant embeddings, connecting $e_0$ with a loop $h_1$ in the same class as $e_1$. Indeed, for every $x \in \Fix(r)$ and $v \in T_xS^1$
	\[
		(du \circ dh_t \circ de_0 \circ dr)(v) = (-dh_t \circ de_0 \circ -\id)(v) = (dh_t \circ de_0)(v).
	\]
	By \cite[Thm.\ VI.4.2]{bredonIntroductionCompactTransformation1972}, we can approximate $\tilde{h}$ relative to $S^1 \times \{0,1\}$ by a smooth family of equivariant embeddings connecting $e_0$ and $e_1$. The composition $q \circ \tilde{h}_t$ then defines a continuous path from $[e_0]$ to $[e_1]$ in $\Loop^{\eq}(S^2)$.
\end{proof}

For future use, let us denote by $G_t \colon \GL_2^+(\R) \to \GL^+_2(\R)$ the classical Gram-Schmidt strong deformation retraction onto $\SO(2)$, defined as
\[
	G_t(A) = (1-t)G(A) + tA,\quad A \in \GL_2^+(\R),
\]
where 
\[
\begin{split}
	G \colon \GL_2^+(\R) &\to \SO(2),\\
	A &\mapsto 
	\begin{cases}
		e_1 \mapsto u_1/\abs{u_1}, \\
		e_2 \mapsto (u_2 - \proj_{u_1}(u_2))/\abs{(u_2 - \proj_{u_1}(u_2))},
	\end{cases}
\end{split}
\]
and $u_i \coloneqq A(e_i)$. 
\begin{lemma}\label{fixgreatcirclecontractible}
	Let $\gamma$ be the great circle of $S^2$ passing through the north and south poles and contained in the $xz$-plane. The subspace $Q^u_\gamma \subseteq Q^u_\fix$ of equivariant diffeomorphisms of $S^2$ fixing $\gamma$ orientation-wise is weakly contractible.
\end{lemma}
\begin{proof}
	By \cref{Palaiseq}, the restriction map 
	\[
		Q^u_\gamma \to \Diff^{+,\eq}_{\{n,s\}}(\gamma)
	\]
	is a Serre fibration. The base space can naturally be identified with the space $\Diff^+_\partial([0,1])$ of orientation-preserving diffeomorphisms of the interval, which is contractible. Thus, the space $Q^u_\gamma$ is weakly homotopy equivalent to the subspace $Q^u_{\gamma,\fix}$ consisting of diffeomorphisms which fix $\gamma$ pointwise.\vspace{0.2cm}
	
	We shall now prove that $Q_{\gamma,\fix}^u$ is weakly contractible. Let $Q_{N\gamma,\fix}^u \subseteq Q_{\gamma,\fix}^u$ be the subgroup of equivariant diffeomorphisms that are the identity in a neighborhood of $\gamma$, and let $E \subseteq S^2$ be the hemisphere $\{(x,y,z) \in S^2 \mid y \ge 0\}$. By the same straightening arguments as in \cite[Prop.\ 1.2--3]{igusaStabilityTheoremSmooth1988} applied to a $u$-invariant tubular neighborhood of $\gamma$, the inclusion $Q_{N\gamma,\fix}^u \hookrightarrow Q_{\gamma,\fix}^u$ is a weak homotopy equivalence. Since every element of $\Diff_{N\partial}^+(E)$ is the identity near $\partial E = \gamma$, the restriction map
	\[
		r \colon  Q_{N\gamma,\fix}^u \longrightarrow \Diff_{N\partial}^+(E), \quad f \longmapsto f|_{E},
	\]
	is a homeomorphism with continuous inverse given by extending any $g \in \Diff_{N\partial}^+(E)$ via $(u \circ g \circ u)$ on $u(E)$. Because $E$ is diffeomorphic to $D^2$, the space $\Diff_{N\partial}^+(E)$ is weakly contractible by \cite{smaleDiffeomorphisms2Sphere1959}. Consequently, $Q_{N\gamma,\fix}^u$ --- and hence also $Q_{\gamma,\fix}^u$ and $Q_{\gamma}^u$ --- is weakly contractible.   
\end{proof}
\begin{prop}\label{Loopweakeqdiffspace}
	Let us equivariantly embed $S^1$ in $S^2$ as a loop $[\gamma] \in \Loop^{\eq}(S^2)$. The space $\Loop^{\eq}(S^2)$ is weakly homotopy equivalent to the space $Q_\fix^u$ via the Serre fibration 
	\[
		\psi \colon Q^u_\fix \to \Loop^{\eq}(S^2)\,,\quad f \mapsto [f \circ \gamma]\,.
	\]
	Thus, $\Loop^{\eq}(S^2)$ is weakly homotopy equivalent to $\SO(2)$ by \cref{eqSmale}. Moreover, for every $p \in \{n,s\}$ the evaluation map
	\[
		\psi_p\colon \Loop^{\eq}(S^2) \to S^1 \subseteq T_pS^2\,,
	\]
	which assigns to each loop class $[e]$ the normalized derivative direction\linebreak $e'(p)/\abs{e'(p)}$, is a weak homotopy equivalence (the space $T_pS^2$ is canonically identified with the $xy$-plane).
\end{prop}
\begin{proof}
	By \cref{loopspaceconnected}, the map $\psi$ is $\pi_0$-surjective, since the codomain $\Loop^{\eq}(S^2)$ is connected. Moreover, $\psi$ is a Serre fibration, as it is the composition of the restriction map
	\[
	\begin{split}
		Q^u_\fix &\to \Emb^{\eq}_\fix(S^1,S^2),\\
		f &\mapsto f \circ \gamma,
	\end{split}
	\] 
	and the quotient map
	\[
	\begin{split}
		\Emb^{\eq}_\fix(S^1,S^2) &\to \Loop^{\eq}(S^2),\\
		e &\mapsto [e],
	\end{split}
	\]
	both of which are Serre fibrations by \cref{Palaiseq,quotientloopisSerre}. The fiber of $\psi$ over $[\gamma]$ is the subgroup $Q^u_{\gamma} \subseteq Q^u_\fix$ of diffeomorphisms fixing the image of $\gamma$ orientation-wise, which is weakly contractible by \cref{fixgreatcirclecontractible}. Therefore, from the long exact sequence \eqref{eq:exactsequenceSerre}, the spaces $Q^u_\fix$ and $\Loop^{\eq}(S^2)$ are weakly homotopy equivalent.
	
	Since $\Loop^{\eq}(S^2)$ is connected, in order to analyze the map $\psi_p$ we can suppose that $S^1$ is embedded via $\gamma$as the unique great circle contained in the $xz$-plane. By \cref{eqSmale}, for every $p \in \{n,s\}$ the top row of the diagram
	\begin{center}
	\begin{tikzcd}
		Q^u_\fix & \GL^+(T_pS^2) \simeq \GL_2^+(\R) & \SO(2)\\
		\Loop^{\eq}(S^2)&&S^1
		\arrow[from=1-1,to=1-2, "d_p"]
		\arrow[from=1-2,to=1-3,"G"]
		\arrow[from=1-1,to=2-1,"\psi"']
		\arrow[from=1-3,to=2-3,"\mathrm{eval}_{e_1}"]
		\arrow[from=2-1,to=2-3,"\psi_p"']
	\end{tikzcd},
	\end{center}
	where $\mathrm{eval}_{e_1}$ is the evaluation map on the basis vector $e_1$, is a composition of weak homotopy equivalences. Since the Gram-Schmidt retraction $G$ preserves the direction of $d_pf(\gamma'(p))$, the diagram commutes. In particular, since $d_p$, $G$, $\mathrm{eval}_{e_1}$, and $\psi$ are weak homotopy equivalences, so is $\psi_p$.
\end{proof}
\begin{rmk}\label{rmkhomotopyS1Loop}
	\Cref{Loopweakeqdiffspace} implies that $\psi_p$ induces a bijection on homotopy classes of maps
	\[
	\psi_{p,*} \colon [S^1,\Loop^{\eq}(S^2)] \to [S^1, S^1], \quad [f] \mapsto [\psi_p \circ f].
	\]
	In particular, two continuous maps $f,g \colon S^1 \to \Loop^{\eq}(S^2)$ are homotopic if and only if $\deg_\Z(\psi_p \circ f) = \deg_\Z(\psi_p \circ g)$.
\end{rmk}

To this end, let $\xi_1$, $\xi_2$ be two anti-invariant tight contact structures on $M$, inducing the same boundary characteristic foliations $S_{\pm 1,\xi_1} = S_{\pm 1, \xi_2}$. Furthermore, let 
\[
	\Gamma_z, \Lambda_z \colon S^1 \times [-1,1] \to S^2
\]
be smooth families of dividing circles for $S_{z,\xi_1}$, $S_{z,\xi_2}$ respectively. These families define paths $z \mapsto [\Gamma_z]$ and $z \mapsto [\Lambda_z]$ in $\Loop^{\eq}(S^2)$. 
\begin{defin}\label{convexrotdif}
	Suppose that $\Gamma_{\pm1} = \Lambda_{\pm1} \eqqcolon \beta_{\pm1}$. The \emph{convex rotational difference} $d^c_{\Lambda,\Gamma,p}(\xi_2, \xi_1)$ is the $\Z$-degree
	\[
		d^c_{\Lambda,\Gamma,p}(\xi_2, \xi_1) \coloneqq \deg_\Z(\psi_p[\Lambda_\bullet] * \psi_p[\Gamma_\bullet]^{-1}).
	\]	
\end{defin}

We shall now prove that \cref{rotdif,convexrotdif} produce the same integer. The main \lcnamecref{rotdifequal} will be preceded by a technical \lcnamecref{difftransv}. 
\begin{lemma}\label{difftransv}
	Let 
	\[
		f_i \colon S^1 \to S^1, \quad i = 0,1,
	\]
	be continuous maps satisfying $f_1(\theta) \neq -f_0(\theta)$ for every $\theta \in S^1$. Then their $\Z$-degrees coincide.
\end{lemma}
\begin{proof}
	Let us regard $S^1$ as the standard unit circle in $\R^2$. Since $f_1(\theta) \neq -f_0(\theta)$ for all $\theta \in S^1$, then
	\[
		F \colon S^1 \times [0,1] \to S^1,\quad (\theta,t) \mapsto \frac{(1-t)\,f_0(\theta) + t\,f_1(\theta)}{\abs{(1-t)\,f_0(\theta) + t\,f_1(\theta)}}
	\]
	is a well-defined and continuous homotopy joining $f_0$ with $f_1$. Hence, $\deg_\Z(f_0) = \deg_\Z(f_1)$.
\end{proof}
\begin{prop}\label{rotdifequal}
	Let $\xi_1$, $\xi_2$ be two anti-invariant tight contact structures on $M$, inducing the same boundary characteristic foliations $S_{\pm 1,\xi_1} = S_{\pm 1, \xi_2}$. Suppose that
	\begin{enumerate}
		\item $S_z$ is $\xi_i$-convex for $i = 1,2$, and
		\item there exist smooth families of oriented circles $\Gamma_z$, $\Lambda_z$ dividing $S_{z,\xi_1}$, $S_{z,\xi_2}$ respectively, such that $\Gamma_{\pm1} = \Lambda_{\pm1}$.
	\end{enumerate}
	Then 
	\begin{enumerate}
		\item the convex rotational difference $d^c_{\Lambda,\Gamma,p}(\xi_2, \xi_1)$ is independent of the triple $(\Lambda,\Gamma,p)$;
		\item the $F$-rotational difference $d_F(\xi_2, \xi_1)$ is independent of the connected component $F \subseteq \Fix(u)$;
		\item the two integers coincide.
	\end{enumerate}
\end{prop}
\begin{proof}
	We first observe that $\psi_{n,*} = \psi_{s,*}$. Let $\gamma$ be a great circle passing through the north and south poles. By \cref{Loopweakeqdiffspace,rmkhomotopyS1Loop}, any map $f \colon S^1 \to \Loop^\eq(S^2)$ is homotopic to the family $R(\theta)$ rotating $\gamma$ by angle $\theta\, m$ ($\theta \in [0,2\pi]$) about the $z$-axis, where $m = \deg_\Z(\psi_n \circ f)$. If $R \colon S^1 \to \Loop^\eq(S^2)$ is such a rotation, then $\psi_n(R(\theta)) = -\psi_s(R(\theta))$ for every $\theta \in S^1$, and thus
	\[
		\psi_{n,*}([f]) = [\psi_n \circ R] = [-\psi_s \circ R] = [\psi_s \circ R] = \psi_{s,*}([f]),
	\]
	since $-\id_{S^1} \colon S^1 \to S^1$ has degree $1$.
	
	As a consequence, $d^c_{\Lambda,\Gamma,p}(\xi_2,\xi_1)$ is independent of $p \in \{n,s\}$. Therefore, we can reduce to comparing the classes $d^c_{\Lambda,\Gamma,p}(\xi_2, \xi_1)$ and $d_{F_p}(\xi_2, \xi_1)$ for \emph{any choice} of the triple $(\Lambda,\Gamma,p \in \{n,s\})$ and $F_p = \{p\} \times [-1,1]$.
	
	To this end, consider the tangent space $T_qS_z = T_pS^2$ as it varies along $q = (p,z) \in F$ (see \cref{fig:angledivcurvecontplane}). For every $z \in [-1,1]$, the dividing curve $\Gamma_z \subseteq S_z$ is transversal to the characteristic foliation $S_{z;\xi_1}$, and so is $\Lambda_z$ to $S_{z,\xi_2}$. In particular, the loops
	\[
		\psi_p(\Lambda_\bullet) * \psi_p(\Gamma^{-1}_\bullet)\quad \text{ and } \quad \ell_{\bullet, \xi_2, p} * \ell_{\bullet,\xi_1,p}^{-1}
	\]
	satisfy the hypotheses of \cref{difftransv}. Hence 
	\[
		\deg_\Z(\psi_p(\Lambda_\bullet) * \psi_p(\Gamma^{-1}_\bullet)) = d^c_{\Lambda,\Gamma,p}(\xi_2, \xi_1)
	\]
	coincides with 
	\[
		\deg_\Z(\ell_{\bullet, \xi_2, p} * \ell_{\bullet,\xi_1,p}^{-1}) = d_F(\xi_2,\xi_1).
	\]
	
	Finally, for any other choice of a triple $(\Lambda',\Gamma',p)$, then 
	\[
		d^c_{\Lambda,\Gamma,p}(\xi_2, \xi_1) = d_{F_p}(\xi_2, \xi_1) = d^c_{\Lambda',\Gamma',p}(\xi_2, \xi_1)
	\]
	holds. Analogously,
	\[
		d_{F_n}(\xi_2,\xi_1) = d^c_{\Lambda,\Gamma,n}(\xi_2,\xi_1) = d^c_{\Lambda,\Gamma,s}(\xi_2,\xi_1) = d_{F_s}(\xi_2,\xi_1)
	\]
	holds. We have then proven that \cref{rotdif,convexrotdif} produce the same integer and depend solely on the pair $(\xi_1, \xi_2)$.
\end{proof}
\begin{figure}[t]
	\centering
	\def\svgwidth{0.5\linewidth}
	%% Creator: Inkscape 1.4.3 (0d15f75, 2025-12-25), www.inkscape.org
%% Accompanies image file 'Angles_svg-tex.pdf' (pdf, eps, ps)
\begingroup%
  \makeatletter%
  \providecommand\color[2][]{%
    \errmessage{(Inkscape) Color is used for the text in Inkscape, but the package 'color.sty' is not loaded}%
    \renewcommand\color[2][]{}%
  }%
  \providecommand\transparent[1]{%
    \errmessage{(Inkscape) Transparency is used (non-zero) for the text in Inkscape, but the package 'transparent.sty' is not loaded}%
    \renewcommand\transparent[1]{}%
  }%
  \providecommand\rotatebox[2]{#2}%
  \newcommand*\fsize{\dimexpr\f@size pt\relax}%
  \newcommand*\lineheight[1]{\fontsize{\fsize}{#1\fsize}\selectfont}%
  \ifx\svgwidth\undefined%
    \setlength{\unitlength}{176.59842303bp}%
    \ifx\svgscale\undefined%
      \relax%
    \else%
      \setlength{\unitlength}{\unitlength * \real{\svgscale}}%
    \fi%
  \else%
    \setlength{\unitlength}{\svgwidth}%
  \fi%
  \global\let\svgwidth\undefined%
  \global\let\svgscale\undefined%
  \makeatother%
  \begin{picture}(1,0.564992)%
    \lineheight{1}%
    \setlength\tabcolsep{0pt}%
    \put(0,0){\includegraphics[width=\unitlength,page=1]{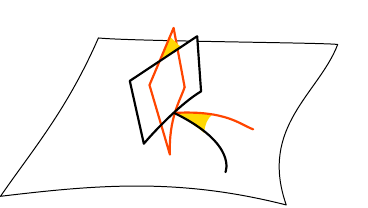}}%
    \put(0.56505337,0.49573289){\color[rgb]{0,0,0}\makebox(0,0)[lt]{\lineheight{1.25}\smash{\begin{tabular}[t]{l}$\xi_1$\end{tabular}}}}%
    \put(0.43220773,0.53200714){\color[rgb]{1,0.27058824,0}\makebox(0,0)[lt]{\lineheight{1.25}\smash{\begin{tabular}[t]{l}$\xi_2$\end{tabular}}}}%
    \put(0.89135878,0.2965068){\color[rgb]{0,0,0}\makebox(0,0)[lt]{\lineheight{1.25}\smash{\begin{tabular}[t]{l}$S_z$\end{tabular}}}}%
    \put(0.67108236,0.09393872){\color[rgb]{0,0,0}\makebox(0,0)[lt]{\lineheight{1.25}\smash{\begin{tabular}[t]{l}$\Gamma_z$\end{tabular}}}}%
    \put(0.66470281,0.27138723){\color[rgb]{1,0.27058824,0}\makebox(0,0)[lt]{\lineheight{1.25}\smash{\begin{tabular}[t]{l}$\Lambda_z$\end{tabular}}}}%
  \end{picture}%
\endgroup%

	\caption{Relationship between dividing curves and contact planes.}
	\label{fig:angledivcurvecontplane}
\end{figure}
\subsection{Proof of Theorem \getrefnumber{mainthmtomographyStwo}}
We can now move forward to prove \cref{mainthmtomographyStwo}. The first step is to prove the following result.
\begin{prop}\label{maketomographyallconvex}
	Let $\xi$ be an anti-invariant contact structure on $M$, such that $S_{z,\xi}$ has generic singularities and is of special type for $z$ close to $\pm1$. There exists an equivariant isotopy $\psi_t \colon M \to M$, which is the identity near the boundary, such that the contact structure $\xi' \coloneqq d\psi_1(\xi)$ satisfies the following properties.
	\begin{enumerate}
		\item For every $z \in [-1,1]$, the foliation $S_{z,\xi'}$ contains only finitely many singular points. On all but finitely many $z$-levels, these singularities are generic; on the remaining levels, the only non-generic singularities are half-hyperbolic.
		\item For every $z \in [-1,1]$, the foliation $S_{z,\xi'}$ contains no retrograde hyperbolic-hyperbolic connections.
	\end{enumerate}
	In particular, all surfaces $S_z$, $z \in [-1,1]$, are convex by \cref{specialconvex}.
\end{prop}

Let us fix a tight contact structure $\xi$ as in \cref{maketomographyallconvex}.
\begin{lemma}\label{existencefamilyvectorfieldgenerating}
	There exists a smooth 1-parameter anti-invariant vector field 
	\[
		X \colon S \times [-1,1] \to TS 
	\]
	such that $X_z$ generates the characteristic foliation $S_{z,\xi}$. Moreover, for every $z \in [-1,1]$, the restriction $X|_{S_z}$ depends solely on the restriction $\alpha|_{S_z}$ of an anti-invariant contact form $\alpha$ defining $\xi$.
\end{lemma}
\begin{proof}
	Let $\Omega$ be a $\gk_S$-invariant area form on $S^2$, and write $\alpha$ as $\beta_z + w_z\,dz$ for a smooth family of $1$-forms $\beta_z \in \Lambda^1(S)$ and for a smooth family of functions $w_z \in C^\infty(S)$. The unique solution $X_z \in \XX(S)$ of
	\[
		\iota_{X_z}\Omega = \beta_z\,,
	\]
	is a smooth 1-parameter vector field generating $S_{z,\xi}$.
\end{proof}
\begin{rmk}\label{rmkperturvectorfieldgivesperturbcont}
	Observe that if $Z$ is a 1-parameter anti-invariant vector field sufficiently $C^\infty$-close to $X$, the convex combination 
	\[
		X^{(t)} \coloneqq (1-t)\,X + t\,Z
	\]
	is again $C^\infty$-close to $X$. Thus, for $Z$ sufficiently $C^\infty$-close to $X$, the form 
	\[
		\alpha_t \coloneqq \iota_{X^{(t)}_z}\Omega + w_z\,dz \in \Lambda^1(M)
	\]
	is a contact form, and by construction $Z_z$ generates the characteristic foliation $S_{z, \Ker(\alpha_1)}$. Moreover, since $M$ is compact we can integrate the vector field defined by \cref{EGray} to obtain an isotopy $\phi_t$, which by \cref{rmkXtofpsitsolving} satisfies $d\phi_t(\xi) = \ker(\alpha_t)$. Furthermore, if $Z$ coincides with $X$ near $\partial M$, then both $X^{(t)}$ and $\phi_t$ are constant near the boundary.
\end{rmk}

Recall that a \emph{retrograde connection} is a separatrix of a negative (half)-hyperbolic point flowing into a positive (half)-hyperbolic point.
\begin{prop}\label{genfilm}
	There exists an equivariant isotopy $\psi_t \colon M \to M$, which is the identity near the boundary, such that the contact structure $\xi' \coloneqq d\psi_1(\xi)$ satisfies the following property. For every $z \in [-1,1]$ outside a finite set $\ZZ$, the foliation $S_{\xi'}^z$ contains only generic singularities and no retrograde connections. On the other hand, for every $z \in \ZZ$ only one of the following scenarios happens:
	\begin{enumerate}
		\item $S_{z,\xi'}$ contains non-generic singularities, necessarily half-hyperbolic, or
		\item $S_{z,\xi'}$ contains retrograde connections, necessarily finitely many. 
	\end{enumerate}
\end{prop}
\begin{proof}
	Let $X$ be a 1-parameter anti-invariant vector field which defines $S_{z,\xi}$, and let $\Sing(X_z)$ be the set of singularities of $X_z$. By \cref{rmkperturvectorfieldgivesperturbcont}, the statement is proven once we specify a 1-parameter vector field $W_z \in \XX(S)$ which is sufficiently $C^\infty$-close to $X$ and which satisfies properties 1) and 2).
	
	For every $x \in \Sing(X_z)$, the plane $\xi(x,z)$ coincides with the tangent space $T_{(x,z)}S_z$. On the other hand, by \cref{antiinvarianttangenttoF}, the structure $\xi$ is transverse to the tangent spaces $TS_z$ along the fixed-point set. Thus, the sets $\Fix(u)$ and $\Sing(X_z)$ are disjoint for every $z \in [-1,1]$. Therefore, by \cref{nondegequiv} we can approximate $X$ by a 1-parameter anti-invariant vector field $Y$, close to $X$ and transversal to the zero-section. 
	
	Since the retrograde connections cannot accumulate (see \cite[p.\ 248]{geigesIntroductionContactTopology2008}), there are only finitely many $z$-levels at which $S_{z,\xi}$ admits retrograde connections. Let $\ZZ_r'$ be the set of such levels.
	
	We want to apply the local perturbations defined by \cite[Prop.\ II.2.1, p.\ 37]{sotomayorgenericoneparameter}. To be precise, fix a \emph{finite} family of neighborhoods
	\[
		\UU = \{U_{(x,z)} \mid x \in \Sing(Y_z),\ (x,z) \in U_{(x,z)} \subseteq M\}
	\]
	such that 
	\[
		U_{u(x,z)} = u(U_{(x,z)}), \quad U_{u(x,z)} \cap U_{(x,z)} = \eset, \quad \forall x \in \Sing(Y_z) \text{ and } z \in [-1,1].
	\]
	Moreover, let $\UU_1 \sqcup \UU_2 = \UU$ be any disjoint decomposition satisfying 
	\[
		\UU_2 = \{u(U) \mid U \in \UU_1\}\,.
	\]
	As proven by Sotomayor, if the sets in $\UU$ are small enough, there exist $C^\infty$-close 1-parameter perturbations $\eta^{(1)}_U$ of $Y$ such that
	\begin{enumerate}
		\item $\eta^{(1)}_U$ is supported inside $U$, and
		\item for $z$ outside a finite set $\ZZ_h$, the 1-parameter vector field $W'_z$ defined by
		\[
			(x,z) \mapsto 
			\begin{cases}
				\eta^{(1)}_{U,z}(x) &\text{if $x \in U$},\\
				Y_z(x) & \text{otherwise}, 
			\end{cases}
		\]
		has only generic singularities; on the other hand, for $z \in \ZZ_h$ the vector field $W'_z$ has only half-hyperbolic singularities as non-generic ones. Since $\ZZ_r'$ is finite, we can suppose that the set of $z$-levels in which $W'_z$ has retrograde connections is disjoint from $\ZZ_h$.
	\end{enumerate}
	
	Finally, for every $u(U) \in \UU_2$, the 1-parameter vector field $-du \circ \eta^{(1)}_U \circ u$ shares the same dynamical properties as $\eta^{(1)}_U$, and it is $C^\infty$-close to $-du \circ Y \circ u = Y$. Define $W$ as the 1-parameter anti-invariant vector field
	\[
		(x,z) \mapsto
		\begin{cases}
			\eta^{(1)}_{U,z}(x) & \text{for $x$ belonging to some $U \in \UU_1$},\\
			(-du \circ \eta^{(1)}_{U,z}\circ u)(x) &  \text{for $x$ belonging to some $u(U) \in \UU_2$},\\
			Y_z(x) & \text{otherwise}.
		\end{cases}
	\]
	By construction, $W$ satisfies properties 1) and 2).
\end{proof}

\begin{proof}[Proof of \cref{maketomographyallconvex}]
	The proof is similar to that of \cref{Cinftyspecialtype}.
	
	First, up to an equivariant isotopy that is the identity near the boundary, we can suppose that $\xi$ satisfies the properties of \cref{genfilm}. Moreover, let us fix a level $z \in \ZZ$ such that $S_{z,\xi}$ contains retrograde connections. In particular, $S_{z,\xi}$ contains only generic singularities.	Set $H_\xi$ as the number of hyperbolic-hyperbolic connections on $S_{z,\xi}$, and $R_\xi$ as the number of those that are retrograde. 
	
	Let $\gamma$ be a retrograde connection, which joins a negative hyperbolic point $h_-$ with a positive one. Moreover, let $(\omega_n,\delta_n)_n$ be an $\omega$-$\delta$-sequence starting at $(h_-,\gamma)$, as defined in \cref{ssec:specialtype}. By \cref{sequenceeliminating}, there exists some positive integer $m$ for which $\delta_m$ connects a positive (or negative) hyperbolic point with a positive (resp.\ negative) elliptic one.
	
	Since the involution $u$ reverses the sign of the singular points, the set $\{\omega_m, \omega_{m+1}\}$  is disjoint from $\{u(\omega_m), u(\omega_{m+1})\}$.  Thus, there exists a neighborhood $U \subseteq S \times [-1,1]$ of $\delta_m$ disjoint from $u_\times(U)$. Let $\psi_t \colon M \to M$ be the isotopy defined in \cite[Prop.\ 4.9.7]{geigesIntroductionContactTopology2008}, supported in $U$. The following equivariant isotopy
	\[
	\phi_t \coloneqq
	\begin{cases}
		\psi_t^{-1} & \text{on } U\\
		u_\times \circ \psi_t^{-1} \circ u _\times  & \text{on } u_\times(U)\\
		\id_M & \text{elsewhere}
	\end{cases}
	\]
	is well-defined. 
	
	By Geiges' arguments, the deformed contact structure $\xi' \coloneqq d\phi_1(\xi)$ has the property that $S_{z,\xi'}$ has two fewer hyperbolic-hyperbolic connection. Furthermore, no other hyperbolic-hyperbolic connections are introduced, and $S_{z,\xi'}$ satisfies conditions 1)-3) of \cref{specialtype}. If $U$ is chosen sufficiently small no retrograde connections nor non-generic singularities are introduced at nearby levels.
	
	Hence, the chain of inequalities
	\[
		R_{\xi'} \le H_{\xi'} \le H_{\xi} - 2 < H_\xi
	\]
	holds. In particular, after a finite number of perturbations $\xi_1, \dots, \xi_n$, we obtain $R_{\xi_n} = 0$. The foliation $S_{z,\xi_n}$ is then of special type.
	
	By performing the above perturbation for every $z \in \ZZ$ containing retrograde connections, we obtain a final contact structure $\xi'$ such that $S_{z,\xi'}$ is of special type for every $z \in [-1,1]$.
\end{proof}

Let $r$ be the standard reflection of $S^1$ about the $x$-axis.
\begin{lemma}\label{2manifoldmadeofcircles}
	Let $\Gamma \subseteq M$ be an invariant, properly embedded, codimension-1 submanifold such that the intersection $\Gamma \cap S_z$ is transversal and defines an anti-invariant circle for every $z \in [-1,1]$. There exists a smooth map
	\[
		\Phi \colon S^1 \times [-1,1] \to M
	\]
	defining a family of equivariant embedded circles such that $\Phi(\cdot,z)$ parameterizes $\Gamma \cap S_z$.
\end{lemma}
\begin{proof}
	Set $\Gamma_z \coloneqq \Gamma \cap S_z$. Since $\Gamma$ intersects the spheres $S_z$ ($z \in [-1,1]$) transversely, for every point $q \in \Gamma_z$ there exists a unique vector $X_q \in T_q\Gamma$ satisfying
	\begin{equation}\label{eq:systemX}
	\begin{cases}
		X_q \in T_q\Gamma \cap (T_q\Gamma_z)^\bot, \\
		\gk(X_q, \partial_z) = 1,
	\end{cases}
	\end{equation}
	where the orthogonal complement is taken inside $T_qM$. 
	
	Furthermore, $S^1$ admits a unique orientation-reversing smooth involution up to smooth conjugacy\footnote{The result is classical, used for example in \cite[Remark 3.8]{duggerInvolutionsSurfaces2019}.}, whence there exists an equivariant embedding $\rho \colon S^1 \to S_{-1}$ onto $\Gamma_{-1}$. Let 
	\[
	\Phi \colon S^1 \times [-1,1] \to M
	\]
	be the map 
	\[
	\Phi(x,z) = \Phi^X_{z+1}(\rho(x)),
	\]
	where $\Phi^X$ is the integral flow of  $X$.
	
	By the second condition in \cref{eq:systemX}, the map $\Phi$ is of the form $(\phi(x,z),z)$, and thus the image
	\[
		A_z \coloneqq \{\Phi(x,z) \mid x \in S^1\}
	\]
	is an embedded circle contained in $S_z$. Since $X$ is tangent to $\Gamma$, $A_z$ is contained in $\Gamma \cap S_z = \Gamma_z$. Since $\Gamma_z$ is an embedded circle by hypothesis, the sets $A_z$ and $\Gamma_z$ coincide. 
	
	Finally, since $\gk$ is $u_\times$-invariant, so is the vector field $X$. Thus
	\[
	\begin{split}
		\Phi(r(x), z) &= \Phi^X_{z+1}(\rho(r(x))) = \Phi^X_{z+1}((u_\times \circ \rho)(x)) \\
		&= u_\times(\Phi^X_{z+1}(\rho(x))) = u_\times(\Phi(x,z))\,.
	\end{split}
	\]	
	Hence, $\Phi$ is a smooth family of equivariant embedded circles, proving the assertion.
\end{proof}
\begin{prop}\label{choosezsmoothdividingset}
	Let $\xi$ be an anti-invariant tight contact structure on $M$ satisfying conditions 1) and 2) of \cref{maketomographyallconvex}. There exists a smooth family
	\[
		\Gamma \colon S^1 \times [-1,1] \to S^2
	\]
	of equivariant embedded circles such that the circle $\Gamma_z$ divides $S_{z,\xi}$. 
	
	Moreover, the map $\Gamma$ depends solely on a positive anti-invariant contact form $\alpha$ defining $\xi$.
\end{prop}
\begin{proof}	
	Let us fix a positive anti-invariant form $\alpha$ for $\xi$. The proof is inspired by \cite[Lemma I.10.1]{dominicHomotopyTypeSpace2018}. First, by \cref{existencefamilyvectorfieldgenerating}	we can use $\alpha$ to define a 1-parameter anti-invariant vector field $X$ defining $S_{z,\xi}$.
	
	Fix $z_0 \in [-1,1]$, and let $\beta^{z_0}$ be the pull-back $i^*_{S_{z_0} \hookrightarrow M}\alpha$. By \cref{specialconvex}, the surface $S_{z_0}$ admits a dividing circle, depending solely on $X$. In addition, by \cref{divconvex} there exists an anti-invariant function
	\[
		w^{z_0} \colon S \to \R
	\]
	such that the form $\alpha_{z_0} \coloneqq \beta^{z_0} + w^{z_0}\,d\zeta$ is a contact form in a neighborhood $S \times I_{z_0}$ ($I_z \subseteq [-1, 1]$) of $S_{z_0}$ (where the $\R$-coordinate is parameterized by $\zeta$). In particular, by \cref{invariantcontactcondition} the condition 
	\[
		\beta^{z_0} \wedge dw^{z_0} + w^{z_0}\, d\beta^{z_0} > 0
	\]
	holds. Thus, there exists a neighborhood  $S \times N_{z_0}$ of $S_{z_0}$ ($N_{z_0} \subseteq I_{z_0}$) such that the 1-form  
	\[
	\begin{split}
		S \times N_{z_0} &\to \Lambda^1T(S \times N_{z_0}) \\
		(x,z) &\mapsto \beta^z(x) + w^{z_0}(x)\,dz
	\end{split}
	\]
	satisfies the condition
	\[
		\Lambda^2S_z \ni \beta^z \wedge dw^{z_0} + w^{z_0}\, d\beta^z > 0,\ \forall z \in N_{z_0}\,.
	\]
	
	From the collection $\{N_z\}_{z \in [-1,1]}$, we can choose a finite number of levels $\{z_i\}_{i=1}^m$ and a finite number of neighborhoods $\{N_i \coloneqq N_{z_i}\}_{i=1}^m$ covering $[-1,1]$. Let us choose a partition of unity $\{\phi_i\}_{i = 1}^m$ subordinate to $\{N_i\}_{i = 1}^m$, and define
	\[
		K \colon S \times [-1,1] \to \R
	\]
	as
	\[
		(x,z) \mapsto \sum_{i = 1}^m \phi_i(z)\cdot w^{z_i}(x)\,.
	\]
	
	For every $z \in [-1,1]$, let $K^z$ be the restriction $K|_{S_z}$. The vertically invariant 1-form
	\[
		\delta^z \coloneqq \beta^z + K^z\,d\zeta
	\]
	satisfies the $\zeta$-invariant contact condition
	\[
		\beta^z \wedge dK^z + K^z\,d\beta^z = \sum_{i=1}^m\phi_i(z)\cdot \left[ \beta^z \wedge dw^{z_i} + w^{z_i}\cdot d\beta^z\right] > 0\,.
	\]
	Moreover, $\delta^z$ defines on $S_z$ the same characteristic foliation as that of $\xi$. Thus by \cref{rmkconveximpliesdividingcurve}, the set
	\[
		\Gamma_z \coloneqq \{x \in S^2 \mid K^z(x) = 0\}
	\] 
	is a dividing curve for $S_{z,\xi}$. In particular, $\Gamma_z$ is a circle by \cref{rmkdivsetS2iscircle}. Moreover, by \cref{invariantcontactcondition} the differential $d_xK^z$ is surjective for every $x \in (K^z)^{-1}(0)$. 
	
	Consequently, also the full differential $d_{(x,z)}K$ is surjective for every $(x,z) \in K^{-1}(0)$, and the preimage $K^{-1}(0) \subseteq M$ is a 2-dimensional embedded submanifold. Moreover, the surjectivity of $d_xK^z$ implies that $K^{-1}(0)$ is transverse to the horizontal spheres. The result then follows by \cref{2manifoldmadeofcircles}.
\end{proof}

Let us now prove \cref{mainthmtomographyStwo}. Let $X$ be an anti-invariant vector field defined on and tangent to $S_{\pm 1}$. We claim that the set
\[
T_X \coloneqq \{\xi \in \Tight^{\ant}(M) \mid S_{\pm1, \xi} = [X]\}\,,
\]
if not empty, defines a quotient 
\[
\Xi_X \coloneqq T_X/\sim_\eq
\]
(where $\sim_\eq$ denotes equivariant isotopy of $M$ rel boundary) which is a $\Z$-torsor. In particular, we claim that the map 
\[
	\varphi \colon \Xi_X \times \Xi_X \to \Z,\quad ([\xi_1],[\xi_0]) \mapsto d(\xi_1,\xi_0)
\]
is well-defined and satisfies the conditions:
\begin{enumerate}[label={C\arabic*)}]
	\item $\varphi([\xi_1], [\xi_0]) = 0$ if and only if $[\xi_1] = [\xi_0]$,
	\item for any class $[\xi_0] \in \Xi_X$ and every $n \in \Z$ there exists a class $[\xi_1]$ such that $\varphi([\xi_1], [\xi_0]) = n$,
	\item $\varphi([\xi_2],[\xi_1]) + \varphi([\xi_1], [\xi_0]) = \varphi([\xi_2], [\xi_0])$ for all $[\xi_0], [\xi_1], [\xi_2] \in \Xi_X$.
\end{enumerate}
Since the rotational difference $d(\xi_1,\xi_2)$ is invariant under equivariant isotopy rel boundary (by \cref{rotdiffinvaiso}), the map $\varphi$ is well-defined.\vspace{0.1cm}

In order to prove condition C1), let us fix $[\xi_1], [\xi_0] \in \Xi_X$ such that $\varphi([\xi_1], [\xi_0]) = 0$. We wish to prove that $\xi_1$ and $\xi_0$ are isotopic relative to $\partial M$. 
\begin{figure}[t]
	\centering
	\def\svgwidth{0.9\linewidth}
	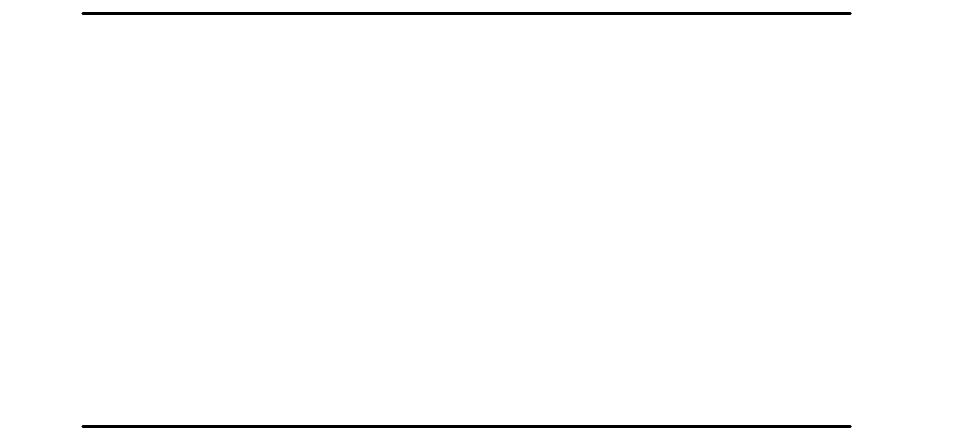
	\caption{Various neighborhoods of the boundary spheres.}
	\label{fig:ReductionTomography}
\end{figure}
\begin{prop}\label{reduceeasythmtomography}
	We can restrict ourselves to the case in which $\xi_1$ and $\xi_0$ coincide in neighborhoods of $S_{\pm1}$, and $S_{z,\xi_i}$ has generic singularities and is of special type for $i = 1,0$ and $z$ close to $\pm1$.
\end{prop}
\begin{proof}
	See \cref{fig:ReductionTomography}. By applying \cref{Girouxbound} with $\phi = \id$, we can suppose that $\xi_1$ and $\xi_0$ coincide in a neighborhood $\NN(S_{\pm1})$ of the boundary spheres without changing their isotopy class rel boundary.
	
	Let $S_{\pm1}' \subseteq \NN(S_{\pm1})$	be boundary-parallel spheres, close to $S_{\pm1}$. By \cref{Cinftyspecialtype}, there exists an equivariant isotopy $\psi_t$, supported in small open vertical neighborhoods $\NN'(S_{\pm1}') \subsetneq \NN$, such that $S_{\pm 1}'$ has generic singularities and is of special type. By \cref{specialconvex}, $S_{\pm1}'$ are convex surfaces for $d\psi_1(\xi_1)$ (and thus for $d\psi_1(\xi_0)$). 
	
	Let us still denote such deformed structures by $\xi_i$, $i = 1,0$. By openness of the convex condition, we can suppose that every surface $S_z$ contained in $\NN'$ is both $\xi_1$- and $\xi_0$-convex.
	
	Finally, let $M'$ be the product manifold bounded by $S_{\pm1}'$. Since $\xi_1$, $\xi_0$ coincide when restricted to $\NN$, the rotational difference $d(\xi_1|_{M'}, \xi_0|_{M'})$ coincides with $d(\xi_1, \xi_0) = 0$. Moreover, if  $\xi_1|_{M'}$, $\xi_0|_{M'}$ are equivariantly isotopic relative to a neighborhood of $\partial M'$, then $\xi_1$, $\xi_0$ are equivariantly isotopic relative to a neighborhood of $\partial M$. 
	
	By construction, the structures $\xi_1|_{M'}$, $\xi_0|_{M'}$ coincide when restricted to $M' \cap \NN'$, and moreover all the surfaces $S_z$ contained in $M' \cap \NN'$ are convex for both $\xi_1|_{M'}$ and $\xi_0|_{M'}$. We have thus reduced the proof to analyzing the case in the statement.
\end{proof}

By \cref{reduceeasythmtomography}, we can suppose that $\xi_1$ and $\xi_0$ coincide in neighborhoods of $S_{\pm1}$, and $S_{z,\xi_i}$ has generic singularities and is of special type for $i = 1,0$ and $z$ close to $\pm1$. Moreover, set $\alpha_1$, $\alpha_0$ to be normalized positive anti-invariant contact forms for $\xi_1$, $\xi_0$ respectively. By uniqueness of such contact forms, $\alpha_1$ and $\alpha_0$ coincide near $S_{\pm 1}$.

According to \cref{maketomographyallconvex}, we can also suppose, without changing the isotopy class of $\xi_1$ and $\xi_0$ relative to a neighborhood of $\partial M$, that the surfaces $S_z$ ($z \in [-1,1]$) are $\xi_1$- and $\xi_0$-convex. According to \cref{choosezsmoothdividingset}, there exist families of dividing curves $\Lambda$, $\Gamma$ dividing $\xi_1$, $\xi_0$ respectively. Using $\alpha_1$ and $\alpha_0$ to define such families, we have $\Gamma_z = \Lambda_z$ for $z$ close to $\pm 1$. By \cref{rotdifequal}, the integer $d(\xi_1, \xi_0)$ coincides with the convex rotational difference $d^c_{\Lambda,\Gamma}(\xi_1, \xi_0)$, which therefore vanishes. We shall prove that the families $\Gamma$ and $\Lambda$ can be joined by an isotopy that is the identity on a neighborhood of $\partial M$.
\begin{prop}\label{makedividcurvecoincide}
	There exists an equivariant isotopy rel boundary $h_t \colon M \to M$ such that the composition $h_1 \circ \Gamma$ is a family of dividing curves for $\xi_1$. Moreover, $h_t$ can be chosen to be the identity on a neighborhood of the boundary. 
\end{prop}
\begin{proof}
	Let $\gamma \coloneqq \Gamma_{-1} = \Lambda_{-1}$. By \cref{Loopweakeqdiffspace}, the map 
	\[
	\psi \colon Q^u_\fix \to \Loop^{\eq}(S^2)\,,\quad f \mapsto [f \circ \gamma]\,.
	\]
	is a Serre fibration with weakly contractible fiber. Thus, there exist maps
	\[
	f_1, f_0 \colon [-1,1] \to Q^u_\fix 
	\]
	such that 
	\[
	\begin{cases}
		[f_{0,z} \circ \Gamma_{-1}] = [\Gamma_z],\\
		[f_{1,z} \circ \Gamma_{-1}] = [\Lambda_z],\\
		f_{i,z} = \id_{S^2} &\text{near $z = -1$ for $i = 0,1$},\\
		f_{0,z} = f_{1,z} &\text{near $z = 1$}.
	\end{cases}
	\]
	Moreover, by the statement of \cref{Palaiseq}, we can suppose that $f_{i,z}$ ($i = 0,1$) is a smooth isotopy.
	
	Since $d(\xi_1, \xi_0) = d^c_{\Lambda,\Gamma}(\xi_1,\xi_0)$ vanishes, the paths $[\Gamma_\bullet]$, $[\Lambda_\bullet]$ are homotopic relative to $\{-1,+1\}$. Moreover, since the map 
	\[
		\psi_* \colon \pi_1(Q^u_\fix) \to \pi_1(\Loop^{\eq}(S^2))
	\]
	is an isomorphism, the paths $f_{0,z}$, $f_{1,z}$ (which join the same points in $Q^u_\fix$) are also homotopic. Let
	\[
		H \colon [-1,1] \times [0,1] \to Q^u_\fix
	\]
	be a homotopy relative to the endpoints, joining $f_{0,z}$ with $f_{1,z}$. The composition $K_{z,t} \coloneqq H_{z,t} \circ f_{0,z}^{-1}$ satisfies 
	\[
	\begin{cases}
		K_{z,0} = \id_{S^2}& \text{for $z \in [-1,1]$},\\
		[K_{z,1} \circ \Gamma_z] = [H_{z,1} \circ \Gamma_{-1}] = [\Lambda_z] & \text{for $z \in [-1,1]$},\\
		K_{z,t} = H_{z,t} \circ f^{-1}_{0,z} = \id_{S^2} & \text{for $z$ close to $\pm1$}.\\
	\end{cases}
	\]
	
	Finally, the map $K$ naturally defines a continuous 2-parameter family of equivariant diffeomorphisms
	\[
		\hat{K} \colon S^2 \times [-1,1] \times [0,1] \to S^2.
	\]
	According to \cite[Thm.\ VI.4.2]{bredonIntroductionCompactTransformation1972}, we can approximate $\hat{K}$ relatively to the closed set $S^2 \times [-1,1] \times \{0,1\}$ by a smooth family $h_t \colon M \to M$ of diffeomorphisms. Such a map defines the required isotopy. Since the family $\hat{K}$ is constantly the identity near $z = \pm1$, the approximation $h_t$ is also the identity for $z$ close to $\pm1$.
\end{proof}
Let $h_t \colon M \to M$ be the isotopy defined by \cref{makedividcurvecoincide}. Since $h_1$ is the identity along the boundary $\partial M$, the contact structures $dh_1(\xi_0)$, $\xi_1$ belong to $T_X$. Moreover, they admit the same set of dividing circles by construction. \Cref{uniqueness} finally proves that $dh_1(\xi_0)$, $\xi_1$ are isotopic relative to a neighborhood of the boundary.
\begin{prop}\label{uniqueness}
	Choose $\xi_0, \xi_1 \in T_X$, and suppose that the following hold.
	\begin{enumerate}
		\item The structures $\xi_0$, $\xi_1$ coincide for $z$ close to $\pm 1$.
		\item Each surface $S_z$ ($z \in [-1,1]$) is convex for either contact structure, and there exists a smooth family $\Gamma$ of equivariant embedded circles dividing both $S_{z,\xi_0}$ and $S_{z,\xi_1}$.
	\end{enumerate}
	Then $\xi_0$ and $\xi_1$ are equivariantly isotopic relative to a neighborhood of the boundary.
\end{prop}
\begin{proof}
	Let us fix the unique normalized positive anti-invariant contact forms 
	\[
	\alpha^i = \beta^i_z + w^i_z\, dz, \quad i=0, 1,
	\]
	which define $\xi_i$. By uniqueness of such forms, they coincide close to $S_{\pm 1}$.
	
	According to \cref{existencefamilyvectorfieldgenerating}, there exist 1-parameter anti-invariant vector fields $Y^i$ ($i = 0,1$), generating $S_{z,\xi_i}$, which coincide when restricted to $S_{\pm 1}$. Since $\Gamma_z$ divides both $S_{z,\xi_0}$ and $S_{z,\xi_1}$, by \cref{divconvex} the vector fields $Y^i$ determine smooth families of anti-invariant functions $v^i_z \in C^\infty(S)$  such that the forms
	\[
	\beta_z^i + v^i_z\,d\zeta,\quad i = 0,1,
	\] 
	satisfy the vertically invariant contact condition \eqref{eq:invariantcontcond} (the vertical parameter is denoted by $\zeta$). 
	
	By proceeding as in \cite[Lemma 4.9.2]{geigesIntroductionContactTopology2008}, there exists a smooth family of contact forms $\alpha_t$ that is stationary near $S_{\pm1}$, joining $\alpha^0$ with $\alpha^1$. Since the involution $u$ is level-preserving, it is immediate to check that the constructed contact forms $\alpha_t$ are anti-invariant.
	
	Finally, set $\xi_t \coloneqq \Ker(\alpha_t)$. By \cref{EGray}, there exists a unique invariant time-dependent vector field $X_t\in \xi_t$ which vanishes on $S_{\pm1}$ and satisfies \cref{eq:Xt}. Since $S \times [-1,1]$ is compact, we can integrate $X_t$ up to time $t = 1$ across the entire manifold $M$. Let $\psi_t$ be the integral flow defined by $X$. 
	
	By \cref{rmkXtofpsitsolving}, the distribution $d\psi_1(\xi_0)$ coincides with $\xi_1$. Since $\alpha_t$ is stationary for $z$ close to $\pm1$, the time-dependent vector field $X_t$ constantly vanishes near $S_{\pm 1}$. Thus, the isotopy $\psi_t$ is the identity near $\partial M$.
\end{proof}
\begin{proof}[Proof of \cref{mainthmtomographyStwo}]
	By the discussion above, the map 
	\[
		\varphi \colon \Xi_X\times\Xi_X \to \Z
	\]
	is well-defined and satisfies condition C1). Moreover, by the same discussion, if $\xi_1,\xi_2 \in T_X$ are equivalent structures that coincide near $\partial M$, the connecting isotopy can be taken to be the identity near the boundary.\vspace{0.3cm}
	
	Let us prove condition C2). Fix $\xi_0 \in T_X$ and $n \in \Z$. Let $\Phi_n \colon M \to M$ be the equivariant diffeomorphism
	\[
		\Phi_n(\hat{v},z) = (A_z^n\,\hat{v},z)\,, 
	\]
	where $A^n_z \colon \R^3 \to \R^3$ is the linear map given by the matrix
	\[
	\begin{bmatrix}
		\cos(n\pi(1+z)) & -\sin(n\pi(1+z)) & 0 \\
		\sin(n\pi(1+z)) & \cos(n\pi(1+z)) & 0 \\
		0 & 0 & 1
	\end{bmatrix}\,.
	\]
	By construction, $\Phi_n|_{S_{\pm1}} = \id_{S_{\pm1}}$, and $d\Phi_n(\xi_0)$ thus belongs to $T_X$. Moreover, $d\Phi_n(\xi_0)$ performs $n$ full twists with respect to $\xi_0$ along the $z$-axis. Hence,
	\[
		\varphi([d\Phi_n(\xi_0)], [\xi_0]) = d(d\Phi_n(\xi_0), \xi_0) = n\,.
	\] 
	We have thus proven condition C2).\vspace{0.3cm}
	
	Finally, for every choice $[\xi_0],[\xi_1],[\xi_2] \in \Xi_X$,
	\[
	\begin{split}
		\varphi([\xi_2], [\xi_1]) + \varphi([\xi_1],[\xi_0])&=d(\xi_2,\xi_1) + d(\xi_1,\xi_0)\\
		&= \deg_\Z(\ell_{\bullet,\xi_2,n} * \ell^{-1}_{\bullet,\xi_1,n}) + \deg_\Z(\ell_{\bullet,\xi_1,n} * \ell^{-1}_{\bullet,\xi_0,n})\\
		&=\deg_\Z(\ell_{\bullet,\xi_2,n} * \ell^{-1}_{\bullet,\xi_0,n})\\
		&= \varphi([\xi_2],[\xi_0]),
	\end{split}
	\]
	since the $\Z$-degree is additive under concatenations. We have thus proven condition C3).
\end{proof}
\section{Proof of Theorems \getrefnumber{mainthmRthree} to \getrefnumber{mainthmDthree}}\label{sec:proofmainthm}
We shall prove the following results.
\begin{thm}\label{mainthmRthree}
	The pair $(\R^3, u_{\std})$ admits a unique anti-invariant tight contact structure up to equivariant isotopy.
\end{thm}
\begin{thm}\label{mainthmDthreecomplementary}
	Two anti-invariant tight contact structures on $(\R^3 \setminus B^3_{\std}, u_{\std})$ inducing the same characteristic foliation on $\partial D^3_{\std}$ are equivariantly isotopic relative to $\partial D^3_\std$.
\end{thm}
\begin{thm}\label{mainthmSthree}
	Up to equivariant isotopy, the set of anti-invariant tight contact structures on $(S^3,u_{\std})$ is a $\Z$-torsor.
\end{thm}
\begin{thm}\label{mainthmS2timesS1}
	Up to equivariant isotopy, the set of anti-invariant tight contact structures on $(S^2 \times S^1, u_\times)$ is non-empty and a $\Z$-torsor.
\end{thm}
\begin{thm}\label{mainthmDthree}
	Let $X \in \XX(\partial D^3_\std)$ be an anti-invariant vector field which satisfies the divergence condition \eqref{eq:divcond}. Up to equivariant isotopy relative to $\partial D^3_\std$, the set of anti-invariant tight contact structures on $(D^3_{\std}, u_{\std})$ inducing the characteristic foliation $\ang{X}$, if non-empty\footnote{see \cref{XiXnonempty}}, is a $\Z$-torsor.
\end{thm}
We will need some auxiliary results. Let $\gk$ be the standard Euclidean metric on $\R^3$.
\begin{prop}\label{Rthreestandardxaxis}
	Let $\xi_1$, $\xi_0$ be anti-invariant tight contact structures on $\R^3$ which coincide on a small ball $B_\epsilon(\zero)$. There exists an equivariant isotopy $\phi_t \colon \R^3 \to \R^3$ such that $d\phi_1(\xi_0)$ coincides with $\xi_1$ when restricted to the $z$-axis. Moreover, $\phi_t$ can be chosen to be the identity on $B_\epsilon(\zero)$.
\end{prop}
\begin{proof}
	Let $\alpha_i$ ($i = 0,1$) be the positive normalized contact form generating $\xi_i$, and let $v_i$ be the $\gk$-dual vector of $\alpha_i$. For every $z \in \R$ let $\theta_i(z) \in S^1$ be the projection of $v_i(0,0,z)$ on the canonical circle centered on $(0,0,0)$ and contained in the $xy$-plane. Set
	\[
		\tilde{\theta}_i \colon \R \to \R
	\]
	to be the smooth lift of $\theta_i$ to the universal cover 
	\[
		\R \to S^1, \quad 0 \mapsto (1,0,0).
	\]
	such that $\tilde{\theta}_0(0) = \tilde{\theta}_1(0)$. Then
	\[
	\phi_t 	\colon \R^3 \to \R^3,\quad (x,y,z) \mapsto A_{t\tilde{\theta}_1(z) - t\tilde{\theta}_0(z)}(x,y,z),
	\]
	where $A_\nu \colon \R^3 \to \R^3$ is the linear map defined by the matrix
	\[
	A_{\nu} \coloneqq
	\begin{bmatrix}
		\cos\nu & -\sin\nu & 0\\
		\sin\nu & \cos\nu & 0\\
		0 & 0 & 1
	\end{bmatrix},
	\]
	is the desired isotopy.
\end{proof}

\begin{prop}\label{contastrnearspherestandard}
	Let $\xi_1$, $\xi_0$ be two anti-invariant tight contact structures defined on a neighborhood of a standard 2-sphere $S \subseteq \R^3$ of any radius. Then there exists an equivariant isotopy $\psi_t \colon \R^3\to\R^3$ supported arbitrarily close to $S$, such that $d\psi_1(\xi_0) = \xi_1$ close to $S$.
\end{prop}
\begin{proof}
	First, by \cref{Cinftyspecialtype}, we can suppose that $S_{\xi_0}$ and $S_{\xi_1}$ are of special type, without changing the equivariant isotopy class of $\xi_0$ and $\xi_1$. In particular, by \cref{specialconvex}, $S$ is a convex surface with respect to both $\xi_0$ and $\xi_1$, and thus $S_{\xi_0}$, $S_{\xi_1}$ admit oriented dividing curves $\Gamma$, $\Lambda$ respectively, satisfying \cref{defdividingcurve}. By \cref{rmkdivsetS2iscircle}, such curves are (anti-invariant) circles. Let us orient $\Gamma$ such that, for every $p \in \Gamma$, $(\Gamma'(p), S_{\xi_0}(p))$ defines a positive orientation on $S$. Let us analogously orient $\Lambda$.
	
	The curves $\Gamma$ and $\Lambda$ define classes $[\Gamma]$, $[\Lambda]$ in $\Loop^{\eq}(S^2)$, which are connected by \cref{loopspaceconnected}. Moreover, the map
	\[
	\psi \colon Q^u_\fix \to \Loop^{\eq}(S), \quad f \mapsto [f \circ \Gamma]
	\]
	is a Serre fibration, and a weak homotopy equivalence, by \cref{Loopweakeqdiffspace}. Thus, there exists a continuous map 
	\[
	f_t \colon [0,1] \to Q^u_\fix
	\]
	such that $f_0 = \id_{S^2}$ and $f_1(\Gamma) = \Lambda$ \emph{as oriented circles}. By \cite[Thm.\ VI.4.2]{bredonIntroductionCompactTransformation1972}, we can approximate such isotopy with a smooth map 
	\[
	h \colon S \times [0,1] \to S \hookrightarrow \R^3
	\]
	that coincides with $f$ when restricted to $S \times \{0,1\}$. Moreover, we can extend such an isotopy to a global equivariant isotopy 
	\[
	\tilde{h} \colon \R^3 \times [0,1] \to \R^3
	\]
	supported arbitrarily close to $S$ (see \cite[Thm.\ VI.3.1]{bredonIntroductionCompactTransformation1972}). The structures $\xi_1$ and $d\tilde{h}_1(\xi_0)$ are such that $S_{\xi_1}$ and $S_{d\tilde{h}_1(\xi_0)}$ are divided by the same dividing circle.
	
	We have thus restricted to study two anti-invariant tight contact structures $\xi_0$, $\xi_1$ such that $S_{\xi_0}$, $S_{\xi_1}$ are divided by the same dividing circle $\Gamma = \Lambda$. Let $X_0, X_1 \in \XX(S)$ generate $S_{\xi_0}$, $S_{\xi_1}$ respectively, and let $\Omega_0$, $\Omega_1$ be the invariant area forms such that $\diver_{\Omega_i}(X_i)$ satisfies the divergence condition appearing in \cref{defdividingcurve}. Let $g \colon S \to \R_{>0}$ be the smooth, positive, invariant function satisfying $\Omega_1 = g\,\Omega_0$.
	
	As in the non-equivariant setting, see \cite[Thm.\ 4.8.11]{geigesIntroductionContactTopology2008}, the anti-invariant vector field
	\[
		X_t \coloneqq t\,gX_1 + (1-t)\,X_0
	\]
	is divided by the set $\Gamma$. In particular, by \cref{divconvex}, there exists a smooth family $\alpha_t \coloneqq \iota_{X_t}\Omega_0 + w_t\,dz$ of vertically invariant, anti-invariant contact forms, defined in an arbitrary invariant tubular neighborhood $\NN(S) \simeq S \times \R$.
	
	Let $Y$ be the 1-parameter vector field solving \cref{eq:Xt} for $(\alpha_t)_{t \in [0,1]}$. Since $Y_t$ is vertically invariant at every time, it integrates up to time $t = 1$ on the whole $\NN(S)$. Let $\NN'(S) = S \times [-N,N]$ be a compact subset of $\NN(S)$ which contains $\bigcup_{t \in [0,1]}\Phi^Y_t(S)$. 
	
	The required isotopy $\psi_t$ is then defined by integrating the vector field $\rho\cdot Y_t$, where $\rho \colon \NN(S) \to \R_{\ge 0}$ is a compactly supported smooth function, depending only on the vertical coordinate, such that $\rho|_{\NN'} \equiv 1$. Indeed, by \cref{rmkXtofpsitsolving},
	\[
		d\psi_1(\xi_0)(p) = d\Phi^{\rho Y}_1(\xi_0)(p) = d\Phi^Y_1(\xi_0)(p) = \xi_1(p) 
	\]
	for every $p \in \NN'$.
\end{proof}
\begin{proof}[Proof of \cref{mainthmRthree}]
	First, the standard contact structure
	\[
		\xi_{\std} \coloneqq \Ker(dx - z\,dy)
	\]
	is anti-invariant and tight, see \cite[Cor.\ 6.5.10]{geigesIntroductionContactTopology2008}.
	
	Let $\xi_1$ and $\xi_0$ be anti-invariant tight contact structures on $\R^3$.
	
	Thanks to \cref{contastrnearspherestandard}, we can suppose that $\xi_1$ and $\xi_0$ coincide in a small ball $B_\epsilon$ around the origin without changing their equivariant isotopy class. Analogously, by \cref{Rthreestandardxaxis},  we can suppose that both $\xi_1$ and $\xi_0$ are constant on the whole $z$-axis.
	 
	Let $S^2_n\subseteq\R^3$ be the standard 2-sphere of radius $n$. According to \cref{contastrnearspherestandard}, we can suppose that $\xi_1$, $\xi_0$ induce the same characteristic foliation on $S_n$ without changing their equivariant isotopy classes. Thus, according to \cref{Girouxbound}, we can equivariantly perturb $\xi_0$ to a new contact structure, still called $\xi_0$, that coincides with $\xi_1$ in small neighborhoods of the spheres $S_n^2$, $n \in \N_{>0}$. Let $B^3_n$ ($n > 0$) be standard $3$-ball of radius $n$. 
	
	Set $B^3_0 \coloneqq B_\epsilon$, and let $I_n$ be the closure of $B_{n+1}^3 \setminus B_n^3$ for $n \in \N$. Observe that for every $n \in \N$, the map
	\[
		\phi_n \colon I_n \to S^2 \times [-1,1], \quad v \mapsto (v/\abs{v}, 2\abs{v} -2n-1),
	\]
	is a $(u_{\std},u_\times)$-equivariant diffeomorphism. Analogously, also the map
	\[
		\phi_0 \colon I_0 \to S^2 \times [-1,1], \quad v \mapsto (v/\abs{v}, (2\abs{v} -\epsilon-1)/(1-\epsilon)),
	\]
	is $(u_\std,u_\times)$-equivariant.
	
	Recall \cref{rotdif}. Since $\xi_1$ and $\xi_0$ are constant, and hence equal, along the $z$-axis, the rotational difference between the structures $d\phi_n(\xi_1|_{I_n})$ and $d\phi_n(\xi_0|_{I_n}))$ vanishes. According to \cref{mainthmtomographyStwo}, the structures $d\phi_n(\xi_0|_{I_n})$ and $d\phi_n(\xi_1|_{I_n})$ are equivariantly isotopic relative to a neighborhood of the boundary. Let 
	\[
		h_{n,t} \colon S^2 \times [-1,1] \to S^2 \times [-1,1]
	\]
	be such an isotopy mapping $d\phi_n(\xi_0|_{I_n})$ onto $d\phi_n(\xi_1|_{I_n})$. Then the map
	\[
		\R^3 \times [0,1] \to \R^3, \quad (p,t) \mapsto
		\begin{cases}
			p & p \in B_\epsilon,\\
			(\phi_n^{-1} \circ h_{n,t} \circ \phi_n)(p) & p \in I_n,
		\end{cases}
	\]
	defines an equivariant isotopy of $\R^3$ mapping $\xi_0$ onto $\xi_1$.
\end{proof}
\begin{prop}\label{Bthreestandardxaxis}
	Let $\xi_1$, $\xi_0$ be two anti-invariant tight contact structures on $M \coloneqq \R^3 \setminus B^3_\std$, which coincide on a collared neighborhood of $\partial D^3_\std$. There exists an equivariant isotopy $\phi_t \colon M \to M$ relative to such neighborhood such that $d\phi_1(\xi_0)$ coincides with $\xi_1$ on the $z$-axis.
\end{prop}
\begin{proof}
	Let $\alpha_i \in \Lambda^1(M)$ be the positive normalized contact forms generating $\xi_i$, and let $v_i$ be the $\gk$-dual vector field of $\alpha_i$. Since $\xi_0$ and $\xi_1$ coincide on a collared neighborhood of $\partial D^3_\std$, we have $v_0(0,0,z) = v_1(0,0,z)$ for every point $(0,0,z)$ in this neighborhood. 
	
	For every $z \in (-\infty,-1] \cup [1,+\infty)$, let $\theta_i(z) \in S^1$ be the projection of $v_i(0,0,z)$ onto the canonical circle centered at $(0,0,0)$ and contained in the $xy$-plane. Set
	\[
		\tilde{\theta}_i \colon (-\infty,-1] \cup [1,+\infty) \to \R
	\]
	to be the smooth lift of $\theta_i$ to the universal cover 
	\[
	\R \to S^1, \quad 0 \mapsto (1,0,0)
	\]
	such that $\tilde{\theta}_1(\pm 1) = \tilde{\theta}_0(\pm 1)$. Since $\xi_0$ and $\xi_1$ coincide on a collared neighborhood of $\partial D^3_\std$, the difference  
	\[
		\Delta \tilde{\theta} \colon(-\infty, -1] \cup [1,+\infty) \to \R,\quad\tilde{\theta}_1(z) - \tilde{\theta}_0(z) 
	\]
	vanishes near $\pm1$, and hence extends to a smooth map from $\R$ to $\R$ that identically vanishes on $[-1, 1]$.
	
	Then
	\[
	\phi_t 	\colon M \to M,\quad (x,y,z) \mapsto A_{t\Delta \tilde{\theta}(z)}(x,y,z),
	\]
	where $A_\nu \colon \R^3 \to \R^3$ is the linear map defined by the matrix
	\[
	A_{\nu} \coloneqq
	\begin{bmatrix}
		\cos\nu & -\sin\nu & 0\\
		\sin\nu & \cos\nu & 0\\
		0 & 0 & 1
	\end{bmatrix},
	\]
	is the desired isotopy. Indeed, $\phi_t$ is the identity on an open neighborhood of $\{(x,y,z) \in M \mid -1 \le z \le 1\}$, and thus on a collared neighborhood of $\partial D^3_\std$.
\end{proof}
\begin{proof}[Proof of \cref{mainthmDthreecomplementary}]
	Let $\xi_1$ and $\xi_0$ be two anti-invariant tight structures on $M \coloneqq \R^3 \setminus B^3_{\std}$ which induce the same characteristic foliation on $\partial D^3_\std$. By \cref{Girouxbound}, we can suppose that $\xi_1$ and $\xi_0$ coincide in a small collared neighborhood of $\partial D^3_{\std}$ without changing their isotopy class relative to $\partial M$. Analogously, by \cref{Bthreestandardxaxis}, we can suppose that $\xi_1$ and $\xi_0$ coincide on the $z$-axis. The proof then follows as in the previous theorem, by replacing the above-mentioned ball $B_\epsilon$ with the standard ball $B^3_\std$.
\end{proof}

We shall now prove \cref{mainthmSthree,mainthmS2timesS1,mainthmDthree}. In the following lemma, let $(M,u,F)$ be any of the following:
\[
\begin{split}
	&(S^3,u_\std, \Fix(u_\std)),\\
	&(D^3_\std, u_\std,\Fix(u_\std|_{D^3_\std})),\\
	&(S^2 \times S^1, u_\times, \{n\} \times S^1).
\end{split}
\]
\begin{lemma}\label{antiinvarianttangenttoFS3}
	Every anti-invariant tight contact structure on $M$ is tangent to $F$.
\end{lemma}
\begin{proof}
	The statement is analogous to that of \cref{antiinvarianttangenttoF}. For any point $q \in F$, the restriction $d_qu_{\std}|_{\xi_q}$ is an orientation-reversing involution. Thus, $d_qu_{\std}|_{\xi_q}$ is conjugate to the matrix $[-1]\oplus [1]$. In particular, the 1-eigenspace of $d_qu_{\std}|_{\xi_q}$ coincides with the unique 1-eigenspace of $d_qu_{\std}$ on $T_q S^3$, which is precisely the line $T_q F$.
\end{proof}
Fix an orientation of $F$ given by a unit tangent vector field $e_1$ along $F$, and use the metric $\gk$ to define a trivialization
\[
	\tau \colon F \times S^1 \to UN(F) \subseteq TS^3|_F
\]
of the unit normal bundle. By \cref{antiinvarianttangenttoFS3}, any anti-invariant tight contact structure $\xi$ on $M$ defines a unique map
\[
	X_\xi \colon F \to S^1
\]
such that $\tau(p,X_\xi(p)) \in \xi(p)$ and $(e_1(p), X_\xi(p))$ defines a positive basis for $\xi(p)$. Set $\tau_p(\cdot) \coloneqq \tau(p,\cdot)$.
\begin{lemma}\label{Xclassindependetofxi}
	For any anti-invariant tight contact structure $\xi$ on $M$, the homotopy class $[X_\xi] \in [F,S^1]$ is independent of the equivariant isotopy class of $\xi$, relative to the boundary if $M = D^3_\std$.
\end{lemma}
\begin{proof}
	Let $h_t$ be an equivariant (orientation-preserving) isotopy of $M$, which is also relative to the boundary if $M = D^3_\std$. Since $h_t$ is equivariant, it fixes $F$ setwise (and preserves its orientation). Hence, if $\pi^\bot \colon TS^3|_F \to (TF)^\bot$ denotes the $\gk$-orthogonal projection onto the normal fiber, $X_{dh_t(\xi)}(p)$ coincides with
	\[
		\tau^{-1}_p\left[\frac{(\pi^\bot_p \circ dh_t \circ \tau_q)(X_\xi(q))}{\abs{(\pi^\bot_p\circ dh_t \circ \tau_q)(X_\xi(q))}}\right].
	\]
	where $q = h_t^{-1}(p)$.
	
	Thus, the composition of
	\[
		F \times [0,1] \to F\times [0,1], \quad (p,t) \mapsto (h_t^{-1}(p),t)
	\]
	with 
	\[
		F\times [0,1] \to S^1,\quad (q,t)\mapsto \tau^{-1}_{h_t(q)}\left[\frac{(\pi^\bot_{h_t(q)}\circ dh_t \circ \tau_q)(X_\xi(q))}{\abs{(\pi^\bot_{h_t(q)}\circ dh_t \circ \tau_q)(X_\xi(q))}}\right]
	\]
	defines a continuous homotopy between $X_\xi$ and $X_{dh_1(\xi)}$. 
\end{proof}

Let us restict to case of $M$ closed. Fix an orientation-preserving diffeomorphism $\psi \colon F \to S^1$. Identifying $S^1$ with $\R/2\pi\Z$, we see that for every pair of classes $[f], [g] \in [F,S^1]$, the $\Z$-degree of the map
\[
	S^1 \to S^1,\quad \theta \mapsto (f\circ \psi^{-1})(\theta) \pm (g\circ \psi^{-1})(\theta) 
\]
is well-defined and independent of $\psi$. We will denote this integer as $\deg_\Z(f \pm g)$.

Let $\Xi$ be the set of anti-invariant tight contact structures on $M$ up to equivariant isotopy.
\begin{defin}\label{rotaionaldefinitionS3}
	Let $\xi_1$, $\xi_0$ be two anti-invariant tight  contact structures on $M$. The rotational difference $d_M(\xi_1,\xi_0)$ is 
	\[
		d_M(\xi_1, \xi_0) \coloneqq \deg_\Z(X_{\xi_1} - X_{\xi_0}).
	\]
\end{defin}
In view of \cref{Xclassindependetofxi}, the map
\[
	\varphi \colon \Xi \times \Xi \to \Z, \quad ([\xi_1], [\xi_0]) \mapsto d_M(\xi_1, \xi_0).
\]
is well-defined.
\begin{proof}[Proof of \cref{mainthmSthree}]
	Let us consider $S^3$ as a subset of $\R^4$. By \cite[Prop.\ 2.1.8]{geigesIntroductionContactTopology2008}, the following contact structure on $S^3$ 
	\[
		\xi \coloneqq \Ker(x_2\,dx_3-x_3\,dx_2 + x_1\,dx_4 -x_4\,dx_1)
	\]
	is tight and satisfies $dp(\xi) = \xi_\std$, where $p$ is the steographic projection
	\[
		p \colon S^3 \setminus \{(0,0,0,1)\} \to \R^3.
	\]
	Moreover, it can be directly checked that $p$ is equivariant with respect to the involution $u_\std$ of $\R^3$ and the involution
	\[
		(x_1,x_2,x_3,x_4) \mapsto (-x_1, -x_2, x_3, x_4)
	\]
	of $S^3$. Thus, $\xi$ is anti-invariant and $\Xi$ is non-empty. \vspace{0.3cm}
	
	We claim that $\varphi \colon \Xi \times \Xi \to \Z$ defines a $\Z$-torsor, i.e.\ that
	\begin{enumerate}[label={C\arabic*)}]
		\item $\varphi([\xi_1], [\xi_0]) = 0$ if and only if $[\xi_1] = [\xi_0]$,
		\item for any class $[\xi_0] \in \Xi$ and every $n \in \Z$, there exists a class $[\xi_1]$ such that $\varphi([\xi_1], [\xi_0]) = n$,
		\item $\varphi([\xi_2],[\xi_1]) + \varphi([\xi_1], [\xi_0]) = \varphi([\xi_2], [\xi_0])$ for all $[\xi_0], [\xi_1], [\xi_2] \in \Xi$.
	\end{enumerate}\vspace{0.3cm}
	
	Choose $[\xi_1], [\xi_0] \in \Xi$ such that $d_{S^3}(\xi_1,\xi_0) = 0$. By definition of $d_{S^3}(\cdot,\cdot)$, there exists a smooth map 
	\[
		H \colon F \times [0,1] \to S^1
	\]
	joining $X_{\xi_0}$ and $X_{\xi_1}$. This map naturally extends to a smooth map
	\[
		F \times [0,1] \to \SO(2),
	\]
	still called $H$.  Let $\NN(F)$ be an invariant tubular neighborhood of $F$, and choose a diffeomorphism $\phi \colon \NN(F) \to F \times \R^2$ which is equivariant with respect to $u$ acting on $\NN(F)$ and
	\begin{equation}\label{eq:actiononproduct}
		\tilde{u} \colon (p,v) \mapsto (p,-v)
	\end{equation}
	acting on the target $F \times \R^2$. By the equivariant isotopy extension theorem (see \cite[Thm.\ VI.3.1]{bredonIntroductionCompactTransformation1972}), the isotopy 
	\[
	\begin{split}
		\NN(F) \times [0,1] &\to \NN(F),\\
		(p,t) &\mapsto  \phi^{-1}(q, H(q,t)v), \text{ where $\phi(p) = (q,v)$},
	\end{split}
	\]
	can be extended to an equivariant smooth isotopy $h_t \colon S^3 \to S^3$, satisfying $X_{dh_1(\xi_0)} = X_{\xi_1}$. We can then reduce to study two structures $\xi_1$, $\xi_0$ which coincide when restricted to $F$.
	
	By applying \cref{equalneighorigin} to $p = \zero, \infty$, we can suppose, without changing the equivariant isotopy class of $\xi_1$ and $\xi_0$, that they coincide on an neighborhood of a small disk $D_\epsilon(\zero)$, and on a neighbohorood of the complementary of a large ball $B_N(\zero)$. Set $M$ to be the closure of $B_N(\zero) \setminus B_\epsilon(\zero)$. The map 
	\[
		\psi \colon M \to S^2 \times [-1,1], \quad v \mapsto (v/\abs{v}, (2\abs{v} - \epsilon - N)/(N-\epsilon)),
	\]
	is an equivariant diffeomorphism onto $S^2 \times [-1,1]$, the latter equipped with the involution $u_\times = u_\std \times \id$. 
	
	Set
	\[
		F_0 \coloneqq \{\psi(0,0,z) \mid z \in [\epsilon,N] \}.
	\]
	Since \cref{rotaionaldefinitionS3} reflects \cref{rotdif}, and since $\xi_1$ and $\xi_0$ coincide when restricted to $F$, the rotational difference $d_{F_0}(d\psi(\xi_1|_M),d\psi(\xi_0|_M))$ vanishes. Hence, by \cref{mainthmtomographyStwo}, there exists an equivariant isotopy $h_t$ of $\psi(M)$ relative to a neighborhood of $\partial M$ sending $d\psi(\xi_0|_M)$ onto $d\psi(\xi_1|_M)$. Then 
	\[
		g \colon S^3 \times [0,1] \to S^3, \quad (p,t) \mapsto
		\begin{cases}
			(\psi^{-1} \circ h_t \circ \psi)(p) & \text{if $p \in M$},\\
			p & \text{if $p \notin M$},
		\end{cases}
	\]
	is an equivariant isotopy of $S^3$ satisfying $dg_1(\xi_0) = \xi_1$, proving C1).\vspace{0.3cm}
	
	Fix $[\xi_0] \in \Xi$ and $n \in \Z$. As before, consider an equivariant trivialization 
	\[
		\phi \colon \NN(F) \to F \times \R^2.
	\]
	of an invariant tubular neighborhood. Moreover, fix an identification of $F$ with $S^1$ (which we will omit for ease of notation). Set $r_n$ to be the map
	\[
		r_n \colon F \times \R^2 \to F \times \R^2, \quad (e^{i\theta},re^{i\varphi}) \mapsto (e^{i\theta},re^{i\varphi + i\rho(r)n\theta}),
	\]
	for some smooth function $\rho \colon \R_{\ge0} \to \R_{\ge 0}$ satisfying
	\begin{equation}\label{eq:rhor}
	\begin{cases}
		\rho(r) = 1& \text{for $r$ close to $0$},\\
		\rho(r)  = 0& \text{for $r$ close to $1$}.
	\end{cases}
	\end{equation}
	The map $r_n$ is equivariant with respect to the involution \eqref{eq:actiononproduct}, since
	\[
	\begin{split}
		r_n(e^{i\theta},-re^{i\varphi}) &= r_n((e^{i\theta},re^{i\varphi + i\pi})) \\
		&= (e^{i\theta},re^{i(\varphi + \pi) + i\rho(r)n\theta}) \\
		&= (e^{i\theta},-re^{i\varphi + i\rho(r)n\theta})\\
		&= (\tilde{u} \circ r_n)(e^{i\theta},-re^{i\varphi}).
	\end{split}
	\] 
	Moreover, since $r_n$ has compact support, the composition 
	\[
		\phi^{-1} \circ r_n \circ \phi \colon \NN(F) \to \NN(F)
	\]
	can be extended to a smooth equivariant diffeomorphism $R_n \colon S^3 \to S^3$ being the identity outside $\NN(F)$. By construction, the action of the differential $dR_n(re^{i\theta})$ on the normal bundle $NF(re^{i\theta})$ is a rotation of degree ranging from $0$, for $\theta = 0$, to $n2\pi$, for $\theta = 2\pi$. Thus, $d_{S^3}(dR_n(\xi_0), \xi_0) = n$, proving C2)\vspace{0.3cm}.
	
	Finally, choose $[\xi_2], [\xi_1], [\xi_0] \in \Xi$. Then 
	\[
	\begin{split}
		d_{S^3}(\xi_2,\xi_1) + d_{S^3}(\xi_1,\xi_0) &= \deg_\Z(X_{\xi_2} - X_{\xi_1}) + \deg_\Z(X_{\xi_1} - X_{\xi_0}) \\
		&= \deg_\Z(X_{\xi_2} - X_{\xi_1} + X_{\xi_1} - X_{\xi_0}) \\
		&= \deg_\Z(X_{\xi_2}-X_{\xi_0}) = d_{S^3}(\xi_2,\xi_0),
	\end{split}
	\]
	proving C3).
\end{proof}
\begin{proof}[Proof of \cref{mainthmS2timesS1}]
	Let us view $S^2 \times S^1$ as a subset of $\R^3 \times S^1$. Since 
	\[
		x\,d\theta + y\,dz - z\,dy 
	\]
	is an anti-invariant tight contact form, see \cite[Cor.\ 6.5.10]{geigesIntroductionContactTopology2008}, the set $\Xi$ is non-empty. As in \cref{mainthmSthree}, we claim that the map $\varphi$ defines a $\Z$-torsor. We retain conditions C1) - C3) defined above.
	
	Choose $[\xi_1], [\xi_0] \in \Xi$ such that $d_{S^2\times S^1}(\xi_1,\xi_0) = 0$. As in the proof of \cref{mainthmSthree}, we can suppose that $\xi_1$ and $\xi_0$ coincide when restricted to $F = \{n\} \times S^1$. Once such condition is achieved, by \cref{contastrnearspherestandard} we can suppose that $\xi_1$, $\xi_0$ coincide in an arbitrary small neighborhood $\NN = S^2 \times (e^{-i\epsilon}, e^{i\epsilon})$ of $S^2 \times \{(1,0)\}$. Finally, the natural identification 
	\[
		f \colon S^1 \setminus (e^{-i\epsilon/2}, e^{i\epsilon/2}) \to [-1,1]
	\]
	defines an equivariant diffeomorphism
	\[
		\psi \colon(S^2 \times S^1) \setminus \NN \to S^2 \times [-1,1], \quad (s,\theta) \to (s,f(\theta)).
	\]
	Set 
	\[
		F_0 \coloneqq \psi( F \cap M).
	\]
	As in \cref{mainthmSthree}, the rotational difference 
	\[
		d_{F_0}(d\psi(\xi_0|_M), d\psi(\xi_1|_M)) = d_{S^2 \times S^1}(\xi_0,\xi_1)
	\]
	vanishes, implying that $d\psi(\xi_1)$ and $d\psi(\xi_0)$ are equivariantly isotopic relative to a neighborhood of $S^2 \times [-1,1]$. In particular, $\xi_1$ and $\xi_0$ are equivariantly isotopic as strutures on $S^2 \times S^1$.
	
	For proving properties C2) and C3), one can proceed as in the case of $S^3$. Indeed, the computations performed for proving C2) are local around $F$, while the proof of C3) uses only the properties of the $\Z$-degree.
\end{proof}\vspace{0.3cm}

We shall now prove \cref{mainthmDthree}. Let $X$ be a vector field on $S^2 = \partial D^3_{\std}$, and set 
\[
	T_X \coloneqq \{\xi \mid \text{$\xi$ is a tight anti-invariant cont.\ struct.\ on $D^3_\std$ with $S^2_\xi = \ang{X}$}\}.
\]
Define
\[
	\Xi_X \coloneqq T_X/\sim,
\] 
where $\sim$ is the equivalence relation induced by the equivariant isotopies of $D^3_\std$ relative to the boundary $\partial D^3_{\std}$.

Let us canonically embed $S^1$ as the circle $\{(x,y) \in \R^2 \mid x^2 + y^2 = 1\}$ in the $xy$-plane. By \cref{antiinvarianttangenttoFS3}, any anti-invariant tight contact structure $\xi \in T_X$ is tangent to $F = \Fix(u_\std)$. Thus, if we fix an orientation on $F$ given by a unit tangent vector field $e_1$, there is a unique smooth map 
\[
	X_\xi \colon [-1,1] \to S^1
\]
such that $(X_\xi(z),0)$ belongs to $\xi(0,0,z)$ and $(e_1(0,0,z), X_\xi(z))$ is a positively oriented basis of $\xi(0,0,z)$. 

Since two structures $\xi_1, \xi_0 \in T_X$ induce the same characteristic foliation on $S^2$, the maps $X_{\xi_1}$, $X_{\xi_0}$ share the same endpoints.

\begin{defin}
	Let $\xi_1$, $\xi_0$ be two anti-invariant tight contact structures on $D^3_\std$ belonging to the set $T_X$. The rotational difference $d_{D^3_\std}(\xi_1,\xi_0)$ is 
	\[
	d_{D^3_\std}(\xi_1, \xi_0) \coloneqq \deg_\Z(X_{\xi_1} * X_{\xi_0}^{-1})
	\]
\end{defin}  
By \cref{Xclassindependetofxi}, since equivariant isotopies relative to $\partial D^3_\std$ fix the endpoints of $F$, the rotational difference depends only on the classes $[\xi_1], [\xi_0] \in \Xi_X$. Thus, the map
\[
\varphi \colon \Xi_X \times \Xi_X \to \Z, \quad ([\xi_1], [\xi_0]) \mapsto d_{D^3_\std}(\xi_1, \xi_0).
\]
is well-defined.
\begin{proof}[Proof of \cref{mainthmDthree}]
	Let us suppose that $\Xi_X$ is non-empty. As in the proof of \cref{mainthmSthree}, we claim that the map $\varphi$ satisfies the following:
	\begin{enumerate}[label={C\arabic*)}]
		\item $\varphi([\xi_1], [\xi_0]) = 0$ if and only if $[\xi_1] = [\xi_0]$,
		\item for any class $[\xi_0] \in \Xi_X$ and every $n \in \Z$, there exists a class $[\xi_1]$ such that $\varphi([\xi_1], [\xi_0]) = n$,
		\item $\varphi([\xi_2],[\xi_1]) + \varphi([\xi_1], [\xi_0]) = \varphi([\xi_2], [\xi_0])$ for all $[\xi_0], [\xi_1], [\xi_2] \in \Xi_X$.
	\end{enumerate}\vspace{0.3cm}
	
	Choose $[\xi_1], [\xi_0] \in \Xi_X$ such that $d_{D^3_\std}(\xi_1,\xi_0) = 0$. By \cref{Girouxbound}, we can suppose that $\xi_1$, $\xi_0$ coincide outside a ball $B_N(\zero) \subseteq D^3_\std$ without changing their equivariant isotopy class relative to the boundary. Moreover, via the same arguments as in the proof of \cref{mainthmSthree}, we can suppose that $\xi_1$, $\xi_0$ are chosen so that they coincide along
	\[
		F = \{(0,0,z) \mid z \in [-1,1]\}.
	\]
	Finally, by \cref{equalneighorigin}, we can suppose that $\xi_1$ and $\xi_0$ coincide on a neighborhood of a small Darboux disk $D_\epsilon(\zero)$. 
	
	As in the proof of \cref{mainthmSthree}, the closure $M$ of $D^3_\std \setminus B_\epsilon(\zero)$ is equivariantly diffeomorphic to $S^2 \times [-1,1]$ via the map 
	\[
		\psi \colon M \to S^2 \times [-1,1], \quad v \mapsto (v/\abs{v}, (2\abs{v} -\epsilon - 1)/(1-\epsilon)).
	\]
	Set 
	\[
		F_0 \coloneqq \psi(F \cap M),
	\]
	and observe, as before, that
	\[
		d_{F_0}(d\psi(\xi_1|_M), d\psi(\xi_0|_M)) = d_{D^3_\std}(\xi_1, \xi_0) = 0,
	\]
	implying that the structures $d\psi(\xi_1|_M)$ and $d\psi(\xi_0|_M)$ are equivariantly isotopic relative to a neighborhood of $S^2 \times [-1,1]$. In particular, the structures $\xi_1$ and $\xi_0$ are equivariantly isotopic in $D^3_\std$ relative to the boundary.\vspace{0.3cm}
	
	For proving condition C2) we apply the same arguments as above, by replacing the map $r_n$ with 
	\[
	\begin{split}
		F \times \R^2 &\to F \times \R^2, \\
		\left(p, re^{i\varphi}\right) &\mapsto (p,r\exp(i\varphi + i2n\pi\rho(r)\eta(z))), \quad (p = (0,0,z)),
	\end{split}
	\]
	for some smooth function $\rho \colon \R_{\ge0} \to \R_{\ge 0}$ satisfying conditions \eqref{eq:rhor} and some smooth function $\eta \colon [-N,N] \to [0,1]$ such that
	\[
	\eta(z) = 0 \text{ for $z$ near $-N$, and } \eta(z) = 1 \text{ for $z$ near $N$}.
	\]
	
	Finally, property C3) follows by the property of the $\Z$-degree, analogously as in \cref{mainthmSthree}.
\end{proof}
Altought there is not a direct criterion telling when a characteristic foliation $X \in \XX(\partial D^3_\std)$ extends to an anti-invariant tight contact structure on $D^3_\std$, the following result gives a direct condition for when $X$ is of special type.
\begin{prop}\label{XiXnonempty}
	Suppose that the following occur:
	\begin{enumerate}
		\item $X$ satisfies the divergence condition \eqref{eq:divcond};
		\item $X$ is of special type;
		\item $X$ admits a dividing set made of a single curve. 
	\end{enumerate}
	Then $\Xi_X$ is non-empty.
\end{prop}
\begin{proof}
	Let $\Omega$ be an invariant volume form on $\partial D^3_\std$. Since $X$ satisfies condition \eqref{eq:divcond}, by \cite[Lemma 4.6.3]{geigesIntroductionContactTopology2008}, there exists an explicit 1-form $\alpha$, defined on a small open neighborhood $\NN$ of $\partial D^3_\std$, such that 
	\[
	i^*_{\partial D^3_\std} \alpha = \iota_X\Omega.
	\] 
	One can directly verify that since $X$ is anti-invariant, Geiges' construction provides an anti-invariant form $\alpha$. Set $\xi \coloneqq \Ker(\alpha)$. By construction, $X$ generates the characteristic foliation $S_\xi$, which is divided by a single circle $\Gamma$.
	
	Since $S_\xi$ is dividied by a single circle, the proof of \cref{contastrnearspherestandard} follows straightforward: there exists an equivariant isotopy
	\[
	\psi_t \colon \R^3 \to \R^3,
	\]
	supported inside $\NN$, such that $d\psi_1(\xi_\std)$ coincides with $\xi$ close to $\partial D^3_\std$. The restriction of $d\psi_1(\xi_{\std})$ to $D^3_\std$ is a tight anti-invariant contact structure contained in $T_X$. In particular, $\Xi_X$ is non-empty.
\end{proof}
\section*{Acknowledgements}
The author would like to thank \emph{Università degli Studi di Firenze}, and \emph{Diparitmento di Matematica e Informatica ``U.\ Dini''}, for their academic support and resources, as this work is part of the author's ongoing PhD thesis research.

The author is a member of \emph{Gruppo Nazionale per le Strutture Algebriche, Geometriche e le loro Applicazioni}, which in turn is part of \emph{Istituto Nazionale di Alta Matematica} (GNSAGA - INdAM).
\section*{AI Declaration}
During the preparation of this manuscript, Google Gemini (Goole AI) was used to assist with language editinga and literature searches. The author reviewed and revised the AI-assisted output and takes full responsibility for the mathematical content and final manuscript.
\printbibliography

@book{geigesIntroductionContactTopology2008,
	title = {An Introduction to Contact Topology},
	author = {Geiges, Hansjörg},
	date = {2008},
	publisher = {Cambridge University Press},
}

@book{bredonIntroductionCompactTransformation1972,
	title = {Introduction to Compact Transformation Groups},
	author = {Bredon, Glen E.},
	date = {1972},
	series = {Pure and {{Applied Mathematics}}},
	number = {46}
}

@article{sotomayorgenericoneparameter,
	author = {Sotomayor, Jorge},
	title = {Generic one-parameter families of vector fields on two-dimensional manifolds},
	year = {1974},
	volume = {43},
	pages = {5--46},
	journal = {Publications mathématiques de l’I.H.É.S.},
}

@article{smaleDiffeomorphisms2Sphere1959,
	author = {Smale, Stephen},
	title = {Diffeomorphisms of the {$2$}-sphere},
	journal = {Proceedings of the American Mathematical Society},
	volume = {10},
	year = {1959},
	pages = {621--626},
}

@article{PalaisLocaltrivialityoftherestrictionmap,
	author = {Palais, Richard S.},
	title = {Local triviality of the restriction map for embeddings},
	journal = {Commentarii Mathematici Helvetici},
	volume = {34},
	year = {1960},
	pages = {305--312},
}

@book{leeGroupsJOG1975,
	title = {On the Groups {{JO}}({{G}})},
	author = {Lee, Chung-Nim and Wasserman, Arthur G.},
	date = {1975},
	series = {Memoirs of the {{American Mathematical Society}}},
	number = {159},
	publisher = {American Mathematical Society}
}

@article{duggerInvolutionsSurfaces2019,
	title = {Involutions on Surfaces},
	author = {Dugger, Daniel},
	year = {2019},
	journaltitle = {Journal of Homotopy and Related Structures},
	volume = {14},
	number = {4},
	pages = {919--992}
}

@book{stillwellClassicalTopologyCombinatorial1993,
	title = {Classical {{Topology}} and {{Combinatorial Group Theory}}},
	author = {Stillwell, John},
	year = {1993},
	series = {Graduate {{Texts}} in {{Mathematics}}},
	number = {72},
	publisher = {Springer},
}

@book{hatcherAlgebraicTopology2002,
	title = {Algebraic {{Topology}}},
	author = {Hatcher, Allen},
	date = {2002},
	publisher = {Cambridge University Press}
}

@article{girouxStructuresContactDimension2000,
	title = {Structures de Contact En Dimension Trois et Bifurcations Des Feuilletages de Surfaces},
	author = {Giroux, Emmanuel},
	date = {2000},
	journaltitle = {Inventiones Mathematicae},
	volume = {141},
	pages = {615--689}
}

@article{ernstManifoldEmbeddingsClosed1981,
	title = {The Manifold of Embeddings of a Closed Manifold},
	author = {Ernst, Binz and Fischer, Hans R.},
	date = {1981},
	journal = {Proceedings of the Differential  Geometric Methods  in Theoretical  Physics},
	series = {Springer {{Lecture Notes}} in {{Physics}}},
	number = {139},
	pages ={310-325}
}

@book{benedettiLecturesDifferentialTopology2021,
	title = {Lectures on Differential Topology},
	author = {Benedetti, Riccardo},
	date = {2021},
	series = {Graduate {{Studies}} in {{Mathematics}}},
	number = {218},
	publisher = {American Mathematical Society}
}

@phdthesis{dominicHomotopyTypeSpace2018,
	title = {The Homotopy Type of the Space of Tight Contact Structures on the 3-Sphere},
	author = {Dominic, J{\"a}nichen},
	year = 2018,
	school = {Justus--Liebig--Universit\"at Gie\ss en}
}

@article{girouxConvexiteTopologieContact1991,
	title = {{Convexit\'e en topologie de contact.}},
	author = {Giroux, E.},
	year = {1991},
	journal = {Commentarii mathematici Helvetici},
	volume = {66},
	number = {4},
	pages = {637--677}
}

@article{gramainTypeDhomotopieGroupe1973,
	title = {{Le type d'homotopie du groupe des diff\'eomorphismes d'une surface compacte}},
	author = {Gramain, Andr{\'e}},
	year = 1973,
	journal = {Annales scientifiques de l'\'Ecole Normale Sup\'erieure},
	volume = {4e s\`erie, 6},
	number = {1},
	pages = {53--66},
	publisher = {Elsevier}
}

@book{farbPrimerMappingClass2011,
	title = {A {{Primer}} on {{Mapping Class Groups}}},
	author = {Farb, Benson and Margalit, Dan},
	year = {2011},
	publisher = {Princeton University Press}
}

@article{igusaStabilityTheoremSmooth1988,
	title = {The Stability Theorem for Smooth Pseudoisotopies},
	author = {Igusa, Kiyoshi},
	year = {1988},
	journal = {$K$-Theory. An Interdisciplinary Journal for the Development, Application, and Influence of $K$-Theory in the Mathematical Sciences},
	volume = {2},
	number = {1-2},
	pages = {1--355}
}
\end{document}